\documentclass[12pt]{article}        

\usepackage[numbers]{natbib}
\usepackage[stable]{footmisc}

\usepackage{amsfonts, amssymb, amsmath, amsthm}
\usepackage{mathrsfs}
\usepackage{float}
\usepackage{fancyhdr}

\usepackage{tikz}     

\usetikzlibrary{arrows}

\usepackage{MnSymbol}
\usepackage{bbold}

\usepackage{color,hyperref}
    \catcode`\_=11\relax
    \newcommand\email[1]{\_email #1\q_nil}
    \def\_email#1@#2\q_nil{%
      \href{mailto:#1@#2}{{\emailfont #1\emailampersat #2}}
    }
    \newcommand\emailfont{\sffamily}
    \newcommand\emailampersat{{\color{red}\small@}}
    \catcode`\_=8\relax

\theoremstyle{plain}
\newtheorem{thm}{Theorem}
\newtheorem{lemma}[thm]{Lemma}
\newtheorem{cor}[thm]{Corollary}
\newtheorem{remark}[thm]{Remark}

\theoremstyle{definition}
\newtheorem{definition}[thm]{Definition}
\newtheorem{example}[thm]{Example}

\makeatletter
\newcommand{\superimpose}[2]{%
  {\ooalign{$#1\@firstoftwo#2$\cr\hfil$#1\@secondoftwo#2$\hfil\cr}}}
\makeatother

\newcommand{\ltensor}{\mathpalette\superimpose{{\bigcirc}{\ltimes}}}             
\newcommand{\rtensor}{\mathpalette\superimpose{{ \bigcirc}{\rtimes}}}

\newcommand{\be}{\begin{equation}} 
\newcommand{\ee}{\end{equation}}

\newcommand{\Rp}{\mathbb{R}^{^p}}
\newcommand{\Rn}{\mathbb{R}^{^n}}
\newcommand{\Rv}[1]{\mathbb{R}^{^#1}}

\newcommand{\ch}{\mathbb{1}}
\newcommand{\A}{\mathcal{A}}
\newcommand{\C}{\mathcal{C}}
\newcommand{\F}{\mathcal{F}}
\newcommand{\Set}{\mathcal{S}et}
\newcommand{\X}{\textbf{X}}
\newcommand{\Y}{\textbf{Y}}

\newcommand{\G}{\mathcal{G}}
\newcommand{\GP}{\mathcal{G}\mathcal{P}}

\newcommand{\prob}{\mathcal{P}}
\newcommand{\sa}{\Sigma}
\newcommand{\mO}[1]{(#1,\Sigma_{#1})}
\newcommand{\M}{\mathcal{M}eas}

\newcommand{\px}{\mathscr{P}}
\newcommand{\xv}{\mathbf{x}}
\newcommand{\uv}{\mathbf{u}}                
\newcommand{\vv}{\mathbf{v}}
\newcommand{\mv}{\mathbf{m}}
\newcommand{\kv}{\mathbf{k}}
\newcommand{\wv}{\mathbf{w}}
\newcommand{\yv}{\mathbf{y}}
\newcommand{\zv}{\mathbf{z}}
\newcommand{\av}{\mathbf{a}}
\newcommand{\e}{\mathbf{e}}
\newcommand{\U}{\mathcal{U}}
\newcommand{\E}{\mathbb{E}}                
\newcommand{\NN}{\mathcal{N}}            
\newcommand{\fn}{\ulcorner f \urcorner}    

\newcommand{\p}[1]{P_{#1}}

\newcommand{\mcS}{\mathcal{S}}
\newcommand{\mcI}{\mathcal{I}}

\newcommand{\mcSS}{\mathcal{S}^{\mathbf{x}}}      
\newcommand{\mcII}{\mathcal{I}^{\mathbf{x}}}           

\def\@normalsize{\@setsize\normalsize{14.5pt}\xiipt\@xiipt
\abovedisplayskip 12\p@ plus3\p@ minus7\p@
\belowdisplayskip \abovedisplayskip
\abovedisplayshortskip  \z@ plus3\p@
\belowdisplayshortskip  6.5\p@ plus3.5\p@ minus3\p@
\let\@listi\@listI}
\def\keywords{\vspace{.3em}
{\textit{Keywords}:\,\relax%
}}

\title{Bayesian Machine Learning via Category Theory}

\author{Jared Culbertson and Kirk Sturtz}

\fancypagestyle{plain}{
\fancyhf{} 
\fancyfoot[C]{\thepage} 
  
}

\begin{document}

\maketitle
 
\thispagestyle{empty}

\begin{abstract}   From the Bayesian perspective, the category of conditional probabilities (a variant of the Kleisli category of the Giry monad, whose objects are measurable spaces and arrows are Markov kernels) gives a nice framework for conceptualization and analysis of many aspects of machine learning. Using categorical methods, we construct models for parametric and nonparametric Bayesian reasoning on function spaces, thus providing a basis for the supervised learning problem. In particular, stochastic processes are arrows to these function spaces which serve as prior probabilities. The resulting inference maps can often be analytically constructed in this symmetric monoidal weakly closed category. We also show how to view general stochastic processes using functor categories and demonstrate the Kalman filter as an archetype for the hidden Markov model.

 \begin{flushleft}
 \\

\end{flushleft}
 \end{abstract}
  
\tableofcontents 
 
\section{Introduction}   

 Speculation on the utility of using categorical methods in machine learning (ML) has been expounded by numerous people, including by the denizens at the n-category cafe blog~\cite{ncat} as early as 2007.  Our approach to realizing categorical ML is based upon viewing ML from a probabilistic perspective and using categorical Bayesian probability.   Several recent texts ({\em e.g.}, \cite{Barber, Murphy}), along with countless research papers on ML have emphasized the subject from the perspective of Bayesian reasoning.  Combining this viewpoint with the recent work  \cite{Culbertson}, which provides  a categorical framework for Bayesian probability, we develop a category theoretic perspective on ML.   The abstraction provided by category theory serves as a basis not only for an organization of ones thoughts on the subject, but also provides an efficient graphical method for model building in much the same way that probabilistic graphical modeling (PGM) has provided for Bayesian network problems.
 
In this paper, we focus entirely on the supervised learning problem, {\em i.e.,} the regression or function estimation problem. The general framework applies to any Bayesian machine learning problem, however. For instance, the unsupervised clustering or density estimation problems can be characterized in a similar way by changing the hypothesis space and sampling distribution. For simplicity, we choose to focus on regression and leave the other problems to the industrious reader. For us, then, the Bayesian learning problem is to determine  a function $f: X \rightarrow Y$ which takes an input $\xv \in X$, such as a feature vector, and associates an output (or class) $f(\xv)$ with $\xv$.  Given a measurement $(\xv,y)$, or a set of measurements  
$\{ (\xv_i,y_i)\}_{i=1}^N$ where each $y_i$ is a labeled output ({\em i.e.,} training data),   we interpret this problem as an estimation problem of an unknown function $f$ which lies in $Y^X$, the space of all measurable functions\footnote{Recall that a $\sigma$-algebra $\sa_{X}$ on $X$ is a collection of subsets of $X$ that is closed under complements and countable unions (and hence intersections); the pair $(X, \sa_X)$ is called a {\em measurable space} and any set $A \in \sa_{X}$ is called a {\em measurable set} of $X$.  A {\em measurable} function $f \colon X \to Y$ is defined by the property that for any measurable set $B$ in the $\sigma$-algebra of $Y$, we have that $f^{-1}(B)$ is in the $\sigma$-algebra of $X$. For example, all continuous functions are measurable with respect to the Borel $\sigma$-algebras.} from $X$ to $Y$ such that $f(\xv_i) \approx y_i$.  When $Y$ is a vector space the space $Y^X$ is also a vector space that is infinite dimensional when $X$ is infinite.    If we choose to allow all such functions (every  function $f \in Y^X$ is a valid model) then the problem is nonparametric. On the other hand, if we only allow functions from  some subspace $V \subset Y^X$ of \emph{finite} dimension $p$, then we have a parametric model characterized by a measurable map $i: \Rp \rightarrow Y^X$.  The  image of $i$ is then the space of functions which we consider as valid models of the unknown function for the Bayesian estimation problem. Hence, the elements $\av \in \Rp$ completely determine the valid modeling functions $i(\av) \in Y^X$.  Bayesian modeling splits the problem into two aspects: (1) specification of the hypothesis space, which consist of the ``valid''  functions $f$, and (2) a noisy measurement model such as $y_i = f(\xv_i) + \epsilon_i$, where the noise component $\epsilon_i$ is often modeled by a Gaussian distribution.  Bayesian reasoning with the hypothesis space taken as $Y^X$ or any subspace $V \subset Y^X$ (finite or infinite dimensional) and the noisy measurement model determining a sampling distribution  can then be used to  efficiently estimate (learn) the function $f$ without over fitting the data. 

We cast this whole process into a graphical formulation using category theory, which like PGM, can in turn be used as a modeling tool itself.   In fact, we view the components of these various models, which are just Markov kernels, as interchangeable parts. An important piece of the any solving the ML problem with a Bayesian model consists of choosing the appropriate parts for a given setting.  The close relationship between parametric and nonparametric models comes to the forefront in the analysis with the measurable map $i: \Rp \rightarrow Y^X$ connecting the two different types of models.  To illustrate this point suppose we are given a normal distribution $P$ on $\Rp$ as a prior probability on the unknown parameters. Then the push forward measure\footnote{A measure $\mu$ on a measurable space $(X, \sa_{X})$ is a nonnegative real-valued function $\mu \colon X \to {\mathbb R}_{\geq 0}$ such that $\mu(\emptyset) = 0$ and $\mu(\cup_{i=1}^{\infty} A_{i}) = \sum_{i=1}^{\infty} \mu(A_{i})$. A probability measure is a measure where $\mu(X) = 1$. In this paper, all measures are probability measures and the terminology ``distribution'' will be synonymous with ``probability measure.''} of $P$ by $i$ is a Gaussian process, which is a basic tool in nonparametric modeling. When composed with a noisy measurement model, this provides the whole Bayesian model required for a complete analysis and an inference map can be analytically constructed.\footnote{The inference map need not be unique.}  Consequently, given any measurement $(\xv,y)$ taking the inference map conditioned at $(\xv,y)$ yields the updated prior probability which is another normal distribution on $\Rp$.

The ability to do Bayesian probability involving function spaces relies on the fact that the category  of measurable spaces, $\M$, has the structure of a symmetric monoidal closed category (SMCC). Through the evaluation map, this in turn provides the category of conditional probabilities $\prob$ with the structure of a symmetric monoidal \emph{weakly} closed category (SMwCC), which is necessary for modeling stochastic processes as probability measures on function spaces.  On the other hand, the ordinary product $X \times Y$ with its product $\sigma$-algebra is used for the Bayesian aspect of updating joint (and marginal) distributions.  From a modeling viewpoint, the SMwCC structure is used for carrying along a parameter space (along with its relationship to the output space through the evaluation map). Thus we can describe training data and measurements as ordered pairs $(\xv_i,y_i) \in X \otimes Y$, where $X$ plays the role of a parameter space.
 
\paragraph{A few notes on the exposition.}
In this paper our intended audience consists of (1) the practicing ML engineer with only a passing knowledge of category theory ({\em e.g.}, knowing about objects, arrows and commutative diagrams), and (2) those knowledgeable of category theory with an interest of how ML can be formulated within this context. For the ML engineer familiar with Markov kernels, we believe that the presentation of $\prob$ and its applications can serve as an easier introduction to categorical ideas and methods than many standard approaches.  While some terminology will be unfamiliar, the examples should provide an adequate understanding to relate the knowledge of ML to the categorical perspective. If ML researchers find this categorical perspective useful for further developments or simply for modeling purposes, then this paper will have achieved its goal. 

In the categorical framework for Bayesian probability, Bayes' equation is replaced by an integral equation where the integrals are defined over probability measures.  The analysis requires these integrals be evaluated on arbitrary measurable sets and this is often possible using the three basic rules provided in Appendix A.  Detailed knowledge of measure theory is not necessary outside of understanding these three rules and the basics of $\sigma$-algebras and measures, which are used extensively for evaluating integrals in this paper. Some proofs require more advanced measure-theoretic ideas, but the proofs can safely be avoided by the unfamiliar reader and are provided for the convenience of those who might be interested in such details.  

For the category theorist, we hope the paper makes the fundamental ideas of ML transparent, and conveys our belief that  Bayesian probability can be characterized categorically and usefully applied to  fields such as ML.  We believe the further development of categorical probability can be motivated by such applications and in the final remarks we comment on one such direction that we are pursuing.

These notes are intended to be tutorial in nature, and so contain much more detail that would be reasonable for a standard research paper. As in this introductory section, basic definitions will be given as footnotes, while more important definitions, lemmas and theorems Although an effort has been made to make the exposition as self-contained as possible, complete self-containment is clearly an unachievable goal.  In the presentation, we avoid the use of the terminology of \emph{random variables} for two reasons: (1) formally a random variable is a measurable function $f:X \rightarrow Y$ and a probability measure $P$ on $X$ gives rise to the distribution of the random variable $f_{\star}(P)$ which is the push forward measure of $P$.  In practice the random variable $f$ itself is more often than not impossible to characterize functionally (consider the process of flipping a coin), while reference to the random variable using a binomial distribution, or any other distribution, is simply making reference to some probability measure.  As a result, in practice the term ``random variable'' is often not making reference to any measurable function $f$  and the pushforward measure of some probability measure $P$ at all but rather is  just referring to a probability measure;   (2) the term ``random variable'' has a connotation that, we believe, should be de-emphasized in a Bayesian approach to modeling uncertainty.  Thus while a random variable can be modeled as a push forward probability measure within the framework presented we feel no need to single them out as having any special relevance beyond the remark already given.  In illustrating the application of categorical  Bayesian probability we do however show how to translate the familiar language of random variables into the unfamiliar categorical framework for the particular case of Gaussian distributions which are the most important application for ML since Gaussian Processes are characterized on finite subsets by Gaussian distributions.  This provides a particularly nice illustration of the non uniqueness of conditional sampling distribution and inference pairs given a joint distribution.

\paragraph{Organization.}
The paper is organized as follows:  The theory of Bayesian probability  in $\prob$ is first addressed and applied to elementary problems on finite spaces where the detailed solutions to inference, prediction and decision problems are provided.  If one understands the ``how and why'' in solving these problems then the extension to solving problems in ML is a simple step as one uses the same basic paradigm with only the hypothesis space changed to a  function space.   Nonparametric modeling is  presented next,  and then the parametric model can seen as a submodel  of the nonparametric model.   We then proceed to give a  general definition of stochastic process as a special type of arrow in a functor category $\prob^X$, and by varying the category $X$ or placing conditions on the projection maps onto subspaces one obtains the various types of stochastic processes such as Markov processes or GP.  Finally, we remark on the area where category theory may have the biggest impact on applications for ML by integrating the probabilistic models with decision theory into one common framework.

The results presented here derived from a categorical analysis of the ML problem(s) will come as no surprise to ML professionals.  We acknowledge and thank our colleagues who are experts in the field who provided assistance and feedback.

\section{The Category of Conditional Probabilities}

The development of a categorical basis for probability was initiated  by Lawvere \cite{Lawvere}, and further developed by Giry \cite{Giry} using monads to characterize the adjunction given in Lawvere's original work.  The Kleisli category of the Giry monad $\G$ is what Lawvere called the category of probabilistic mappings and what we shall refer to as the category of conditional probabilities.\footnote{Monads had not yet been developed at the time of Lawvere's work. However the adjunction construction he provided was the Giry monad on measurable spaces.}  Further progess was given in the unpublished dissertation of  Meng~\cite{Meng} which provides a wealth of information and provides a basis for thinking about stochastic processes from a categorical viewpoint.  While this work does not  address the  Bayesian perspective it does provide an alternative ``statistical viewpoint''  toward solving such problems using generalized metrics.  Additional interesting work on this category is presented in a seminar by Voevodsky, in Russian,  available in an online video \cite{Voev}.  The extension of categorical probability to the Bayesian viewpoint is given in the paper~\cite{Culbertson}, though Lawvere and Peter Huber were aware of a similar approach in the 1960's.\footnote{In a personal communication Lawvere related that he and Peter Huber gave a seminar in Zurich around 1965 on ``Bayesian sections.'' This refers to the existence of inference maps in the Eilenberg--Moore category of $\G$-algebras.  These inference maps are discussed in Section 3, although we discuss them only in the context of the category $\prob$.}  Coecke and Speckens~\cite{Coecke} provide an alternative graphical language for Bayesian reasoning under the assumption of finite spaces which they refer to as standard probability theory.  In such spaces the arrows can be represented by stochastic matrices~\cite{Fritz}. More recently Fong~\cite{Brendan} has provided further applications of the category of conditional probabilities to Causal Theories for Bayesian networks.  

Much of the material in this section is directly from \cite{Culbertson}, with some additional explanation where necessary. The category%
\footnote{A {\em category} is a collection of (1) objects and (2) morphisms (or arrows) between the objects (including a required identity morphism for each object), along with a prescribed method for associative composition of morphisms.} 
of conditional probabilities, which we denote by $\prob$, has countably generated%
\footnote{A space $(X,\sa_X)$ is countably generated if there exist a countable set of measurable sets $\{A_i\}_{i=1}^\infty$ which generated the $\sigma$-algebra $\sa_X$.} 
measurable spaces $(X, \sa_X)$ as  objects and an arrow between two such objects
\begin{equation*}
 \begin{tikzpicture}[baseline=(current bounding box.center)]

         \node (X) at  (0,0)  {$(X, \sa_X)$};
         \node (Y)  at  (4,0)   {$(Y, \sa_Y)$};

	\draw[->,above] (X) to node {$T$} (Y);

 \end{tikzpicture}
\end{equation*}
is a Markov kernel (also called a \emph{regular} conditional probability) assigning to each element $x \in X$ and each measurable set $B \in \sa_Y$ the probability of $B$ given $x$, denoted  $T(B \mid x)$.  The term ``regular''  refers to the fact that the function $T$ is conditioned on points rather than measurable sets $A \in \sa_X$. When $(X,\sa_X)$ is a countable set (either finite or countably infinite) with the discrete $\sigma$-algebra then every singleton $\{x\}$ is  measurable and the term ``regular'' is unnecessary.  
More precisely,  an arrow $T\colon X \rightarrow Y$ in $\prob$  is  a function $T\colon~\sa_Y  \times X~\rightarrow~[0,1]$ satisfying
\begin{enumerate}
\item for all $B\in \sa_Y$, the function $T(B \mid \cdot)\colon X \rightarrow [0,1]$ is measurable, and
\item for all $x \in X$, the function $T(\cdot \mid x)\colon \sa_Y \rightarrow [0,1]$ is a perfect probability measure\footnote{A perfect probability measure $P$ on $Y$ is a probability measure  such that for any measurable function $f:Y \rightarrow \mathbb{R}$ there exist a real Borel set $E \subset f(Y)$ satisfying $P(f^{-1}(E))=1$.} on $Y$.
\end{enumerate}
For technical reasons it is necessary that the probability measures in (2) constitute an equiperfect family of probability measures to avoid pathological cases which prevent the existence of inference maps necessary for Bayesian reasoning.\footnote{Specifically, the subsequent Theorem  \ref{regularConditional} is a constructive procedure which requires perfect probability measures.  Corollary \ref{inferenceExistence} then gives the inference map.  Without the hypothesis of perfect measures a pathological counterexample can be constructed as in \cite[Problem 10.26]{Dudley}.  The paper by Faden~\cite{Faden} gives conditions on the existence of conditional probabilities and this constraint is explained in full detail in \cite{Culbertson}. Note that the class of perfect measures is quite broad and includes all probability measures defined on Polish spaces.}

The notation $T(B\mid x)$ is chosen as it coincides  with the standard notation ``$p(H \mid D)$''  of conditional probability theory.
For an arrow $T\colon \mO{X} \rightarrow \mO{Y}$, we occasionally denote the measurable function $T(B \mid \cdot)\colon \Sigma_Y \rightarrow [0,1]$ by $T_B$ and the probability measure $T(\cdot \mid x)\colon \Sigma_Y \rightarrow [0,1]$ by $T_x$.    
Hereafter, for notational brevity we write a measurable space $(X,\sa_X)$ simply as $X$ when referring to a generic $\sigma$-algebra $\sa_X$.  

Given two arrows
\begin{equation*}
 \begin{tikzpicture}[baseline=(current bounding box.center)]

         \node (X) at  (0,0)  {$X$};
         \node (Y)  at  (3,0)   {$Y$};
         \node (Z) at  (6,0)   {$Z$};
	\draw[->,above] (X) to node {$T$} (Y);
         \draw[->,above] (Y)  to node {$U$} (Z);
 \end{tikzpicture}
 \end{equation*}
the composition $U \circ T\colon \sa_Z  \times X   \rightarrow [0,1]$ is \emph{marginalization over $Y$} defined by
\begin{equation*} 
(U \circ T)(C \mid x) = \int_{y \in Y} U(C \mid y) \, dT_x.
\end{equation*}

The integral of any real valued measurable
function $f\colon X \rightarrow \mathbb{R}$ with respect to any measure $P$ on $X$  is 
\begin{equation} \label{expectation} 
\E_{P}[f] = \int_{x \in X} f(x) \, dP,
\end{equation}
called the \emph{$P$-expectation of $f$}.
Consequently the composite $(U \circ T)(C \mid x)$ is the
$T_x$-expectation of $U_{C}$,
\begin{equation*}
 (U \circ T)(C \mid x)= \E_{T_x}[U_{C}]. 
 \end{equation*}

 Let $\M$ denote the category of measurable spaces where the objects are measurable spaces $(X,\sa_X)$ and the arrows are measurable functions $f\colon X \rightarrow Y$.   
 Every measurable mapping $f\colon X \rightarrow Y$ may be regarded as a $\prob$ arrow
\begin{equation*}
 \begin{tikzpicture}[baseline=(current bounding box.center)]

         \node (X) at  (0,0)  {$X$};
         \node (Y)  at  (3,0)   {$Y$};
	\draw[->,above] (X) to node {$\delta_f$} (Y);

	 \end{tikzpicture}
 \end{equation*}
defined by the Dirac (or one point) measure
\begin{equation*}
\begin{array}{lclcl}
\delta_f&:&X \times \sa_Y & \rightarrow &  [0,1] \\
&:& (B \mid x) & \mapsto & \left\{ \begin{array}{c} 1 \quad \textrm{ If }f(x)\in B
\\ 0 \quad \textrm{If }f(x)\notin B. \end{array} \right.
\end{array}
\end{equation*}
The relation between the dirac measure and the characteristic (indicator) function $\ch$ is 
\be \nonumber
\delta_f(B \mid x) = \ch_{f^{-1}(B)}(x)
\ee
and this property is used ubiquitously in the analysis of integrals.

Taking the measurable mapping $f$ to be the identity map on $X$ gives for each object $X$ the morphism $X \stackrel{\delta_{Id_X}}{\longrightarrow} X$ given by
 \be \nonumber
 \delta_{Id_X}(B \mid x) =  \left \{ 
 \begin{array}{lcl}
 1 &  \textrm{ if } x\in B \\
 0 & \textrm{ if } x \notin B
 \end{array} \right. 
 \ee
which is the identity morphism for $X$ in $\prob$.   Using standard notation we denote the identity mapping on any object $X$ by $1_X = \delta_{Id_X}$, or for brevity simply by $1$ if the space $X$ is clear from the context.
 With these objects and arrows, law of composition, associativity, and identity, standard measure-theoretic arguments show that $\prob$ forms  a category.

There is a distinguished object in $\prob$ that play an important role in Bayesian probability. For any set $Y$ with the indiscrete $\sigma$-algebra $\sa_Y=\{Y, \emptyset\}$, there is a unique arrow from any object $X$ to $Y$ since any arrow $P\colon X \rightarrow Y$ is completely determined by the fact that $P_{x}$ must be a probability measure on $Y$.  Hence $Y$ is a \emph{terminal} object, and we denote the unique arrow by $!_X : X \rightarrow Y$.   Up to isomorphism,  the canonical terminal object is the one-element set which we denote by $1 = \{ \star\}$ with the only possible $\sigma$-algebra. It follows that any  arrow $P:1 \rightarrow X$
from the terminal object to any space $X$ is an (absolute) probability measure on $X$, {\em i.e.}, it is an  ``absolute'' probability measure on $X$ because there is no variability (conditioning) possible within the singleton set $1=\{\star\}$.  
\begin{figure}[H]
\begin{equation}  \nonumber
 \begin{tikzpicture}[baseline=(current bounding box.center)]

         \node (1) at  (0,0)  {$1$};
         \node (X)  at  (3,0)   {$X$};
	\draw[->,above] (1) to node {$P$} (X);

 \end{tikzpicture}
 \end{equation}
 \caption{The representation of a probability measure in $\prob$.}
\end{figure}
\noindent
We refer to any arrow $P\colon 1 \rightarrow X$ with domain $1$ as either a probability measure or a distribution on $X$.  If $X$ is countable then $X$ is isomorphic in $\prob$ to a discrete space $\mathbf{m} = \{0,1,2, \ldots, m-1\}$ with the discrete $\sigma$-algebra where the integer $m$ corresponds to the number of atoms in the $\sigma$-algebra $\sa_X$.  
Consequently every finite space is, up to isomorphism, just a discrete space and therefore every distribution  $P\colon 1 \rightarrow X$ is of the form $P = \sum_{i=0}^{m-1} p_i \delta_i$ where $\sum_{i=0}^{m-1} p_i =1$.

 \subsection{(Weak) Product Spaces and Joint Distributions}  
 \label{sec:products}

  In Bayesian probability, determining the joint distribution on a ``product space'' is often the problem to be solved.  In many applications for which Bayesian reasoning in appropriate, the problem reduces to computing a particular marginal or conditional probability; these can be obtained in a straightforward way if the joint distribution is known. Before proceeding to  formulate precisely what the term ``product space'' means in $\prob$, we describe the categorical construct of a \emph{finite product space} in any category.  
  
Let $\C$ be an arbitary category and $X, Y \in_{ob} \C$.  We say the product of $X$ and $Y$ exists if there is an object, which we denote by $X \times Y$, along with two arrows $p_X\colon X \times Y \rightarrow X$ and $p_Y\colon X \times Y \rightarrow Y$  in $\C$ such that given any other object $T$ in $\C$ and arrows $f:T \rightarrow X$ and $g:T \rightarrow Y$ there is a \emph{unique} $\C$ arrow $\langle f,g \rangle\colon T \rightarrow X \times Y$ that makes the diagram
\begin{equation} \label{binary products}
 \begin{tikzpicture}[baseline=(current bounding box.center)]
 	\node	(T)	at	(0,0)		         {$T$};
	\node	(X)	at	(-3,-3)	         {$X$};
	\node	(Y)	at	(3,-3)               {$Y$};
	\node        (XY)  at      (0,-3)               {$X \times Y$};
	
	\draw[->, left] (T) to node  {$f$} (X);
	\draw[->,right] (T) to node {$g$} (Y);
	\draw[->,right, dashed] (T) to node {$\langle f,g \rangle$} (XY);
	\draw[->,below] (XY) to node {$p_{X}$} (X);
	\draw[->,below] (XY) to node {$p_{Y}$} (Y);
 \end{tikzpicture}
 \end{equation}
commute. If the given diagram is a product then we often write the product as a triple $(X \times Y, p_X, p_Y)$.  We must not let the notation deceive us; the object $X \times Y$ could just as well be represented by $P_{X,Y}$. The important point is that it is an object in $\C$ that we need to specify in order to show that binary products exist.   Products are an example of a universal construction in categories.  The term ``universal''  implies that these constructions are unique up to a unique isomorphism.  Thus if $(P_{X,Y}, p_X, p_y)$ and $(Q_{X,Y}, q_X, q_Y)$ are both products for the objects $X$ and $Y$ then there exist unique arrows $\alpha \colon P_{X,Y} \rightarrow Q_{X,Y}$ and $\beta \colon Q_{X,Y} \rightarrow P_{X,Y}$ in $\C$ such that $\beta \circ \alpha = 1_{P_{X,Y}}$ and $\alpha \circ \beta = 1_{Q_{X,Y}}$ so that the objects $P_{X,Y}$ and $Q_{X,Y}$ are isomorphic.

If the product of all object pairs $X$ and $Y$ exist in $\C$ then we say binary products exist in $\C$.  The existence of binary products implies the existence of arbitrary finite products in $\C$. So if $\{X_i\}_{i=1}^N$ is a finite set of objects in $\C$ then there is an object which we denote by $\prod_{i=1}^N X_i$  (in general, this need not be the cartesian product) as well as arrows $\{p_{X_j}: \prod_{i=1}^N X_i \rightarrow X_j\}_{j=1}^N$. Then if we are given an arbitrary $T \in_{ob} C$ and a family of arrows $f_j: T \rightarrow X_j$ in $\C$ there exists a unique $\C$ arrow $\langle f_1,\ldots,f_N \rangle$ such that for every integer $j \in \{1,2,\ldots,N\}$ the  diagram 
 \begin{equation*}  
 \begin{tikzpicture}[baseline=(current bounding box.center)]
 	\node	(T)	at	(0,0)		         {$T$};
	\node	(Xj)	at	(-3,-3)	         {$X_j$};
	\node        (XY)  at      (0,-3)               {$\displaystyle{\prod_{i=1}^N} X_i$};
	
	\draw[->, left] (T) to node  {$f_j$} (Xj);
	\draw[->,right, dashed] (T) to node {$\langle f_1,\ldots,f_N \rangle$} (XY);
	\draw[->,below] (XY) to node {$p_{X_j}$} (Xj);

 \end{tikzpicture}
 \end{equation*}
commutes.  The arrows $p_{X_i}$ defining a product space are often called the projection maps due to the analogy with the cartesian products in the category of sets, $\Set$. 
  
In $\Set$, the product of two sets $X$ and $Y$ is the cartesian product $X \times Y$ consisting of all pairs $(x,y)$ of elements with $x \in X$ and $y \in Y$ along with the two projection mappings $\pi_X\colon X \times Y \rightarrow X$ sending $(x,y) \mapsto x$ and $\pi_Y\colon X \times Y \rightarrow Y$ sending $(x,y) \mapsto y$.  Given any pair of functions $f\colon T \rightarrow X \times Y$ and $g\colon T \rightarrow X \times Y$ the function $\langle f,g \rangle \colon T \rightarrow X \times Y$ sending $t \mapsto (f(t),g(t))$ clearly makes Diagram~\ref{binary products} commute.  But it is also the unique such function because if $\gamma\colon T \rightarrow X \times Y$ were any other function making the diagram commute then the equations
  \be \label{productEqs}
 ( p_X \circ \gamma)(t) = f(t) \quad \textrm{ and } \quad (p_Y \circ \gamma)(t) = g(t)
  \ee
  would also be satisfied.  But since the function $\gamma$ has codomain $X \times Y$ which consist of ordered pairs $(x,y)$ it follows that for each $t \in T$ that $\gamma(t) =  \langle \gamma_1(t), \gamma_2(t) \rangle$ for some functions $\gamma_1\colon T \rightarrow X$ and $\gamma_2\colon T \rightarrow Y$.  Substituting $\gamma = \langle \gamma_1, \gamma_2 \rangle$ into equations \ref{productEqs} it follows that
  \be \nonumber
  \begin{array}{c}
   f(t) = (p_X \circ ( \langle \gamma_1, \gamma_2 \rangle))(t) = p_X(\gamma_1(t), \gamma_2(t)) = \gamma_1(t) \\
    g(t) = (p_Y \circ ( \langle \gamma_1, \gamma_2 \rangle))(t) = p_Y(\gamma_2(t), \gamma_2(t)) = \gamma_2(t)
    \end{array}
 \ee
 from which it follows $\gamma = \langle \gamma_1, \gamma_2 \rangle = \langle f,g \rangle$ thereby proving that there exist at most one such function $T \rightarrow X \times Y$ making the requisite Diagram~\ref{binary products} commute.  If the requirement of the uniqueness of the arrow $\langle f,g \rangle$ in the definition of a product  is dropped then we have the definition of a \emph{weak product} of $X$ and $Y$.

  Given the relationship between the categories $\prob$ and $\M$ it is worthwhile to examine products  in $\M$.  Given  $X, Y \in_{ob} \M$  the product $X \times Y$ is the cartesian product $X \times Y$ of sets endowed with the smallest $\sigma$-algebra such that the two set projection maps $\pi_X\colon X \times Y \rightarrow X$ sending $(x,y) \mapsto x$ and $\pi_Y\colon X \times Y \rightarrow Y$ sending $(x,y)\mapsto y$ are measurable.  In other words, we take the smallest subset of the powerset of $X \times Y$ such that for all $A \in \sa_X$ and for all $B \in \sa_Y$ the preimages $\pi_X^{-1}(A) = A \times Y$ and $\pi_Y^{-1}(B) = X \times B$ are measurable.  Since a $\sigma$-algebra requires that the intersection of any two measurable sets is also measurable it follows that $\pi_X^{-1}(A) \cap \pi_Y^{-1}(B) = A \times B$ must also be measurable.  Measurable sets of the form $A \times B$ are called rectangles and \emph{generate} the collection of all measurable sets defining the $\sigma$-algebra $\sa_{X \times Y}$ in the sense that $\sa_{X \times Y}$ is equal to the intersection of all $\sigma$-algebras containing the rectangles. When the $\sigma$-algebra on a set is determined by the a family of maps $\{p_k\colon X \times Y \rightarrow Z_k\}_{k \in K}$, where $K$ is some indexing set such that all of these maps $p_k$ are measurable we say the $\sigma$-algebra is induced (or generated) by the family of maps $\{p_k\}_{k \in K}$.\footnote{The terminology \emph{initial} is also used in lieu of induced.}  The cartesian product $X \times Y$ with the $\sigma$-algebra induced by the two projection maps $\pi_X$ and $\pi_Y$ is easily verified to be a product of $X$ and $Y$ since given any two measurable maps $f\colon Z \rightarrow X$ and $g\colon Z \rightarrow Y$ the map $\langle f,g \rangle\colon Z \rightarrow X \times Y$ sending $z \mapsto (f(z),g(z))$ is the unique measurable map satisfying the defining property of a product for $(X\times Y, \pi_X,\pi_Y)$.  This $\sigma$-algebra induced by the projection maps $\pi_X$ and $\pi_Y$ is called the product $\sigma$-algebra and the use of the notation $X \times Y$ in $\M$ will imply the product $\sigma$-algebra on the set $X \times Y$.
  
  Having the product $(X \times Y, \pi_X, \pi_Y)$ in $\M$ and the fact that every measurable function $f \in_{ar} \M$ determines an arrow $\delta_f \in_{ar} \prob$,  it is tempting to consider the triple $(X \times Y, \delta_{\pi_X}, \delta_{\pi_Y})$ as a potential product in $\prob$.   However taking this triple fails to be a product space of $X$ and $Y$ in $\prob$ because the uniqueness condition fails; given two probability measures $P\colon 1 \rightarrow X$ and $Q\colon1 \rightarrow Y$ there are many joint distributions $J$ making the diagram 
  \begin{equation}   \label{weakproducts}
 \begin{tikzpicture}[baseline=(current bounding box.center)]
 	\node	(1)	at	(0,0)		         {$1$};
	\node	(Xj)	at	(-3,-3)	         {$X$};
	\node         (Y)   at       (3,-3)               {$Y$};
	\node        (XY)  at      (0,-3)               {$X \times Y$};
	
	\draw[->, left] (1) to node  {$P$} (X);
	\draw[->,right] (1) to node {$Q$} (Y);
	\draw[->,right, dashed] (1) to node {$J$} (XY);
	\draw[->,below] (XY) to node {$\delta_{\pi_X}$} (X);
	\draw[->,below] (XY) to node {$\delta_{\pi_Y}$} (Y);
 \end{tikzpicture}
 \end{equation}
commute. In particular, the tensor product measure defined on rectangles by $(P \otimes Q)(A \times B) = P(A) Q(B)$ extends to a joint probability measure on $X \times Y$ by
\be \label{ltensor1}
(P \otimes Q)(\varsigma) = \int_{y \in Y} P(\Gamma_{\overline{y}}^{-1}(\varsigma)) \, dQ \quad \forall \varsigma \in \sa_{X \times Y}
\ee
or equivalently,
\be \label{rtensor1}
(P \otimes Q)(\varsigma) = \int_{x \in X} Q(\Gamma_{\overline{x}}^{-1}(\varsigma)) \, dP \quad \forall \varsigma \in \sa_{X \times Y}.
\ee
Here $\overline{x}\colon Y \to X$ is the constant function at $x$ and $\Gamma_{\overline x}\colon Y \to X \times Y$ is the associated graph function, with $\overline y$ and $\Gamma_{\overline y} $ defined similarly. 
The fact that $Q \otimes P = P \otimes Q$ is Fubini's Theorem; by taking a rectangle $\varsigma = A \times B \in \sa_{X \times Y}$ the equality of these two measures is immediate since
\be \label{Fubini}
\begin{array}{lcl}
(P \otimes Q)(A \times B) &=& \int_{y \in Y} P( \underbrace{\Gamma_{\overline{y}}^{-1}(A \times B)}_{= \left\{ \begin{array}{ll} A & \textrm{ iff } y \in B \\ \emptyset & \textrm{ otherwise } \end{array} \right.  }) \, dQ \\
&=& \int_{y \in B} P(A) \, dQ \\
&=& P(A) \cdot Q(B) \\
&=& \int_{x \in A} Q(B) \, dP \\
&=& \int_{x \in X} Q( \Gamma_{\overline{x}}^{-1}(A \times B)) \, dP \\
&=&(Q \otimes P)(A \times B)
\end{array}
\ee
Using the fact that every measurable set $\varsigma$ in $X \times Y$ is a countable union of rectangles, Fubini's Theorem follows.

It is clear that in $\prob$ the uniqueness condition required in the definition of a product of $X$ and $Y$ will always fail unless at least one of $X$ and $Y$ is a terminal object $1$, and consequently  only  weak products exist in $\prob$.    However it is the nonuniqueness of products in $\prob$ that makes this category interesting.  Instead of referring to weak products in $\prob$ we shall abuse terminology and simply refer to them as products with the understanding that all products in $\prob$ are weak.

\subsection{Constructing a Joint Distribution Given Conditionals}
 \label{section::construct_joint}
 We now show how marginals and conditionals can be used to determine joint distributions in $\prob$.  Given a conditional probability measure $h\colon  X \to Y$ and a probability measure  $P_{X}\colon  1 \to X$  on $X$, consider the diagram 
  \begin{equation}   \label{jointDefinition}
 \begin{tikzpicture}[baseline=(current bounding box.center)]
 	\node	(1)	at	(0,0)		         {$1$};
	\node	(X)	at	(-4,-3)	                  {$X$};
	\node	(Y)	at	(4,-3)               {$Y$};
	\node        (XY)  at      (0,-3)               {$X \times Y$};
	
	\draw[->, left, above] (1) to node  {$P_X$} (X);
	\draw[->,dashed,left,above] (XY) to node {$\delta_{\pi_{X}}$} (X);
	\draw[->,dashed,right,auto] (XY) to node {$\delta_{\pi_{Y}}$} (Y);
	\draw[->,dashed, below,auto] (1) to node {$J_{h}$} (XY);
	\draw[->]  (X) to [out=305,in=235,looseness=.5]  (Y) ;
	\draw (0,-3.9) node {$h$};
 \end{tikzpicture}
 \end{equation}
where $J_{h}$ is the uniquely determined joint distribution on the product space $X \times Y$ defined on the rectangles of the $\sigma$-algebra $\sa_{X} \times \sa_{Y}$ by
 \be \label{jointD}
 J_{h}(A \times B) = \int_{A} h_{B} \, dP_X.
 \ee
 The marginal of $J_{h}$ with respect to $Y$ then satisfies $\delta_{\pi_{Y}} \circ J_{h} = h \circ P_X$ and the marginal of $J_{h}$ with respect to $X$ is $P_X$. By a symmetric argument, if we are given a probability measure $P_Y$  and conditional probability $k\colon  Y \to X$
 then we obtain a unique joint distribution $J_{k}$ on the product space $X \times Y$ given on the rectangles by
 \be \nonumber
 J_{k}(A \times B) = \int_{B} k_{A} \, dP_Y.
 \ee
However if we are given $P_{X}, P_{Y}, h, k$ as indicated in the diagram
\be
 \begin{tikzpicture}[baseline=(current bounding box.center)]
 	\node	(1)	at	(0,0)		         {$1$};
	\node	(X)	at	(-3,-4)	                 {$X$};
	\node	(Y)	at	(3,-4)               {$Y$,};
	\node        (XY)  at      (0,-2)               {$X \times Y$};
	
	\draw[->, above right] (1) to node  {$P_Y$} (Y);
	\draw[->, above left] (1) to node {$P_{X}$} (X);
	\draw[->,dashed, right] (XY) to node[xshift=8pt] {$\delta_{\pi_{X}}$} (X);
	\draw[->,dashed, left] (XY) to node {$\delta_{\pi_{Y}}$} (Y);
	\draw[->,dashed, right] ([xshift=2pt] 1.south) to node {$J_{k}$} ([xshift=2pt] XY.north);
	\draw[->,dashed, left] ([xshift=-2pt] 1.south) to node {$J_{h}$} ([xshift=-2pt] XY.north);
	\draw[->, above] ([yshift=2pt] X.east) to node {$h$} ([yshift=2pt] Y.west);
	\draw[->, below] ([yshift=-2pt] Y.west) to node {$k$} ([yshift=-2pt] X.east);

 \end{tikzpicture}
 \ee
then we have that $J_{h}=J_{k}$ if and only if the compatibility condition \index{compatibility condition} is satisfied on the rectangles 
\be  
\label{eqn::product_rule}
 \int_{A} h_{B} \, dP_X = J( A \times B ) = \int_{B} k_{A} \, dP_Y \quad \forall A \in \sa_X, \forall B \in \sa_Y.
\ee

 In the extreme case, suppose we have a conditional $h\colon  X \to Y$ which factors through the terminal object $1$ as 
 \begin{equation}  \nonumber
 \begin{tikzpicture}[baseline=(current bounding box.center)]
 
 	\node	(X)	at	(0,0)		         {$X$};
	\node	(Y)	at	(3,0)	         {$Y$};
	\node	(1)	at	(1.5,-1)          {$1$};

	\draw[->, above] (X) to node  {$h$} (Y);
	\draw[->, below left] (X) to node {$!$} (1);
	\draw[->,below right]  (1)  to node {$Q$}  (Y);

 \end{tikzpicture}
 \end{equation}
 where $!$ represents the unique arrow from $X \to 1$. If we are also given a probability measure $P\colon  1 \to X$, then we can calculate the joint distribution determined by $P$ and $h=Q \circ !$ as
\be  \nonumber
\begin{array}{lcl}
J(A \times B) &=& \int_{A} (Q \circ !)_B \, dP \\
&=& P(A) \cdot Q(B) 
\end{array}
\ee
so that  $J=P \otimes Q$.  In this situation we say that the marginals $P$ and $Q$ are \emph{independent}.   Thus in $\prob$ independence corresponds to  a special instance of a conditional---one that factors through the terminal object.

\subsection{Constructing Regular Conditionals given a Joint Distribution} 

The following result is the theorem from which the inference maps in Bayesian probability theory are constructed. The fact that we require equiperfect families of probability measures is critical for the construction. 

\begin{thm}  \label{regularConditional} Let $X$ and $Y$ be countably generated measurable spaces and $(X \times Y, \sa_{X \times Y})$ the  product in $\M$ with projection map $\pi_Y$.  If $J$ is a joint distribution on $X \times Y$ with marginal $P_Y=\delta_{\pi_Y} \circ J$ on $Y$, then there exists a $\prob$ arrow $f$ that makes the diagram
  \begin{equation}   \label{productDiagram}
 \begin{tikzpicture}[baseline=(current bounding box.center)]
 	\node	(1)	at	(0,0)		         {$1$};
	\node	(Y)	at	(3,-2)	               {$Y$};
	\node        (XY)  at      (0,-2)               {$X \times Y$};
	
	\draw[->,right] (1) to node  {$P_Y$} (Y);
	\draw[->,left] (1) to node {$J$} (XY);
	\draw[->,above] (XY) to node {$\delta_{\pi_Y}$} (Y);
	\draw[->,dashed] (Y) to [out=210,in=-20,looseness=.5] (XY);
	         \draw (1.5, -2.65) node {$f$};
 \end{tikzpicture}
 \end{equation}
commute and satisfies
\be \nonumber
	\int_{A \times B} {\delta_{\pi_{Y}}}_{C}\,dJ = \int_{C} f_{A \times B}\,dP_{Y}. 
\ee
Moreover, this $f$ is the unique $\prob$-morphism with these properties,  up to a set of $P_Y$-measure zero.  
\end{thm}
\begin{proof} 
Since $\sa_X$ and $\sa_Y$ are both countably generated, it follows that $\sa_{X \times Y}$ is countably generated as well.  Let $\G$ be a countable generating set for $\sa_{X \times Y}$.  For each $A \in \G$, define a measure $\mu_{A}$ on $Y$ by 
\[
	\mu_{A}(B) = J( A \cap \pi_{Y}^{-1}B).
\]
Then $\mu_{A}$ is absolutely continuous with respect to $P_{Y}$ and hence we can let $\widetilde{f}_{A} = \frac{d\mu_{A}}{dP_{Y}}$, the Radon--Nikodym derivative.  For each $A \in \G$ this Radon--Nikodym derivative is unique up to a set of measure zero,  say $\hat{A}$. Let $N= \cup_{A \in \A} \hat{A}$ and $E_1=N^c$.  Then $\widetilde{f}_A|_{E_1}$ is unique for all $A \in \A$.  Note that  $f_{X \times Y} =1$ and $f_{\emptyset}=0$ on $E_1$. The  condition $\widetilde{f}_A \leq 1$ on $E_1$  for all $A \in \A$ then follows.

For all $B \in \sa_Y$ and any countable union $\cup_{i=1}^{n} A_i$ of disjoint sets of $\A$  we have
\be \nonumber
\begin{array}{lcl}
	\int_{B \cap E_1} \widetilde{f}_{\cup_{i=1}^{n} A_i} dP_Y &=& J\left((\cup_{i=1}^{n} A_i) \cap \pi_Y^{-1}B \right)\\
														&=& \sum_{i=1}^{n} J(A_i \cap \pi_Y^{-1}B)\\
														&=& \int_{B \cap E_1} \sum_{i=1}^{n} \widetilde{f}_{A_i} dP_Y,
\end{array}
\ee
with the last equality following from the Monotone Convergence Theorem and the fact that all of the $\widetilde{f}_{A_i}$ are nonnegative. From the uniqueness of the Radon--Nikodym derivative  it follows
\be \nonumber
	\widetilde{f}_{\cup_{i=1}^{n} A_i} = \sum_{i=1}^{n} \widetilde{f}_{A_i}  \quad P_Y\text{-a.e.}
\ee
Since there exist only a countable number of finite collection of sets of $\A$  we can find a set $E \subset E_1$ of $P_Y$-measure one such that the normalized set function $\widetilde{f}_{\cdot}(y) \colon  \A \rightarrow [0,1]$ is finitely additive on $E$.

  These facts altogether show there exists a set $E \in \sa_Y$ with $P_Y$-measure one where for all $y \in E$,
\begin{enumerate}
 \item  $0 \le \widetilde{f}_A(y) \le 1 \quad \forall A \in \A$, 
 \item  $\widetilde{f}_{\emptyset}(y) = 0$ and $\widetilde{f}_{X \times Y}(y) = 1$, and
\item for any finite collection $\{ A_i \}_{i=1}^n$ of disjoint sets of $\A$ we have $\widetilde{f}_{\cup_{i=1}^{n} A_i}(y) = \sum_{i=1}^{n} \widetilde{f}_{A_i}(y)$.
\end{enumerate}
Thus the set function $\widetilde{f}\colon E  \times \A \rightarrow [0,1]$ satisfies the condition that $\widetilde{f}(y,\cdot)$ is a probability measure on the algebra  $\A$.  By the Caratheodory extension theorem there exist a unique extension of $\widetilde{f}(y,\cdot)$ to a probability measure \mbox{$\hat{f}(y,\cdot)\colon \sa_{X \times Y} \rightarrow [0,1]$}.
Now define a set function $f\colon  Y \times \sa_{X \times Y} \to [0,1]$ by 
\be \nonumber
f(y, A) = \left\{
\begin{array}{ll}
\hat{f} (y,A) & \textrm{if $y \in E$}\\
J(A) & \textrm{if $y \notin E$}
\end{array} \right..
\ee
Since each $A \in \sa_{X \times Y}$ can be written as the pointwise limit of an increasing sequence $\{A_n\}_{n=1}^{\infty}$ of sets $A_n \in \A$ it follows that  $f_A = \lim_{n \rightarrow \infty} f_{A_n}$ is measurable.  From this we also obtain the desired commutativity of the diagram
\be \nonumber
\begin{array}{lcl}
	f \circ P_Y(A) &=& \int_Y f_A dP_Y = \int_E f_A dP_Y =\lim_{n \rightarrow \infty} \int_E \widetilde{f}_{A_n} dP_Y \\
	&=& \lim_{n \rightarrow \infty} \int_Y \widetilde{f}_{A_n} dP_Y  \\
	&=&  \lim_{n \rightarrow \infty} J(A_n) \\
	&=& J(A)
\end{array}
\ee
\end{proof}

We can use the result from Theorem~\ref{regularConditional} to obtain a broader understanding of the situation. 

\begin{cor} \label{inferenceExistence}
Let $X$ and $Y$ be countably generated measurable spaces and $J$ a joint distribution on $X \times Y$ with marginal distributions $P_X$ and $P_Y$ on $X$ and $Y$, respectively.  Then there exist  $\prob$ arrows $f$ and $g$
such that the diagram
  \begin{equation}  \nonumber
 \begin{tikzpicture}[baseline=(current bounding box.center)]
 	\node	(1)	at	(0,0)		         {$1$};
	\node	(Y)	at	(3,-4)	               {$Y$};
	\node	(X)	at	(-3,-4)	               {$X$};
	\node        (XY)  at      (0,-2)               {$X \times Y$};
	
	\draw[->,right] (1) to node  {$P_Y$} (Y);
	\draw[->,left] (1) to node  {$P_X$} (X);
	\draw[->,left] (1) to node {$J$} (XY);
	\draw[->,above] (XY) to node {$\delta_{\pi_{Y}}$} (Y);
	\draw[->,above] (XY) to node {$\delta_{\pi_{X}}$} (X);
	\draw[->] (X) to [out=0,in=180,looseness=.5] (Y);
	\draw[->, below,out=210,in=-30,looseness=.5] (Y) to node {$\delta_{\pi_X} \circ f$} (X);
	\draw[->,dashed] (X) to [out=20,in=250,looseness=.5] (XY);
	\draw[->,dashed] (Y) to [out=160,in=290,looseness=.5] (XY);
         \draw (0.7, -3.2) node {$f$};
          \draw (-.7, -3.2) node {$g$};

         \draw (0, -3.7) node {$\delta_{\pi_Y} \circ g$};
 \end{tikzpicture}
 \end{equation}
commutes and 
\be \nonumber
\int_U (\delta_{\pi_Y} \circ g)_V \, dP_X = J(U \times V) = \int_V (\delta_{\pi_X} \circ f)_U \,dP_Y.
\ee

\end{cor}
\begin{proof}  From Theorem~\ref{regularConditional} there exist a $\prob$ arrow $Y \stackrel{f}{\longrightarrow} X \times Y$ satisfying $J =f \circ P_Y$. Take the composite  $\delta_{\pi_X} \circ f$ and note  $(\delta_{\pi_X} \circ f)_U(y) = f_y(U \times Y)$ giving
\be \nonumber
\begin{array}{lcl}
\int_V (\delta_{\pi_X} \circ f)_U dP_Y &=& \int_V   f_{U \times Y}  dP_Y \\
&=& J(U \times Y \cap \pi_{Y}^{-1}V) \\
&=& J(U \times V)
\end{array}
\ee

Similarly using a $\prob$ arrow $X \stackrel{g}{\longrightarrow} X \times Y$ satisfying $J =g\circ P_X$ gives 
\be  \nonumber
\int_U (\delta_{\pi_Y} \circ g)_V dP_X = J(U \times V). 
\ee
\end{proof}

Note that if the joint distribution $J$ is \emph{defined} by a probability measure $P_X$ and a conditional $h\colon X \rightarrow Y$ using Diagram~\ref{jointDefinition}, then using the above result and notation it follows $h= \delta_{\pi_Y} \circ g$.

\section{The Bayesian Paradigm using $\prob$}  \label{sec:BayesianModel}

  The categorical paradigm of Bayesian probability can be compactly summarized with as follows. Let $D$ and $H$ be measurable spaces, which model a data and hypothesis space, respectively.  For example, $D$ might be a Euclidean space corresponding to some measurements that are being taken and $H$ a parameterization of some decision that needs to be made. 
 
 \begin{figure} [H]
 \begin{equation} \nonumber
 \begin{tikzpicture}[baseline=(current bounding box.center)]
         \node         (1)    at      (0,2)         {$1$};
	\node	(H)	at	(-2,0)	      {$H$};
	\node	(D)	at	(2,0)               {$D$};	
	\draw[->,above left] (1) to node {$P_H$} (H);
	\draw[->, above] ([yshift=2pt] H.east) to node {$\mathcal{S}$} ([yshift=2pt] D.west);
         \draw[->, below,dashed] ([yshift=-2pt] D.west) to  node {$\mathcal I$} ([yshift=-2pt] H.east);
 \end{tikzpicture}
 \end{equation}
 \caption{The generic Bayesian model.}
 \label{fig:genericBM}
 \end{figure}
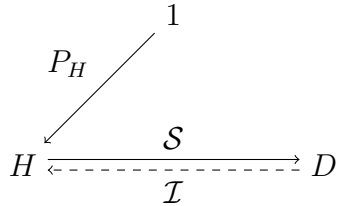
 
The notation $\mcS$ is used to emphasize the fact we think of $\mcS$ as a \emph{sampling distribution} on $D$.  In the context of Bayesian probability the (perfect) probability measure $P_H$ is often called a \emph{prior probability} or, for brevity,  just a \emph{prior}.   Given a prior $P$ and sampling distribution $\mcS$ the joint distribution $J\colon 1 \rightarrow H \times D$ can be constructed using Definition~\ref{jointD}.  Using the marginal $P_D = \mcS \circ P_H$ on $D$ it follows by  Corollary 2.2 there exist an arrow $f\colon D \rightarrow H \times D$ satisfying $J = f \circ P_D$.  Composing this arrow $f$ with the coordinate projection $\delta_{\pi_H}$ gives an arrow $\mcI = \delta_{\pi_H} \circ f\colon  D \rightarrow H$ which we refer to as the inference map, and it satisfies
 \be  \label{productRule}
 \int_{B} \mathcal{I}_{A} \, dP_D = J(A \times B) =  \int_{A} \mathcal{S}_{B} \, dP_H \quad \forall A \in \sa_H, \textrm{ and } \forall B \in \sa_D
 \ee
which is called the product rule.  

With the above in mind we formally  define a {\em Bayesian model}  to consist of 
\begin{enumerate}
	\item[(i)]  two measurable spaces $H$ and $D$ representing hypotheses and data, respectively,
 	\item[(ii)] a probability measure $P_H$ on the $H$ space called the prior probability,
 	\item[(iii)] a $\prob$ arrow $\mathcal{S}\colon H \rightarrow D$ called the sampling distribution,
\end{enumerate}
The sampling distribution $\mcS$ and inference map $\mcI$ are often written as $\p{D\mid Y}$ and $\p{H\mid D}$, respectively, although using the notation $\p{\cdot\mid\cdot}$ for all arrows in the category which are necessarily conditional probabilities is notationally redundant and nondistinguishing (requiring the subscripts to distinguish arrows).  

Given this model and a measurement  $\mu$, which is often just a point mass on $D$ (i.e., $\mu = \delta_d\colon 1 \to D$), there is an update procedure that incorporates this measurement and the prior probability.  Thus the measurement $\mu$ can itself be viewed as a probability measure on $D$, and the  ``posterior'' probability measure can be calculated as $\hat{P}_H =\mcI \circ \mu$ on $H$ provided the measurement $\mu$ is absolutely continuous with respect to $P_D$, which we write as $\mu \ll P_{D}$. Informally, this means that the observed measurement is considered ``possible'' with respect to prior assumptions. 

Let us expand upon this condition $\mu \ll P_{D}$ more closely. We know from Theorem~\ref{regularConditional} that the inference map $\mathcal I$ is uniquely determined by $P_{H}$ and $\mcS$ up to a set of $P_{D}$-measure zero. In general, there is no reason {\em a priori} that an arbitrary (perfect) probability measurement $\mu\colon  1 \to D$ is required to be absolutely continuous with respect to $P_{D}$. If $\mu$ is not absolutely continuous with respect to $P_{D}$, then a different choice of inference map $\mathcal I'$ could yield a different posterior probability---{\em i.e.}, we could have $\mathcal I \circ \mu \neq \mathcal I' \circ \mu$. Thus we make the assumption that measurement probabilities on $D$ are absolutely continuous with respect to the prior probability $P_{D}$ on $D$. 

In practice this condition is often not met.  For example  the probability measure $P_D$ may be a normal distribution on $\mathbb{R}$ and consequently $P_D(\{y\})=0$ for any point $y \in \mathbb{R}$. Since Dirac measurements do not satisfy $\delta_y  \ll P_{D}$, this could create a problem.  However, it is clear that the Dirac measures can be approximated arbitrarily closely by a limiting process of sharply peaked normal distributions which do satisfy this absolute continuity condition.  Thus while the absolute continuity  condition may not be satisfied precisely the error in approximating the measurement by assuming a Dirac measure is negligible. Thus it is standard to assume that measurements belong to a particular class of probability measures on $D$ which are broad enough to approximate measurements and known to be absolutely continuous with respect to the prior. 

In summary, the Bayesian process works in the following way. Given a prior probability $P_{H}$ and sampling distribution $\mathcal S$ one determines the inference map $\mathcal I$. (For computational purposes the construction of the entire map $\mcI$ is in general not necessary.) Once a measurement $\mu\colon  1 \to D$ is taken, we then calculate the posterior probability by $\mathcal I \circ \mu$. This updating procedure can be characterized by the diagram
 \begin{equation} \label{BayesModel}
 \begin{tikzpicture}[baseline=(current bounding box.center)]
         \node         (1)    at       (0,2.5)              {$1$};
	\node	(H)	at	(-2.5,0)	      {$H$};
	\node	(D)	at	(2.5,0)               {$D$};
	\draw[->,above left] (1) to node {$P_H$} (H);
	\draw[->,dashed,above right] (1) to node {$\mu$} (D);
	\draw[->, above] ([yshift=2pt] H.east) to node {$\mathcal S$} ([yshift=2pt] D.west);
         \draw[->, densely dotted] ([yshift=-2pt] D.west) to node [yshift=-6pt] {$\mathcal I$}  ([yshift=-2pt] H.east);
	\draw[->,dashed,right, out=255,in=30] (1) to node[xshift=5pt] {$\mathcal I \circ \mu$} (H);
 \end{tikzpicture}
 \end{equation}
 where the solid lines indicate arrows given {\em a priori}, the dotted line indicates the arrow determined using Theorem~\ref{regularConditional}, and the dashed lines indicate the updating after a measurement. Note that if there is no uncertainty in the measurement, then $\mu=\delta_{\{x\}}$ for some $x \in D$, but in practice there is usually some uncertainty in the measurements themselves.  Consequently the posterior probability must be computed as a composite - so the \emph{posterior probability} of an event $A \in \sa_H$  given a measurement $\mu$ is $(\mcI \circ \mu)(A) = \int_D \mcI_{A}(x) \, d\mu$.
  
 Following the calculation of the posterior probability, the sampling distribution is then updated, if required. The process can then repeat:  using the posterior probability and the updated sampling distribution the updated joint probability distribution on the product space is determined and the corresponding (updated) inference map determined (for computational purposes the ``entire map'' $\mcI$ need not be determined if the measurements are deterministic).  We can then continue to iterate as long as new measurements are received.  For some problems, such as with the standard urn problem with replacement of balls, the sampling distribution does not change from iterate to iterate, but the inference map is updated since the posterior probability on the hypothesis space changes with each measurement.

\begin{remark}  Note that for  countable spaces $X$ and $Y$ the compatibility condition reduces to the standard Bayes equation since for any $x \in X$ the singleton $\{x \} \in \sa_X$ and similarly any element $y \in Y$ implies $\{y \} \in \sa_Y$, so that the joint distribution $J\colon  1\rightarrow X \times Y$ on $\{x\} \times \{y\}$ reduces to the equation
\be 
\mcS(\{y\} \mid x) P_X(\{x\}) = J(\{x\} \times \{y\}) = \mcI(\{x\} \mid y) P_Y(\{y\})
\ee
which in more familiar notation is the Bayesian equation
\be 
	P(y \mid x) P(x) = P(x,y) = P(x \mid y)P(y).
\ee

\end{remark}
  
  \section{Elementary applications of Bayesian probability}

Before proceeding to show how the category $\prob$ can be can be applied to ML where the unknowns are functions, we illustrate its use to solve inference, prediction, and decision processes in the more familiar setting where the unknown parameter(s)  are real values. We present two  elementary problems illustrating  basic model building using  categorical diagrams, much like that used in probabilistic graphical models for Bayesian networks, which can serve to clarify the modeling aspect of any probabilistic problem.

To illustrate the inference-sampling distribution relationship and how we make computations in the category $\prob$,  we  consider first an urn problem where we have  discrete $\sigma$-algebras.  The discreteness condition is not critical as we will eventually see - it only makes the analysis and \emph{computational} aspect easier. 

\begin{example} \label{urnExample}
\textbf{Million dollar draw}.\footnote{ This problem is taken from Peter Green's tutorial on Bayesian Inference which can be viewed at http://videolectures.net/mlss2011\_green\_bayesian.}
 
 \tikzset{redball/.style={%
	circle,
	shading=ball,
	ball color=red,
  	minimum size=.5cm,
	text = white
	},
	blueball/.style={%
	circle,
	shading=ball,
	ball color = blue,
	text = white
	}}

\begin{center}
\begin{tikzpicture}
  \draw (-.5,0)  rectangle +(3,3);
 \node[redball] at (0,.5) {R};
 \node[blueball] at (1,.5) {B};
 \node[redball] at (2,.5) {R};
 \node[blueball] at (0,1.5) {B};
  \node[redball] at (1,1.5) {R};
  \draw (1,-.5) node {Urn 1};
  \draw (6,-.5) node {Urn 2};

  \draw (4.5,0)  rectangle +(3,3);
  \node[redball] at (5,.5) {R};
 \node[blueball] at (6,.5) {B};
 \node[blueball] at (7,.5) {B};
  \node[blueball] at (5,1.5) {B};
\end{tikzpicture}
\end{center}
 
You are given two draws and if you pull out a red ball you win a million dollars.   You are unable to see the two urns so you don't know which urn you are drawing from and the draw is done without replacement.  The $\prob$ diagram for both inference and calculating sampling distributions is given by

 \begin{equation}  \nonumber
 \begin{tikzpicture}[baseline=(current bounding box.center)]
         \node         (1)     at     (-1.5,2)              {$1$};
	\node	(U)	at	(-3,0)	      {$U$};
	\node	(B)	at	(0,0)               {$B$};
	
	\draw[->, left] (1) to node {$P_U$} (U);
	\draw[->,above] ([yshift=2pt] U.east) to node {$\mcS$} ([yshift = 2pt] B.west);
         \draw[->,below,dashed] ([yshift = -2pt] B.west) to  node {$\mcI$} ([yshift=-2pt] U.east);
         \draw[->,right,dashed] (1) to node {$P_B$} (B);

 \end{tikzpicture}
 \end{equation}
 where the dashed arrows indicate morphisms to be calculated rather than morphisms determined by modeling,
 \be \nonumber
 \begin{array}{l}
 U=\{u_1, u_2\} = \textrm{\{Urn 1, Urn 2\}} \\
 B=\{b,r\} = \textrm{\{blue, red\}}
 \end{array}
 \ee
and 
\be  \nonumber
P_U = \frac{1}{2} \delta_{u_1} + \frac{1}{2} \delta_{u_2}.
\ee
The sampling distribution is the binomial distribution given by 
\be  \nonumber
\begin{array}{ll}
\mcS(\{b\} \mid u_1)=\frac{2}{5} & \mcS(\{r\} \mid u_1)=\frac{3}{5} \\
\mcS(\{b\} \mid u_2)=\frac{3}{4} & \mcS(\{r\} \mid u_2)=\frac{1}{4}.
\end{array}
\ee

Suppose that on our first draw, we draw from one of the urns (which one is unknown) and draw a blue ball.  We ask the following questions:
\begin{enumerate}
\item (Inference) What is the probability that we made the draw from Urn 1 (Urn 2)?
\item (Prediction) What is the probability of drawing a red ball on the second draw (from the same urn)?
\item (Decision) Given you have drawn a blue ball on the first draw should you switch urns to increase the probability of drawing a red ball?
\end{enumerate}

To solve these problems, we implicitly or explicitly construct the joint distribution $J$ via the standard construction given $P_U$ and the conditional $\mcS$

\begin{center}
 \begin{tikzpicture}[baseline=(current bounding box.center)]
 	\node	(1)	at	(0,0)		         {$1$};
	\node	(U)	at	(-4,-3)	                  {$U$};
	\node	(B)	at	(4,-3)               {$B$};
	\node        (UB)  at      (0,-3)               {$U \times B$};
	
	\draw[->, left, below, auto] (1) to node  {$$} (U);
	\draw[->,right,below,auto] (1) to node {$P_B=\mcS \circ P_U$} (B);
	\draw[->,left,auto] (UB) to node {$$} (U);
	\draw[->,right,auto] (UB) to node {$\delta_{\pi_{B}}$} (B);
	\draw (-2.3,-1.3) node {$P_U$};                
	\draw (-2.3, -2.7) node {$\delta_{\pi_{U}}$};
	\draw[->,dashed, below,auto] (1) to node {$J$} (UB);
	\draw[->]  (U) to [out=305,in=235,looseness=.5]  (B) ;
	\draw (0,-3.9) node {$\mcS$};
 \end{tikzpicture}
\end{center}
and then construct the inference map by requiring the compatibility condition, {\em i.e.}, the integral equation
\be \label{integralEq}
\int_{u \in U} \mcS(\mathcal{B} | u) dP_U =J(\mathcal{B} \times \mathcal{H}) =  \int_{c \in B} \mcI(\mathcal{H} | c) dP_B \quad \forall \mathcal{B} \in \sa_B \quad \forall \mathcal{H} \in \sa_U
\ee
is satisfied.  Since our problem is discrete the integral reduces to a sum.

Our first step is  to calculate the prior on $B$ which is the composite $P_B = \mcS \circ P_U$, from which we calculate
\be  \nonumber
\begin{array}{lcl}
P_B(\{b\}) &=& (\mcS \circ P_U)(\{b\}) \\
&=& \int_{v \in U} \mcS(\{b\} | v) dP_U \\
&=& \int_{v \in U} \mcS(\{b\} | v) d( \frac{1}{2} \delta_{u_1} + \frac{1}{2} \delta_{u_2}) \\
&=& \mcS(\{b\} | u_1)\cdot P_U(\{u_1\}) + \mcS(\{b\} | u_2)\cdot P_U(\{u_2\}) \\
&=& \frac{2}{5} \cdot \frac{1}{2} + \frac{3}{4} \cdot \frac{1}{2} \\
&=& \frac{23}{40}
\end{array}
\ee
and similarly
\be  \nonumber
P_B(\{r\}) = \frac{17}{40}.
\ee

To solve the \emph{inference} problem,  we need to compute the values of the inference map $\mcI$
using equation~\ref{integralEq}.  This  amounts to computing the joint distribution on all possible measurable sets,
 \be \nonumber
\begin{array}{l}
\int_{\{u_1\}} \mcS( \{b\} | u) dP_U =J( \{u_1\} \times \{b\}) =  \int_{\{b\}} \mcI( \{u_1\} | c) dP_B \\
\int_{\{u_2\}} \mcS( \{b\}  | u) dP_U =J(\{u_2\} \times \{b\}) =  \int_{\{b\}} \mcI(\{u_2\} | c) dP_B \\
\int_{\{u_1\}} \mcS( \{r\} | u) dP_U =J(\{u_1\} \times \{r\} ) =  \int_{\{r\}} \mcI(\{u_1\} | c) dP_B \\
\int_{\{u_2\}} \mcS( \{r\}  | u) dP_U =J(\{u_2\} \times \{r\} ) =  \int_{\{r\}}  \mcI(\{u_2\} | c) dP_B 
\end{array}
\ee
which reduce to the equations
\be  \nonumber
\begin{array}{l}
\mcS(\{b\} | u_1) \cdot P_U(\{u_1\})  = \mcI(\{u_1\} | b) \cdot P_B(\{b\})  \\
\mcS(\{b\} | u_2) \cdot P_U(\{u_2\}) = \mcI(\{u_2\} | b) \cdot P_B(\{b\})  \\
\mcS(\{r\} | u_1) \cdot P_U(\{u_1\})  =  \mcI(\{u_1\} | r) \cdot P_B(\{r\}) \\
\mcS(\{r\} | u_2) \cdot P_U(\{u_2\}) =  \mcI(\{u_2\} | r) \cdot P_B(\{r\}). \\
\end{array}
\ee
Substituting values for $\mcS$, $P_B$, and $P_I$  one determines
\be  \nonumber
\begin{array}{ll}
\mcI( \{u_1\} | b)= \frac{8}{23} & \mcI(\{u_2\} | b)=\frac{15}{23} \\
\\
\mcI(\{u_1\} | r)=\frac{12}{17} & \mcI(\{u_2\} | r)=\frac{5}{17}
\end{array}
\ee
which answers question (1).  The odds that one drew the blue ball from Urn 1 relative to Urn 2 are $\frac{8}{15}$, so it is almost twice as likely that one made the draw from the second urn. 

\textbf{The Prediction Problem.}
Here we implicitly (or explicitly) need to construct the product space $U \times B_1 \times B_2$ where $B_i$ represents the $i^{th}$ drawing of a ball from the same (unknown) urn.  To do this we use the basic construction for joint distributions using a regular conditional probability, $\mcS_2$, which expresses the probability of drawing either a red or a blue ball \emph{from the same urn} as the first draw.  This conditional probability is given by

\be \nonumber
\begin{array}{ll}
\mcS_2( \{b\} | (u_1,b))= \frac{1}{4} & \mcS_2(\{r\} | (u_1,b))=\frac{3}{4} \\
\mcS_2(\{b\} | (u_2,b))=\frac{2}{3} & \mcS_2(\{r\} | (u_2,b))=\frac{1}{3} \\
\mcS_2(\{b\} | (u_1,r))= \frac{1}{2} & \mcS_2(\{r\} | (u_1,r))=\frac{1}{2} \\
\mcS_2(\{b\} | (u_2,r))=1 & \mcS_2(\{r\} | (u_2,r))=0.
\end{array}
\ee

Now we construct the joint distribution $K$ on the product space $(U \times B_1) \times B_2$

\begin{center}
 \begin{tikzpicture}[baseline=(current bounding box.center)]
 	\node	(1)	at	(0,0)		         {$1$};
	\node	(UB1)	at	(-4,-3)	                  {$U \times B_1$};
	\node	(B2)	at	(4,-3)               {$B_2.$};
	\node        (UB12)  at      (0,-3)               {$U \times B_1 \times B_2$};
	
	\draw[->, left, below, auto] (1) to node  {$$} (UB1);
	\draw[->,right,below,auto] (1) to node {$P_{B_2}=\mcS_2 \circ J$} (B2);
	\draw[->,left,auto] (UB12) to node {$$} (UB1);
	\draw[->,right,auto] (UB12) to node {$\delta_{\pi_{B_2}}$} (B2);
	\draw (-2.3,-1.3) node {$J$};                  
	\draw (-2.3, -2.7) node {$\delta_{\pi_{U \times B_1}}$};
	\draw[->,dashed, below,auto] (1) to node {$K$} (UB12);
	\draw[->]  (UB1) to [out=305,in=-135,looseness=.2]  (B2) ;
	\draw (0,-4.1) node {$\mcS_2$};
 \end{tikzpicture}
\end{center}

To answer the prediction question we calculate the odds of drawing a red versus a blue ball.  Thus
\be  \label{prediction}
K(U \times \{b\} \times \{r\}) = \int_{ U \times \{b\}} \mcS_2({\{r\}} | (u,\beta)) dJ,
\ee
where the right hand side follows from the definition (construction) of the iterated product space $(U \times B_1) \times B_2$.  The computation of the expression~\ref{prediction} yields
\be  \nonumber
\begin{array}{lcl}
K(U \times \{b\} \times \{r\}) &=& \int_{ U \times \{b\}} \mcS_2({\{r\}}| (u,\beta)) dJ \\
&=& \underbrace{\mcS( \{r\} | (u_1,b))}_{= \frac{3}{4}} \cdot \underbrace{J(\{u_1\}\times \{b\})}_{=\frac{1}{5}} + \underbrace{\mcS(\{r\} | (u_2,b))}_{=\frac{1}{3} } \cdot \underbrace{J(\{u_2\}\times \{b\})}_{=\frac{3 }{8}} \\
&=& \frac{11}{40}.
\end{array}
\ee
Similarly
$K(U \times \{b\} \times \{b\})=\frac{12}{40}$.  So the odds are
\be  \nonumber
\frac{r}{b} = \frac{11}{12}   \quad Pr( \{r\} | \{b\}) = \frac{11}{23}.
\ee

\textbf{The Decision Problem}
To answer the decision problem we need to consider the conditional probability of switching urns on the second draw which leads to the conditional
\begin{center}
\begin{tikzpicture}[baseline=(current bounding box.center)]
 	\node	(UB1)	at	(0,0)		         {$U \times B_1$};
	\node	(B2)	at	(4,0)	                  {$B_2$};
	\draw[->, right,auto] (UB1) to node  {$\hat{\mcS}_2$} (B2);
 \end{tikzpicture}
\end{center}
given by
\be  \nonumber
\begin{array}{ll}
\hat{\mcS}_2( \{b\} | (u_1,b))= \frac{3}{4} & \hat{\mcS}_2(\{r\} | (u_1,b))=\frac{1}{4} \\
\hat{\mcS}_2(\{b\} | (u_2,b))=\frac{2}{5} & \hat{\mcS}_2(\{r\} | (u_2,b))=\frac{3}{5} \\
\hat{\mcS}_2(\{b\} | (u_1,r))= \frac{3}{4} &\hat{\mcS}_2(\{r\} | (u_1,r))=\frac{1}{4} \\
\hat{\mcS}_2(\{b\} | (u_2,r))=\frac{2}{5} &\hat{ \mcS}_2(\{r\} | (u_2,r))=\frac{3}{5}.
\end{array}
\ee

Carrying out the same computation as above we find the joint distribution $\hat{K}$ on the product space $(U \times B_1) \times B_2$ constructed from $J$ and $\hat{\mcS}_2$ yields
\be   \nonumber
\begin{array}{lcl}
\hat{K}(U \times \{b\} \times \{r\}) &=& \int_{ U \times \{b\}} \hat{ \mcS}_2 (\{r\} | (u,\beta)) dJ \\
&=& \hat{\mcS_2}(\{r\} | (u_1,b)) J(\{u_1\} \times \{b\}) + \hat{\mcS_2}(\{r\} | (u_2,b)) J(\{u_2\} \times \{b\}) \\
&=& \frac{1}{4} \cdot \frac{1}{5} + \frac{3}{5} \cdot \frac{3}{8} \\
&=& \frac{11}{40},
\end{array}
\ee
which shows that it doesn't matter whether you switch or not - you get the same probability of drawing a red ball.

The probability of drawing a blue ball is
\be  \nonumber
\hat{K}(U \times \{b\} \times \{b\}) = \frac{12}{40}=K(U \times \{b\} \times \{b\}),
\ee
so the odds of drawing a blue ball outweigh the odds of drawing a red ball by the ratio $\frac{12}{11}$.  The odds are against you.

\end{example}

Here is an example illustrating that the regular conditional probabilities (inference or sampling distributions) are defined only up to sets of measure zero. 

\begin{example}
We have a rather bland deck of three cards as shown

\tikzset{redcard/.style={%
	rectangle,
	minimum width = 2.5em,
	minimum height = 4.8em,
	text = white,
	fill = red
	},
	greencard/.style={%
	rectangle,
	minimum width = 2.5em,
	minimum height = 4.8em,
	text = white,
	fill = green
	}}

\begin{center}
\begin{tikzpicture}[rounded corners]

\draw (0.3,-.5) node {Card 1};
\draw (2.3,-.5) node {Card 2};
\draw(4.4,-.5) node {Card 3};

\draw(-1.5, 1) node {Front};
\draw(-1.5, -2) node {Back};

\node[redcard] at (.3,1) {$R$};
\node[redcard] at (.3,-2) {$R$};

\node[redcard] at (2.3,1) {$R$};
\node[greencard] at (2.3,-2) {$G$};

\node[greencard] at (4.3,1) {$G$};
\node[greencard] at (4.3,-2) {$G$};

\end{tikzpicture}
\end{center}

We shuffle the deck, pull out a card and expose one face which is red.\footnote{ This problem is taken from David MacKays tutorial on Information Theory which can be viewed at $http://videolectures.net/mlss09uk\_mackay\_it/$.}  The prediction question is 

\begin{center}
{\em What is the probability the other side of the card is red?}
\end{center}

To answer this  note that this card problem is identical to the urn problem with urns being cards and balls becoming the colored sides of each card.  Thus we have an analogous model in $\prob$ for this problem.  Let
\be  \nonumber
\begin{array}{l}
C(ard) = \{1,2,3\} \\
F(ace  \, Color) = \{r, g\}.
\end{array}
\ee
We have the $\prob$ diagram

 \begin{equation}  \nonumber
 \begin{tikzpicture}[baseline=(current bounding box.center)]
         \node         (1)     at     (-6,0)              {$1$};
	\node	(C)	at	(-3,0)	      {$C$};
	\node	(F)	at	(0,0)               {$F$};
	
	\draw[->] (1) to node {} (C);
	\draw (-4.5,.4) node {$P_C$};
	\draw[->] (C) to [out=60,   in=130,looseness=.2] (F);
         \draw[->,dashed] (F) to  [out=240, in = 270,looseness=.4] (C);
         \draw[->,dashed] (1) to [out=290, in=270,looseness=.5] (F);
	\draw (-1.5,.8) node {$\mathcal{S}$};        
	\draw (-1.5,-.2) node {$\mathcal{I}$};
	\draw (-2.9,-1.6) node {$P_F$};

 \end{tikzpicture}
 \end{equation}
 with the sampling distribution given by
 \be  \nonumber
 \begin{array}{lcl}
 \mcS(\{r\} | 1)= 1 & \mcS(\{g\} | 1)=0 \\
 \mcS(\{r\} | 2)= \frac{1}{2} & \mcS(\{g\} | 2)=\frac{1}{2} \\
 \mcS(\{r\} | 3)= 0 & \mcS(\{g\} | 3)=1 .\\
\end{array}
\ee
The prior on $C$ is $P_C = \frac{1}{3} \delta_1 + \frac{1}{3} \delta_2 + \frac{1}{3} \delta_3$.
From this we can construct the joint distribution on $C \times F$
\begin{center}
 \begin{tikzpicture}[baseline=(current bounding box.center)]
 	\node	(1)	at	(0,0)		         {$1$};
	\node	(C)	at	(-4,-3)	                  {$C$};
	\node	(F)	at	(4,-3)               {$F.$};
	\node        (CF)  at      (0,-3)               {$C \times F$};
	
	\draw[->, left, below, auto] (1) to node  {$$} (C);
	\draw[->,right,below,auto] (1) to node {$P_F=\mcS \circ P_C$} (F);
	\draw[->,left,auto] (CF) to node {$$} (C);
	\draw[->,right,auto] (CF) to node {$\delta_{\pi_{F}}$} (F);
	\draw (-2.3,-1.3) node {$P_C$};       
	\draw (-2.3, -2.7) node {$\delta_{\pi_{C}}$};
	\draw[->,dashed, below,auto] (1) to node {$J$} (CF);
	\draw[->]  (C) to [out=305,in=235,looseness=.5]  (F) ;
	\draw (0,-3.9) node {$\mcS$};
 \end{tikzpicture}
\end{center}
Using 
\be  \nonumber
J(A \times B) = \int_{n \in A} \mcS(B | n) dP_C,
\ee
we find
\be  \nonumber
\begin{array}{lcl}
J( \{1\} \times \{r\}) = \frac{1}{3} & J(\{1\} \times \{g\}) = 0 \\
J( \{2\} \times \{r\}) = \frac{1}{6} & J(\{2\} \times \{g\}) = \frac{1}{6} \\
J( \{3\} \times \{r\}) = 0 & J(\{3\} \times \{g\}) = \frac{1}{3} .\\
\end{array}
\ee
Now, like in the urn problem, to predict the next draw (flip of the card), it is necessary to add another measurable set $F_2$ and conditional probability $\mcS_2$ and construct the product diagram and joint distribution $K$
\begin{center}
 \begin{tikzpicture}[baseline=(current bounding box.center)]
 	\node	(1)	at	(0,0)		         {$1$};
	\node	(CF1)	at	(-4,-3)	                  {$C \times F_1$};
	\node	(F2)	at	(4,-3)               {$F_2$.};
	\node        (CF12)  at      (0,-3)               {$C \times F_1 \times F_2$};
	\draw[->, left, below, auto] (1) to node  {$$} (CF1);
	\draw[->,right,below,auto] (1) to node {$P_{F_2}=\mcS_2 \circ J$} (F2);
	\draw[->,left,auto] (CF12) to node {$$} (CF1);
	\draw[->,right,auto] (CF12) to node {$\delta_{\pi_{F_2}}$} (F2);
	\draw (-2.3,-1.3) node {$J$};               
	\draw (-2.3, -2.7) node {$\delta_{\pi_{C \times F_1}}$};
	\draw[->,dashed, below,auto] (1) to node {$K$} (CF12);
	\draw[->]  (CF1) to [out=305,in=-135,looseness=.2]  (F2) ;
	\draw (0,-4.1) node {$\mcS_2$};
 \end{tikzpicture}
\end{center}
The twist now arises in that the conditional probability $\mcS_2$ is not uniquely defined -
what are the values 
\be  \nonumber
\mcS_2( \{r\} | (1,g)) = ~? \quad \mcS_2(\{g\} |  (1,g))=~?
\ee
The answer is it doesn't matter what we put down for these values since they have measure $J(\{1\} \times \{g\})=0$.  We can still compute the desired quantity of interest  proceeding forth with these arbitrarily chosen values on the point sets of measure zero.  Thus we choose
\be  \nonumber
\begin{array}{ll}
\mcS_2( \{g\} | (1,r))= 0 & \mcS_2(\{r\} | (1,r))=1 \\
\mcS_2(\{g\} | (1,g))=1 & \mcS_2(\{r\} | (1,g))=0 \quad \textrm{doesn't matter} \\

\mcS_2(\{g\} | (2,r))=1 & \mcS_2(\{r\} | (2,r))=0 \\
\mcS_2(\{g\} | (2,g))=0 & \mcS_2(\{r\} | (2,g))=1 \\

\mcS_2(\{g\} | (3,r))= 0 & \mcS_2(\{r\} | (3,r))=1 \quad \textrm{doesn't matter} \\
\mcS_2(\{g\} | (3,g))=1 & \mcS_2(\{r\} | (3,g))=0.
\end{array}
\ee
We chose the arbitrary values such that $\mcS_2$ is a deterministic mapping which seems appropriate since flipping a given card uniquely determined the color on the other side. 

Now we can solve the prediction problem by computing the joint measure values
\be  \nonumber
\begin{array}{lcl}
K(C \times \{r\} \times \{r\}) &=& \int_{C \times \{r\}} (\mcS_2)_{\{r\}}(n,c) dJ \\
&=& \mcS_2(\{r\} | (1,r)) \cdot J(\{1\} \times \{r\}) + \mcS_2(\{r\} | (2,r)) \cdot J(\{2\} \times \{r\}) \\
&=& 1 \cdot \frac{1}{3} + 0 \cdot \frac{1}{6} \\
&=& \frac{1}{3}
\end{array}
\ee
and
\be  \nonumber
\begin{array}{lcl}
K(C \times \{r\} \times \{g\}) &=& \int_{C \times \{r\}} \mcS_2(\{g\} | (n,c)) dJ \\
&=& \mcS_2(\{g\} | (1,r)) \cdot J(\{1\} \times \{r\}) + \mcS_2(\{g\} | (2,r)) \cdot J(\{2\} \times \{r\}) \\
&=& 0 \cdot \frac{1}{3} + 1 \cdot \frac{1}{6} \\
&=& \frac{1}{6},
\end{array}
\ee
so it is twice as likely to observe a red face upon flipping the card than seeing a green face.  Converting the odds of $\frac{r}{g} = \frac{2}{1}$ to a probability gives $Pr(\{r\}|\{r\})= \frac{2}{3}$.  

\end{example}

To test one's understanding of the categorical approach to Bayesian probability we suggest the following problem.

\begin{example} \textbf{The Monty Hall Problem.}  You are a contestant in a game show in which a prize is hidden behind one of three curtains.  You will win a prize if you select the correct curtain.  After you have picked one curtain but befor the curtain is lifted, the emcee lifts one of the other curtains,  revealing a goat, and asks if you would like to switch from your current selection to the remaining curtain.  How will your chances change if you switch?

There are three components which need modeled in this problem:
\be  \nonumber
\begin{array}{l}
D(oor) = \{1,2,3\} \quad \textrm{The prize is behind this door.}  \\
C(hoice) = \{1,2,3\} \quad \textrm{The door you chose.} \\
O(pend door) = \{1,2,3\} \quad \textrm{The door Monty Hall opens}
\end{array}
\ee
The prior on $D$ is $P_D=\frac{1}{3} \delta_{d_1} + \frac{1}{3}\delta_{d_2}+\frac{1}{3}\delta_{d_3}$.  Your selection of a curtain, say curtain $1$, gives the deterministic measure $P_C = \delta_{C_1}$.  There is a conditional probability from the product space $D \times C$ to $O$

\begin{center}
 \begin{tikzpicture}[baseline=(current bounding box.center)]
 	\node	(1)	at	(0,0)		         {$1$};
	\node	(DC)	at	(-4,-3)	                  {$D \times C$};
	\node	(O)	at	(4,-3)               {$O$};
	\node        (DCO)  at      (0,-3)               {$(D \times C) \times O$};
	
	\draw[->, left, below, auto] (1) to node  {$$} (DC);
	\draw[->,right,below,auto] (1) to node {$P_O=\mcS \circ P_D \otimes P_C$} (O);
	\draw[->,left,auto] (DCO) to node {$$} (DC);
	\draw[->,right,auto] (DCO) to node {$\delta_{\pi_{O}}$} (O);
	\draw (-2.8,-1.3) node {$P_D \otimes P_C$};
	\draw (-2.3, -2.7) node {$\delta_{\pi_{D \times C}}$};
	\draw[->,dashed, below,auto] (1) to node {$J$} (DCO);
	\draw[->]  (DC) to [out=305,in=235,looseness=.5]  (O) ;
	\draw (0,-3.9) node {$\mcS$};
 \end{tikzpicture}
\end{center}
where the conditional probability $\mcS( (i,j),\{k\})$ represents the probability that Monty opens door $k$ given that the prize is behind door $i$ and you have chosen door $j$.  If you have chosen curtain $1$ then we have the partial data given by
\be  \nonumber
\begin{array}{lll}
\mcS( (1,1),\{1\}) = 0 & \mcS( (1,1),\{2\}) = \frac{1}{2} &  \mcS( (1,1),\{2\}) = \frac{1}{2} \\
\mcS( (2,1),\{1\}) = 0 & \mcS( (2,1),\{2\}) = 0 &  \mcS( (2,1),\{3\}) = 1 \\
\mcS( (3,1),\{1\}) = 0 & \mcS( (3,1),\{2\}) = 1 &  \mcS( (3,1),\{3\}) = 0. \\
\end{array}
\ee
Complete the table, as necessary, to compute the inference conditional, $D \times C \stackrel{\mcI}{\longleftarrow} O$, and 
conclude that if Monty opens either curtain $2$ or $3$ it is in your best interest to switch doors.
\end{example}
  
\section{The Tensor Product}
  
Given any function $f\colon X \rightarrow Y$ the graph of $f$ is defined as the set function 
\be  \nonumber
\begin{array}{ccccc}
\Gamma_f &\colon & X & \longrightarrow & X \times Y \\
&\colon & x & \mapsto & (x,f(x)).
\end{array}
\ee
By our previous notation $\Gamma_f = \langle Id_X,f \rangle$.  If $g\colon Y \rightarrow X$ is any function we also refer to the set function
\be  \nonumber
\begin{array}{ccccc}
\Gamma_g &\colon & Y & \longrightarrow & X \times Y \\
&\colon & y & \mapsto & (g(y),y)
\end{array}
\ee
as a graph function.

Any fixed $x \in X$ determines a constant function $\overline{x}\colon Y \rightarrow X$ sending every $y \in Y$ to $x$. These functions are always measurable and consequently determine ``constant'' graph functions $\Gamma_{\overline{x}}\colon Y \rightarrow X \times Y$.  Similarly, every fixed $y \in Y$ determines a constant graph function $\Gamma_{\overline{y}}\colon X \rightarrow X \times Y$. Together, these constant graph functions can be used to define a  $\sigma$-algebra on the set $X \times Y$ which is finer (larger) than the product $\sigma$-algebra $\sa_{X \times Y}$.  Let $X \otimes Y$ denote the set $X \times Y$ endowed with the largest $\sigma$-algebra structure such that all the constant graph functions  $\Gamma_{\overline{x}}\colon X \rightarrow X \otimes Y$ and $\Gamma_{\overline{y}}\colon Y \rightarrow X \otimes Y$ are measurable.   We say this $\sigma$-algebra $X \otimes Y$ is  \emph{coinduced} by the maps $\{\Gamma_{\overline{x}}\colon  X \rightarrow X \times Y\}_{x \in X}$ and  $\{\Gamma_{\overline{y}}\colon  Y \rightarrow X \times Y\}_{y \in Y}$.  Explicitly, this $\sigma$-algebra is given by
\be
	\sa_{X \otimes Y} = \bigcap_{x \in X} {\Gamma_{\overline{x}}}_{\ast}\sa_{Y} \cap \bigcap_{y \in Y}{\Gamma_{\overline y}}_{\ast}\sa_{X},
\ee
where for any function $f\colon  W \to Z$,
\be
	f_{\ast}\sa_{W} = \{C \in 2^{Z} \mid f^{-1}(C) \in \sa_{W} \}.
\ee
This is in  contrast to the smallest $\sigma$-algebra on $X \times Y$, defined in Section~\ref{sec:products} so that the two projection maps $\{\pi_X\colon X \times Y \rightarrow X, \pi_Y\colon X \times Y \rightarrow Y\}$ are measurable. Such a $\sigma$-algebra is said to be \emph{induced} by the projection maps, or simply referred to as the \emph{initial} $\sigma$-algebra.

The following result on coinduced $\sigma$-algebras is used repeatedly.
   
\begin{lemma} \label{coinduced} Let the $\sigma$-algebra of $Y$ be coinduced by a collection of maps $\{f_i \colon  X_i \rightarrow Y \}_{i \in I}$.   Then any map $g\colon Y \rightarrow Z$ is measurable if and only if the composition $g \circ f_i$ is measurable for each $i \in I$.
\end{lemma}
\begin{proof} Consider the diagram
 \begin{figure}[H]
\begin{center}
 \begin{tikzpicture}[baseline=(current bounding box.center)]
 	\node	(X)	at	(0,0)              {$X_i$};
	\node	(Y)	at	(2,0)	               {$Y$};
	\node	(Z)	at	(2,-2)               {$Z$};

	\draw[->, above] (X) to node  {$f_i$} (Y);
	\draw[->,right] (Y) to node [xshift=0pt,yshift=0pt] {$g$} (Z);
	\draw[->,left] (X) to node [xshift = -5pt] {$g \circ f_i$} (Z);

 \end{tikzpicture}
\end{center}
\end{figure}
\noindent
If $B \in \sa_Z$ then $g^{-1}(B) \in \sa_Y$ if and only if  $f_i^{-1}(g^{-1}(B))  \in \sa_{X}$.
\end{proof}
\noindent
This result is used frequently when $Y$ in the above diagram is replaced by a tensor product space $X \otimes Y$.  For example,
using this lemma it follows that the projection maps $\pi_Y\colon  X \otimes Y \rightarrow Y$ and $\pi_X\colon  X \otimes Y \rightarrow X$  are both measurable because the diagrams in Figure~\ref{fig:tensorProduct} commute.
   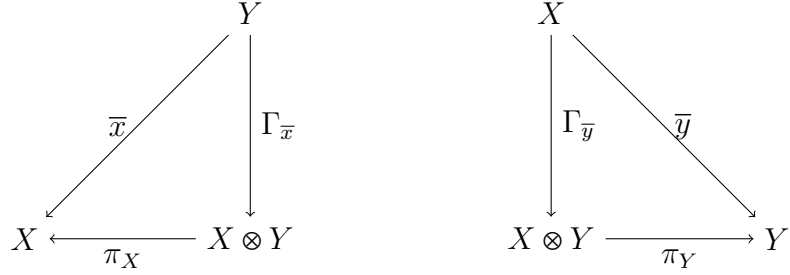
\begin{figure}[H]
  \begin{equation}  \nonumber
 \begin{tikzpicture}[baseline=(current bounding box.center)]
 	\node	(X)	at	(2,0)		         {$X$};
	\node         (Y)   at       (5,-3)               {$Y$};
	\node        (XY)  at      (2,-3)               {$X \otimes Y$};
	
	\draw[->,right] (X) to node {$\overline{y}$} (Y);
	\draw[->,right] (X) to node {$\Gamma_{\overline{y}}$} (XY);
	\draw[->,below] (XY) to node {$\pi_Y$} (Y);

 	\node	(Y2)   	at	(-2,0)		         {$Y$};
	\node         (X2)      at   (-5,-3)              {$X$};
	\node        (XY2)    at      (-2,-3)               {$X \otimes Y$};
	
	\draw[->,left] (Y2) to node {$\overline{x}$} (X2);
	\draw[->,right] (Y2) to node {$\Gamma_{\overline{x}}$} (XY2);
	\draw[->,below] (XY2) to node {$\pi_X$} (X2);

 \end{tikzpicture}
 \end{equation}
 \caption{The commutativity of these diagrams, together with the measurability of the constant functions and constant graph functions, implies the projection maps $\pi_X$ and $\pi_Y$ are measurable.}
 \label{fig:tensorProduct}
 \end{figure} 

By the measurability of the projection maps and the universal property of the product, it follows the identity mapping on the set $X \times Y$ yields a measurable function

   \begin{figure}[H]
  \begin{equation}  \nonumber
 \begin{tikzpicture}[baseline=(current bounding box.center)]
 	\node	(X)	at	(-2,0)		         {$X \otimes Y$};
	\node	(Y)	at	(0,0)	         {$X \times Y$};
	
	\draw[->, above] (X) to node  {$id$} (Y);

 \end{tikzpicture}
 \end{equation}
 \end{figure} 
\noindent
called the \emph{restriction of the $\sigma$-algebra}.
In contrast, the identity function $X \times Y \rightarrow X \otimes Y$ is not necessarily measurable.  Given any probability measure $P$ on $X \otimes Y$ the restriction mapping induces the pushforward probability measure $\delta_{id} \circ P = P( id^{-1}(\cdot))$ on the product $\sigma$-algebra.

\subsection{Graphs of Conditional Probabilities}  The  tensor product of two probability measures $P\colon 1 \rightarrow X$ and $Q\colon 1 \rightarrow Y$ was defined in Equations~\ref{ltensor1} and \ref{rtensor1} as the joint distribution on the product $\sigma$-algebra by either of the expressions
\be  \nonumber
(P \ltensor Q)(\varsigma) = \int_{y \in Y} P(\Gamma_{\overline{y}}^{-1}(\varsigma)) \, dQ \quad \forall \varsigma \in \sa_{X \times Y}
\ee
and
\be  \nonumber
(P \rtensor Q)(\varsigma) = \int_{x \in X} Q(\Gamma_{\overline{x}}^{-1}(\varsigma)) \, dP \quad \forall \varsigma \in \sa_{X \times Y}
\ee
which are equivalent on the product $\sigma$-algebra.  Here we have introduced the new notation of left tensor $\ltensor$ and right tensor $\rtensor$ because we can extend these definitions to be defined on the tensor $\sigma$-algebra though in general the equivalence of these two expressions may no longer hold true.
These definitions can be extended to conditional probability measures $P\colon Z \rightarrow X$ and $Q\colon Z \rightarrow Y$  trivially by conditioning on a point $z \in Z$, 
\be  \label{def1}
(P \ltensor Q)(\varsigma \mid z) = \int_{y \in Y} P(\Gamma_{\overline{y}}^{-1}(\varsigma)) \, dQ_z \quad \forall \varsigma \in \sa_{X \otimes Y}
\ee
and
\be \label{def2}
(P \rtensor Q)(\varsigma \mid z) = \int_{x \in X} Q(\Gamma_{\overline{x}}^{-1}(\varsigma)) \, dP_z \quad \forall \varsigma \in \sa_{X \otimes Y}
\ee
which are equivalent on the product $\sigma$-algebra but not on the tensor $\sigma$-algebra.  However in the special case when $Z=X$ and $P=1_X$, then Equations~\ref{def1} and \ref{def2} do coincide on $\sa_{X \otimes Y}$ because by Equation~\ref{def1}

\be
\begin{array}{lcl}
(1_X \ltensor Q)(\varsigma \mid x) &=& \int_{y \in Y} \underbrace{\delta_{x}(\Gamma_{\overline{y}}^{-1}(\varsigma))}_{= \left\{ \begin{array}{ll} 1 & \textrm{ iff }(x,y) \in \varsigma \\ 0 & \textrm{ otherwise } \end{array} \right.} \, dQ_x   \quad \forall \varsigma \in \sa_{X \otimes Y^X} \\
&=& \int_{y \in Y} \chi_{\Gamma_{\overline{x}}^{-1}(\varsigma)}(y) \, dQ_x \\
&=& Q_x(\Gamma_{\overline{x}}^{-1}(\varsigma)),
\end{array}
\ee
while by Equation~\ref{def2} 
\be  
\begin{array}{lcl}
(1_X \rtensor Q)(\varsigma \mid x) &=&  \int_{u \in X} Q_x(\Gamma_{\overline{u}}^{-1}(\varsigma)) \, d \underbrace{(\delta_{Id_X})_x}_{=\delta_{x}}   \quad \forall \varsigma \in \sa_{X \otimes Y^X} \\
&=& Q_x(\Gamma_{\overline{x}}^{-1}(\U)).
\end{array}
\ee
In this case we denote the common conditional by $\Gamma_Q$, called \emph{the graph of $Q$} by analogy to the graph of a function, and this map gives 
 the commutative diagram in Figure~\ref{fig:graphQ}.
  \begin{figure}[H]
  \begin{equation}   \nonumber
 \begin{tikzpicture}[baseline=(current bounding box.center)]
 	\node	(1)	at	(0,0)		         {$X$};
	\node	(X)	at	(-3,-3)	         {$X$};
	\node         (Y)   at       (3,-3)               {$Y$};
	\node        (XY)  at      (0,-3)               {$X \otimes Y$};
	
	\draw[->, left] (1) to node  {$1_X$} (X);
	\draw[->,right] (1) to node {$Q$} (Y);
	\draw[->,right] (1) to node [yshift=-10pt] {$\Gamma_Q$} (XY);
	\draw[->,below] (XY) to node {$\delta_{\pi_X}$} (X);
	\draw[->,below] (XY) to node {$\delta_{\pi_Y}$} (Y);

 \end{tikzpicture}
 \end{equation}
 \caption{The tensor product of a conditional with an identity map in $\prob$.}
 \label{fig:graphQ}
 \end{figure}
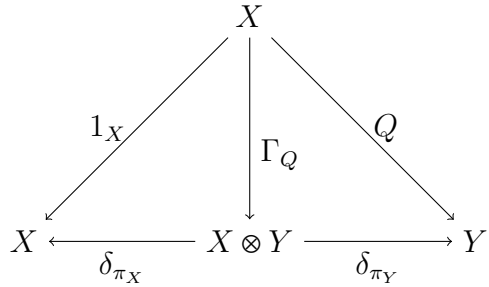 

The commutativity of the diagram in Figure~\ref{fig:graphQ} follows from 
\be
\begin{array}{lcl}
(\delta_{\pi_X} \circ \Gamma_Q)(A\mid x) &=& \int_{(u,v) \in X \otimes Y} \delta_{\pi_X}(A \mid (u,v)) \, d\underbrace{(\Gamma_Q)_x}_{=Q\Gamma_{\overline{x}}^{-1}}  \\
&=&\int_{v \in Y}  \delta_{\pi_X}(A \mid \Gamma_{\overline{x}}(v)) \, dQ_x \\
&=& \int_{v \in Y} \delta_x(A) dQ_x \\
&=& \delta_x(A) \int_{Y} dQ_x \\
&=& 1_X(A \mid x)
\end{array}
\ee
and 
\be
\begin{array}{lcl}
(\delta_{\pi_Y} \circ \Gamma_Q)(B\mid x) &=& \int_{(u,v) \in X \otimes Y} \delta_{\pi_Y}(B \mid (u,v)) \, d((\Gamma_Q)_x) \\
&=&\int_{v \in Y}  \delta_{\pi_Y}(B \mid (x,v)) \, dQ_x \\
&=&\int_{v \in Y}  \chi_{B}(v)  \, dQ_x \\
&=& Q(A \mid x).
\end{array}
\ee

\subsection{A Tensor Product of Conditionals}  \label{sec:tensorP}
Given any conditional $P\colon  Z  \rightarrow Y$ in $\prob$ we can define a tensor product $1_X \otimes P$ by
\be \nonumber
(1_X \otimes P)(\A \mid (x,z)) = P(\Gamma_{\overline{x}}^{-1}(\A) \mid z) \quad \quad \forall  \A \in \sa_{X \otimes Y}
\ee
which makes the diagram in Figure~\ref{fig:tensor} commute and justifies the notation $1_X \otimes P$ (and explains also why the notation  $\Gamma_Q$ for the graph map was used to distinguish it from this map). 
  \begin{figure}[H]
  \begin{equation}   \nonumber
 \begin{tikzpicture}[baseline=(current bounding box.center)]
 	\node	(XZ)	at	(0,0)		         {$X \otimes Z$};
	\node	(X)	at	(-3,0)	         {$X$};
	\node         (Z)   at       (3,0)               {$Z$};
	\node        (XY)  at      (0,-3)               {$X \otimes Y$};
	\node	(X2)	at	(-3,-3)	       {$X$};
	\node         (Y)   at       (3,-3)               {$Y$};
	
	\draw[->, right] (XZ) to node  {$1_{X} \otimes P$} (XY);
	\draw[->,above] (XZ) to node {$\delta_{\pi_X}$} (X);
	\draw[->,above] (XZ) to node {$\delta_{\pi_Z}$} (Z);
	\draw[->,below] (XY) to node {$\delta_{\pi_X}$} (X2);
	\draw[->,below] (XY) to node {$\delta_{\pi_Y}$} (Y);
	\draw[->,right] (X) to node {$1_X$} (X2);
	\draw[->,left] (Z) to node {$P$} (Y);

 \end{tikzpicture}
 \end{equation}
 \caption{The tensor product of conditional $1_X$ and $P$ in $\prob$.}
 \label{fig:tensor}
 \end{figure} 

This tensor product $1_X \otimes P$ essentially 
comes from the diagram
\[
 \begin{tikzpicture}[baseline=(current bounding box.center)]
 	\node	(Z)	at	(0,0)		         {$Z$};
	\node         (Y)   at       (3,0)               {$Y$};
	\node         (XY)  at     (6,0)             {$X \otimes Y$,};
	
	\draw[->, above] (Z) to node  {$P$} (Y);
	\draw[->,above] (Y) to node {$\delta_{\Gamma_{\overline{x}}}$} (XY);

 \end{tikzpicture}
\]
where given a measurable set $\A \in \sa_{X \otimes Y}$ one pulls it back under the constant graph function $\Gamma_{\overline{x}}$ and then applies the conditional $P$ to the pair $(\Gamma_{\overline{x}}^{-1}(\A) \mid z)$.

\subsection{Symmetric Monoidal Categories} A category $\C$ is said to be a  monoidal category  if it possesses the following three properties:
\begin{enumerate}
\item There is a bifunctor 
\be  \nonumber
\begin{array}{lcccc}
\square &\colon & \C \times \C & \rightarrow & \C \\
&\colon _{ob}& (X,Y) & \mapsto & X \square Y \\
&\colon _{ar}& (X,Y) \stackrel{(f,g)}{\longrightarrow} (X',Y') & \mapsto &  X \square Y \stackrel{(f \square g)}{\longrightarrow} X' \square Y'
\end{array}
\ee
which is associative up to isomorphism, 
\be   \nonumber
\square ( \square \times Id_{\C} ) \cong \square (Id_{\C} \times  \square)\colon  \C \times \C \times \C \rightarrow \C
\ee
where $Id_{\C}$ is the identity functor on $\C$.  Hence for every triple $X,Y,Z$ of objects, there is an isomorphism
\be \nonumber
a_{X,Y,Z} \colon  (X \square Y) \square Z \longrightarrow X \square (Y \square Z)
\ee
which is natural in $X,Y,Z$.  This condition is  called  the associativity axiom.
\item There is an object $I \in \C$ such that for every object $X \in_{ob} \C$ there is a left unit isomorphism  
\be \nonumber
l_X\colon  1 \square X \longrightarrow X.
\ee
and a right unit isomorphism 
\be \nonumber
r_X\colon  X \square 1 \longrightarrow X.
\ee
These two conditions are called the unity axioms.
\item For every quadruple of objects $X,Y,W,Z$ the diagram 

  \begin{equation}   \nonumber
 \begin{tikzpicture}[baseline=(current bounding box.center)]
 	\node	(XY)	at	(0,0)		         {$((X \square Y) \square W) \square Z$};
	\node         (YW)   at       (0,-3)               {$(X \square (Y \square W)) \square Z$};
	\node         (YW2)  at     (0,-6)             {$X \square ((Y \square W) \square Z)$};
	\node          (four)  at     (6.5,0)                {$(X \square Y) \square (W \square Z)$};
	\node          (WZ)   at      (6.5,-6)             {$X \square (Y \square (W \square Z))$};
	
	\draw[->, above] (XY) to node  {$a_{X \square Y,W,Z}$} (four);
	\draw[->,above] (YW2) to node {$Id_X \square a_{Y,W,Z}$} (WZ);
	\draw[->,right] (four) to node {$a_{X,Y,W \square Z}$} (WZ);
	\draw[->,left] (XY) to node {$a_{X \square Y,W,Z}$} (YW);
	\draw[->,left] (YW) to node {$a_{X,Y \square W,Z}$}(YW2);

 \end{tikzpicture}
 \end{equation}
commutes. This is called the associativity coherence condition. 
\end{enumerate}

 If $\C$ is a monoidal category under a bifunctor $\square$ and identity $1$ it is denoted \mbox{$(\C,\square,1)$}.  A monoidal category $(\C,\square,1)$ is symmetric if for every pair of objects $X,Y$ there exist an isomorphism 
 \be
 s_{X,Y}\colon  X \square Y \longrightarrow Y \square X 
 \ee
 which is natural in $X$ and $Y$, and the three diagrams in Figure~\ref{fig:symmetric_monoidal} commute.
  \begin{figure}[H]
  \begin{equation}   \nonumber
 \begin{tikzpicture}[baseline=(current bounding box.center)]
 	\node	(XY)	at	(-1,0)		         {$(X \square Y) \square Z$};
	\node         (YW)   at       (-1,-3)               {$X \square (Y \square Z)$};
	\node         (YW2)  at     (-1,-6)                {$Y \square (Z \square X)$};
	\node          (four)  at     (4,0)                {$(Y \square X) \square Z$};
	\node          (five)   at      (4,-3)             {$Y \square (X \square Z)$};
	\node          (six)   at      (4,-6)             {$Y \square (Z \square X)$};

	\draw[->, above] (XY) to node  {$s_{X,Y} \square Id_Z$} (four);
	\draw[->,above] (YW2) to node {$a_{Y,Z,X}$} (six);
	\draw[->,right] (four) to node {$a_{Y,X, Z}$} (five);
	\draw[->,left] (XY) to node {$a_{X ,Y,Z}$} (YW);
	\draw[->,left] (YW) to node {$s_{X,Y \square Z}$}(YW2);
	\draw[->,right] (five) to node {$Id_Y \otimes s_{X,Z}$} (six);
	
	\node (AI) at   (7.5,0)        {$X \square I$};
	\node  (IA) at  (9.5,0)    {$I \square X$};
	\node  (A)  at  (9.5,-2)   {$X$};
	
	\draw[->,above] (AI) to node {$s_{X,1}$} (IA);
	\draw[->,below,left] (AI) to node {$r_{X}$} (A);
	\draw[->,right] (IA) to node {$l_X$} (A);
	
	\node (BI) at   (7.5,-4)        {$X \square Y$};
	\node  (IB) at  (9.5,-4)    {$Y \square X$};
	\node  (B)  at  (9.5,-6)   {$X \square Y$};
	
	\draw[->,above] (BI) to node {$s_{X,Y}$} (IB);
	\draw[->,below,left] (BI) to node {$s_{Y,X}$} (B);
	\draw[->,right] (IB) to node {$Id_X$} (B);

 \end{tikzpicture}
 \end{equation}
\caption{The additional conditions required for a symmetric monoidal category.}
  \label{fig:symmetric_monoidal}
 \end{figure}

The  main example of a symmetric monoidal category is the category of sets, $Set$, under the cartesian product with identity the terminal object $1=\{\star\}$.  
Similarly, for the categories $\M$ and $\prob$, the tensor product $\otimes$ along with the terminal object $1$ acting as the identity element make both $(\M,\otimes,1)$ and $(\prob,\otimes,1)$  symmetric monoidal categories with the above conditions straightforward to verify. This provides a good exercise for the reader new to categorical methods.

\section{Function Spaces}    \label{sec:functions}

For $X,Y \in_{ob} \M$ let $Y^X$ denote the set of all measurable functions from $X$ to $Y$ endowed with the $\sigma$-algebra induced by the set of all point evaluation maps $\{ ev_x\}_{x \in X}$, where 
\be \nonumber
\begin{array}{ccc}
Y^X &\stackrel{ev_x}{\longrightarrow}& Y \\
f & \mapsto & f(x). 
\end{array}
\ee
Explicitly, the $\sigma$-algebra on $Y^{X}$ is given by 
\be
	\sa_{Y^{X}} = \sigma\left( \bigcup_{x\in X} ev_{x}^{-1}\sa_{Y}\right),
\ee
where for any function $f\colon W \to Z$ we have 
\be
	f^{-1}\sa_{Z} = \{ B \in 2^{W} \mid \exists C \in \sa_{Z} \text{ with } f^{-1}(C) = B\}
\ee
and $\sigma(\mathcal B)$ denotes the $\sigma$-algebra generated by any collection $\mathcal B$ of subsets.

Formally we should use an alternative  notation such as  $\fn$  to distinguish between the measurable function \mbox{$f\colon X \rightarrow Y$} and the point \mbox{$\fn\colon 1 \rightarrow Y^X$} of the function space $Y^X$.\footnote{Having defined $Y^X$ to be the set of all measurable functions $f\colon X \rightarrow Y$ it seems contradictory to then define $ev_x$ as acting on ``points'' $\fn\colon  1 \rightarrow Y^X$ rather than the functions $f$ themselves! The apparent self contradictory definition arises because we are interspersing categorical language with set theory; when defining a set function, like $ev_x$, it is implied that it acts on points which are defined as ``global elements'' $1 \rightarrow Y^X$.  A global element is a map with domain $1$. This is the categorical way of defining points rather than using the \emph{elementhood} operator ``$\in$''.    Thus, to be more formal, we could have defined $ev_x$, where $x\colon 1 \rightarrow X$ is any global element, by $ev_x \circ \ulcorner f \urcorner = \ulcorner f(x) \urcorner\colon  1 \rightarrow Y$, where $f(x) = f \circ x$.}
However, it is common practice to let the context define which arrow we are referring to and we shall often follow this practice unless the distinction is critical to avoid ambiguity or awkward expressions.

An alternative notation to $Y^X$ is $\prod_{x \in X} Y_x$ where each $Y_x$ is a copy of $Y$.  The relationship between these representations is that in the former we view the elements as functions $f$ while in the latter we view the elements as the indexed images of a function, $\{f(x)\}_{x \in X}$.  Either representation determines the other since a function is uniquely specified by its values.  

   Because the $\sigma$-algebra structure on tensor product spaces was defined precisely  so that the constant graph functions were all measurable, it follows that in particular the constant graph functions $\Gamma_{\overline{f}}\colon X \rightarrow X \otimes Y^X$ sending $x \mapsto (x,f)$ are measurable. (The graph function symbol $\Gamma_{\cdot}$ is overloaded and will need to be specified directly (domain and codomain) when the context is not clear.)

Define the evaluation function
 \be \label{evDef}
 \begin{array}{ccc}
 X \otimes Y^X & \stackrel{ev_{X,Y}}{\longrightarrow}& Y \\
 (x,f) & \mapsto & f(x)
 \end{array}
 \ee
and observe that for every $\ulcorner f \urcorner  \in Y^X$ the right hand $\M$ diagram in Figure~\ref{fig:SMCC} is commutative as a set mapping, $f = ev_{X,Y} \circ \Gamma_{\overline{f}}$.  

 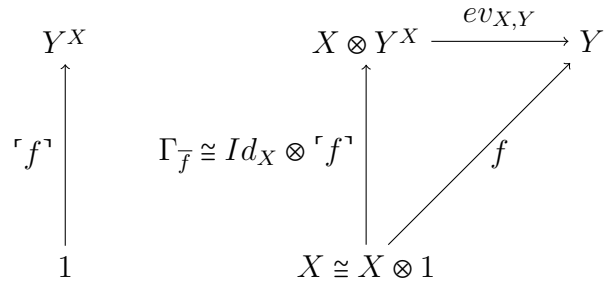
\begin{figure}[H]
 \begin{center}
 \begin{tikzpicture}[baseline=(current bounding box.center)]
 	\node	(X)	at	(0,-3)              {$X \cong X \otimes 1$};
	\node	(XY)	at	(0,0)	               {$ X \otimes Y^X$};
	\node	(Y)	at	(3,0)               {$Y$};
         \node         (1)    at      (-4,-3)             {$1$};
         \node         (YX) at      (-4,0)              {$Y^X$};

	\draw[->, left] (X) to node  {$\Gamma_{\overline{f}}\cong  Id_X \otimes \ulcorner f \urcorner$} (XY);
	\draw[->,below, right] (X) to node [xshift=0pt,yshift=0pt] {$f$} (Y);
	\draw[->,above] (XY) to node {$ev_{X,Y}$} (Y);
	\draw[->,left] (1) to node {$\ulcorner f \urcorner$} (YX);

 \end{tikzpicture}
\end{center}
\caption{The defining characteristic property of the evaluation function $ev$ for graphs. }
\label{fig:SMCC}
\end{figure}

By rotating the diagram in Figure~\ref{fig:SMCC} and also considering the constant graph functions $\Gamma_{\overline{x}}$,  the right hand side of the  diagram in Figure~\ref{fig:SMCC2} also commutes for every $x \in X$.

 \begin{figure}[H]
\begin{center}
 \begin{tikzpicture}[baseline=(current bounding box.center)]
        \node          (X)    at      (-3,0)             {$X$};
 	\node	(YX)	at	(3,0)              {$Y^X$};
	\node	(XY)	at	(0,0)	               {$ X \otimes Y^X$};
	\node	(Y)	at	(0,-3)               {$Y$};

         \draw[->,above] (X) to node {$\Gamma_{\overline{f}}$} (XY);
         \draw[->,left] (X) to node [xshift=-3pt] {$f$} (Y);
	\draw[->, above] (YX) to node  {$\Gamma_{\overline{x}}$} (XY);
	\draw[->,right] (YX) to node [xshift=5pt,yshift=0pt] {$ev_x$} (Y);
	\draw[->,right] (XY) to node {$ev_{X,Y}$} (Y);

 \end{tikzpicture}
\end{center}
\caption{The commutativity of both triangles, the measurability of $f$ and $ev_x$, and the induced $\sigma$-algebra of $X \otimes Y^X$ implies the measurability of $ev$. }
\label{fig:SMCC2}
\end{figure}
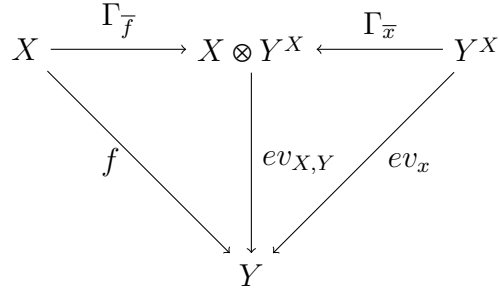

Since $f$ and $\Gamma_{\overline{f}}$ are measurable, as are $ev_{x}$ and $\Gamma_{\overline{x}}$,  it follows by Lemma~\ref{coinduced}  that $ev_{X,Y}$ is measurable since the constant graph functions generate the $\sigma$-algebra of $X \otimes Y^X$.  More generally, given any measurable function $f\colon X \otimes Z \rightarrow Y$ there exists a unique measurable map $\tilde{f} \colon  Z  \rightarrow  Y^X$  defined by $\tilde{f}(z) = \ulcorner f(\cdot,z) \urcorner\colon  1 \rightarrow Y^X$ where $f(\cdot,z)\colon  X \rightarrow Y$ sends $x \mapsto f(x,z)$.
This map $\tilde{f}$ is measurable because the $\sigma$-algebra is generated by the \emph{point evalutation} maps $ev_x$ and the diagram

\[
 \begin{tikzpicture}[baseline=(current bounding box.center)]
        \node          (XZ)    at      (3,-3)             {$X \otimes Z$};
 	\node	(YX)	at	(0,0)              {$Y^X$};
	\node	(Y)	at	(3,0)               {$Y$};
	\node	(Z)	at	(0,-3)	               {$Z$};

         \draw[->,above] (YX) to node {$ev_x$} (Y);
         \draw[->,left] (Z) to node [xshift=-3pt] {$\tilde{f}$} (YX);
	\draw[->, above] (Z) to node  {$\Gamma_{\overline{x}}$} (XZ);
	\draw[->,right] (XZ) to node [xshift=5pt,yshift=0pt] {$f$} (Y);
	\draw[->,right,dashed] (Z) to node {$$} (Y);

 \end{tikzpicture}
\]
commutes so that $\tilde{f}^{-1}( ev_x^{-1}(B)) = (f \circ \Gamma_{\overline{x}})^{-1}(B) \in \sa_Z$.

Conversely given any measurable map $g  \colon  Z \rightarrow Y^X$, it follows the composite 
\be \nonumber
ev_{X,Y} \circ (Id_X \otimes g)
\ee
is a measurable map. This sets up a bijective correspondence between measurable functions denoted by
\[
 \begin{tikzpicture}[baseline=(current bounding box.center)]
         \node  (Z)  at (-.2,0)    {$Z$};
         \node  (YX)  at  (2.7,0)   {$Y^X$};
         \node (X) at   (-.2,-1)  {$X \otimes Z$};
         \node (Y) at  (2.7,-1)  {$Y$};
         \node  (pt1)    at (-1,-.5)   {$$};
         \node  (pt2)    at (3.5,-.5)   {$$};
         \node  (pt3)    at (4.1,.25)   {$$};
         \node  (pt4)    at (4.1,-1.5)   {$$};

	\draw[->,above] (Z) to node  {$\tilde{f}$} (YX);
	\draw[->,below] (X) to node {$f$} (Y);
	
	\draw[-] (pt1) to node {$$} (pt2);
	
	\pgfsetlinewidth{.5ex}
	\draw[<->] (pt3) to node {$$} (pt4);

	 \end{tikzpicture}
\]
or the diagram in Figure~\ref{fig:SMCCDiagram}.

 \begin{figure}[H]
\begin{center}
 \begin{tikzpicture}[baseline=(current bounding box.center)]
 	\node	(X)	at	(0,-3)              {$X \otimes Z$};
	\node	(XY)	at	(0,0)	               {$ X \otimes Y^X$};
	\node	(Y)	at	(3,0)               {$Y$};
         \node         (1)    at      (-4,-3)             {$Z$};
         \node         (YX) at      (-4,0)              {$Y^X$};

	\draw[->, left] (X) to node  {$Id_X \otimes \tilde{f}$} (XY);
	\draw[->,below, right] (X) to node [xshift=0pt,yshift=0pt] {$f$} (Y);
	\draw[->,above] (XY) to node {$ev_{X,Y}$} (Y);
	\draw[->,left] (1) to node {$\tilde{f}$} (YX);

 \end{tikzpicture}
\end{center}
\caption{The evaluation function $ev$ sets up a bijective correspondence between the two measurable maps $f$ and $\tilde{f}$. }
\label{fig:SMCCDiagram}
\end{figure}

The measurable map $\tilde{f}$ is called  the adjunct of $f$ and vice versa, so that $\tilde{\tilde{f}} = f$.  Whether we use the tilde notation for the map $X \otimes Z \rightarrow Y$ or the map $Z \rightarrow Y^X$ is irrelevant, it simply indicates it's the map uniquely determined by the other map.

The map $ev_{X,Y}$, which we will usually abbreviate to simply $ev$ with the pair $(X,Y)$ obvious from context, is called a universal arrow because of this property; it mediates the relationship between the two maps $f$ and $\tilde{f}$.  In the language of category theory
using functors, for a fixed object $X$ in $\M$, the collection of maps $\{ ev_{X,Y}\}_{Y \in_{ob} \M}$ form the components of a natural transformation $ev_{X,-}\colon  (X \otimes \cdot) \circ \_^X \rightarrow Id_{\M}$.
In this situation we say the pair of functors $\{X \otimes \_, \_^X\}$ forms an adjunction denoted $X \otimes \_ \dashv \_^X$.
This adjunction $X \otimes \_ \dashv \_^X$ is the defining property of a  closed category.  We previously showed $\M$ was symmetric monoidal
and combined with the closed category structure we conclude that $\M$ is a  symmetric monoidal closed category (SMCC).  
Subsequently we will show that $\prob$ satisfies a weak version of SMCC, where uniqueness cannot be obtained.

\paragraph{The Graph Map.}
Given the importance of graph functions when working with tensor spaces we define the graph map
\be \nonumber
\begin{array}{ccccc}
\Gamma_{\cdot} &\colon & Y^X  &\rightarrow& (X \otimes Y)^X \\
&\colon & \ulcorner f \urcorner & \mapsto & \ulcorner \Gamma_f \urcorner.
\end{array}
\ee
Thus $\Gamma_{\cdot}(\ulcorner f \urcorner) = \ulcorner \Gamma_f \urcorner$ gives the name of the graph 
 \begin{figure}[H]
\begin{center}
 \begin{tikzpicture}[baseline=(current bounding box.center)]
 	\node	(X)	at	(0,0)		         {$X$};
	\node	(XY)	at	(2,0)	                  {$ X \otimes Y$.};

	\draw[->, above] (X) to node  {$\Gamma_{f}$} (XY);

 \end{tikzpicture}
\end{center}
\end{figure}

The measurability of $\Gamma_{\cdot}$ follows in part from the commutativity of the  diagram in Figure~\ref{fig:graphMap}, where the map $\hat{ev}_{x}\colon  (X \otimes Y)^X \rightarrow X \otimes Y$ denotes the standard point evaluation map sending $g  \mapsto (x,g(x))$.  

 \begin{figure}[H]
\begin{center}
 \begin{tikzpicture}[baseline=(current bounding box.center)]
 	\node	(YX)	at	(-2,0)		  {$Y^X$};
	\node	(XYX) at	(2,0)	           {$(X \otimes Y)^X$};
	\node         (XY)     at  (2,-2)          {$X \otimes Y$};
	\node         (Y)     at  (-2,-2)            {$Y$};

	\draw[->, above] (YX) to node  {$\Gamma_{\cdot}$} (XYX);
	\draw[->,right] (XYX) to node {$\hat{ev}_{x}$} (XY);
	\draw[->,below,left,dashed] (YX) to node [xshift=-3pt] {$\langle \overline{x}, ev_{x} \rangle$} (XY);
	\draw[->, left] (YX) to node {$ev_x$} (Y);
	\draw[->,below] (Y) to node {$\Gamma_{\overline{x}}$} (XY);

 \end{tikzpicture}
\end{center}
\caption{The relationship between  the graph map, point  evaluations, and constant graph maps. }
\label{fig:graphMap}
\end{figure}

We have used the notation $\hat{ev}_{x}$ simply to distinguish this map from the map $ev_{x}$ which has a different domain and codomain.  The $\sigma$-algebra of $(X \otimes Y)^X$ is determined by these point evaluation maps $\hat{ev}_{x}$ so that they are measurable.
The maps $ev_x$ and $\Gamma_{\overline{x}}$ are both measurable and hence their composite $\Gamma_{\overline{x}} \circ ev_x = \langle \overline{x}, ev_x \rangle$ is also measurable.

To prove the measurability of the graph map we use the dual to Lemma~\ref{coinduced} obtained 
 by reversing all the arrows in that lemma to give
\begin{lemma} \label{induced} Let the $\sigma$-algebra of $Y$ be induced by a collection of maps $\{g_i \colon  Y \rightarrow Z_{i} \}_{i \in I}$.   Then any map $f\colon X \rightarrow Y$ is measurable if and only if the composition $g_i \circ f$ is measurable for each $i \in I$.
\end{lemma}
\begin{proof} Consider the diagram
 \begin{figure}[H]
\begin{center}
 \begin{tikzpicture}[baseline=(current bounding box.center)]
 	\node	(X)	at	(0,0)              {$X$};
	\node	(Y)	at	(2,0)	               {$Y$};
	\node	(Z)	at	(2,-2)               {$Z_{i}$};

	\draw[->, above] (X) to node  {$f$} (Y);
	\draw[->,right] (Y) to node [xshift=0pt,yshift=0pt] {$g_i$} (Z);
	\draw[->,left] (X) to node [xshift = -5pt] {$g_i \circ f$} (Z);

 \end{tikzpicture}
\end{center}
\end{figure}
\noindent The necessary condition is obvious. Conversely if $g_i \circ f$ is measurable for each $i \in I$ then $f^{-1}(g_i^{-1}(B)) \in \sa_X$.
Because the $\sigma$-algebra $\sa_Y$ is generated by the measurable sets $g_i^{-1}(B)$ it follows that every measurable $U \in \sa_Y$ also satisfies $f^{-1}(U) \in \sa_X$ so $f$ is measurable.  
\end{proof}
\noindent
Applying this lemma to the diagram in Figure~\ref{fig:graphMap} with the maps $g_i$ corresponding to the point evaluation maps $ev_x$ and the map $f$ being the graph map $\Gamma_{\cdot}$ proves the graph map  is indeed measurable.

 The measurability of both of the maps $ev$ and $\Gamma_{\cdot}$ yield corresponding $\prob$ maps $\delta_{ev}$ and $\delta_{\Gamma_{\cdot}}$ that play a role in the construction of sampling distributions defined on any hypothesis spaces that involves function spaces.

\subsection{Stochastic Processes}   \label{sec:Stoch}

Having defined function spaces $Y^X$, we are now in a position to define stochastic processes using categorical language.  The elementary definition given next suffices to develop all the basic concepts one usually associates with traditional ML and allows for relatively elegant proofs.   Subsequently, using the language of functors,  a more general definition will be given and for which the following definition can be viewed as a special instance.

\begin{definition} \label{stochasticDef} A \textbf{stochastic process} is a $\prob$ map 
 \begin{figure}[H]
\begin{center}
 \begin{tikzpicture}[baseline=(current bounding box.center)]
 	\node	(1)	at	(0,0)		         {$1$};
	\node	(YX)	at	(3,0)	                  {$Y^X$};
	\draw[->, above] (1) to node  {$P$} (YX);

 \end{tikzpicture}
\end{center}
\end{figure}
\noindent
representing a probability measure  on the function space $Y^X$.   A \emph{parameterized}  stochastic process is a $\prob$ map
 \begin{figure}[H]
\begin{center}
 \begin{tikzpicture}[baseline=(current bounding box.center)]
 	\node	(Z)	at	(0,0)		         {$Z$};
	\node	(YX)	at	(3,0)	                  {$Y^X$};
	\draw[->, above] (Z) to node  {$P$} (YX);

 \end{tikzpicture}
\end{center}
\end{figure}
\noindent
representing a family of stochastic processes parameterized by $Z$.
\end{definition}

Just as we did for the category $\M$, we seek a bijective correspondence between two $\prob$ maps, a stochastic process $P$ and a corresponding conditional probability measure $\overline{P}$. In the $\prob$ case, however, the two morphisms do not uniquely determine each other, and we are only able to obtain a symmetric monoidal {\em weakly} closed category (SMwCC).

In Section~\ref{sec:tensorP}  the tensor product  $1_X \otimes P$  was defined, and by replacing the space ``$Y$'' in that definition to be a function space $Y^X$  we obtain the tensor product map
\be \nonumber
1_X \otimes P\colon  X\otimes Z  \rightarrow X \otimes Y^X
\ee
given by (using the same formula as in Section~\ref{sec:tensorP})
 \be  \nonumber
(1_X \otimes P)(\U \mid  (x,z)) =  P( \Gamma_{\overline{x}}^{-1}(\U) \mid  z)
 \ee

For a given parameterized stochastic process $P\colon  Z \rightarrow Y^X$ we obtain the tensor product $1_X \otimes P$, and composing this map with the deterministic $\prob$ map determined by the evaluation map we obtain the composite $\overline{P}$ in the diagram  in Figure~\ref{fig:SMCCProb}.
 \begin{figure}[H]
\begin{center}
 \begin{tikzpicture}[baseline=(current bounding box.center)]
 	\node	(X)	at	(0,-3)              {$X \otimes Z$};
	\node	(XY)	at	(0,0)	               {$ X \otimes Y^X$};
	\node	(Y)	at	(3,0)               {$Y$};
         \node         (1)    at      (-4,-3)             {$Z$};
         \node         (YX) at      (-4,0)              {$Y^X$};
         
	\draw[->, left] (X) to node  {$1_X \otimes P$} (XY);
	\draw[->,below, right] (X) to node [xshift=0pt,yshift=0pt] {$\overline{P}$} (Y);
	\draw[->,above] (XY) to node {$\delta_{ev}$} (Y);
	\draw[->,left] (1) to node {$P$} (YX);

 \end{tikzpicture}
\end{center}
\caption{The defining characteristic property of the evaluation function $ev$ for tensor products of conditionals  in $\prob$. }
\label{fig:SMCCProb}
\end{figure}
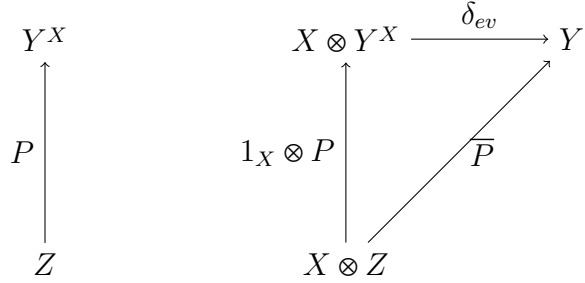

Thus
\be \nonumber
\begin{array}{lcl}
\overline{P}(B \mid (x,z)) &=& \int_{(u,f)\in X \otimes Y^X} {(\delta_{ev})}_B(u,f) \, d(1_X \otimes P)_{(x,z)} \\
&=& \int_{f \in Y^X} \delta_{ev}(B \mid  \Gamma_{\overline{x}}(f)) \, dP_z \\
&=&  \int_{f \in Y^X} \chi_B(ev_{x}(f))  \, dP_z \\
&=& P( ev_{x}^{-1}(B)\mid z)
\end{array}
\ee
and every parameterized stochastic process determines a conditional probability
\[
\overline{P}\colon X \otimes Z \rightarrow Y.
\] 

Conversely, given a conditional probability $\overline P \colon X \otimes Z \to Y$, we wish to define a parameterized stochastic process $P\colon Z \to Y^X$. We might be tempted to define such a stochastic process by letting
\begin{equation}
\label{eq:closed constraint}
	P(ev_x^{-1}(B)\mid z)  = \overline P(B \mid (x,z)),
\end{equation}
but this does not give a well-defined measure for each $z \in Z$. Recall that a probability measure cannot be unambiguously defined on an arbitrary generating set for the $\sigma$-algebra. We can, however, uniquely define a measure on a $\pi$-system\footnote{A $\pi$-system on $X$ is a nonempty collection of subsets of $X$ that is closed under finite intersections.} and then use Dynkin's $\pi$-$\lambda$ theorem to extend to the entire $\sigma$-algebra ({\em e.g.}, see~\cite{Durrett:probability}). This construction requires the following definition.

\begin{definition}
\label{def::unseparated}
Given a measurable space $(X, \sa_X)$, we can define an equivalence relation on $X$ where $x \sim y$ if $x \in A \Leftrightarrow y \in A$ for all $A \in \sa_X$. We call an equivalence class of this relation an {\em atom} of $X$. For an arbitrary set $A \subset X$, we say that $A$ is 
\begin{enumerate}
	\item[$\bullet$] {\em separated} if for any two points $x,y \in A$, there is some $B \in \sa_X$ with $x \in B$ and $y \notin B$
	\item[$\bullet$] {\em unseparated} if $A$ is contained in some atom of $X$.
\end{enumerate}
\end{definition}

This notion of separation of points is important for finding a generating set on which we can define a parameterized stochastic process. The key lemma which we state here without proof\footnote{This lemma and additional work on symmetric monoidal weakly closed structures on $\prob$ will appear in a future paper.} is the following.
\begin{lemma} The class of subsets of $Y^X$
\label{lem::pi-system}
\[
	\mathcal{E} = \emptyset \cup
	\left\{\bigcap_{i = 1}^n ev^{-1}_{x_i}(A_i)
	\quad\middle\mid \quad
	\begin{matrix}
	\{x_i\}_{i=1}^n \text{ is separated in } X,\\
	A_i \in \sa_Y \text{ is nonempty and proper}
	\end{matrix}
	\right\}
\]
 is a $\pi$-system which generates the evaluation $\sigma$-algebra on $Y^X$.
\end{lemma}
We can now define many parameterized stochastic processes ``adjoint'' to $\overline{P}$, with the only requirement being that Equation~\ref{eq:closed constraint} is satisfied. This is not a deficiency in $\prob$, however, but rather shows that we have ample flexibility in this category. 

\begin{remark}
Even when such an expression does provide a well-defined measure as in the case of finite spaces, it does not yield a unique $P$.  Appendix B provides an elementary example  illustrating the failure of the bijective correspondence property in this case.   Also observe that the proposed defining Equation~\ref{eq:closed constraint} can be extended to 
\[
P( \cap_{i=1}^n ev_{x_i}^{-1}(B_i)\mid z)  = \prod_{i=1}^n  \overline P(B_i \mid (x_i,z))
\]
which does provide a well-defined measure by Lemma~\ref{lem::pi-system}.  However it still does not provide a bijective correspondence which is clear as the right hand side implies an independence condition which a stochastic process need not satisfy.  However it does provide for a bijective correspondence if we \emph{impose} an additional  independence condition/assumption.  Alternatively, by imposing the additional condition that for each $z \in Z$,  $P_z$ is a  Gaussian Processes we can obtain a bijective correspondence. In Section~\ref{sec:jointNormal} we illustrate in detail how a joint normal distribution on a finite dimensional space gives rise to a stochastic process, and in particular a GP.
\end{remark}

Often, we are able to exploit the weak correspondence and use the conditional probability $\overline{P}\colon  X \rightarrow Y$ rather than the stochastic process $P\colon 1 \to Y^X$. While carrying less information, the conditional probability is easier to reason with because of our familiarity with Bayes' rule (which uses conditional probabilities) and our unfamiliarity with measures on function spaces.
 
Intuitively it is easier to work with the conditional probability $\overline{P}$ as we can represent the graph of such functions.  In Figure~\ref{fig:GPDiagram} the top diagram shows a prior probability $P\colon 1 \rightarrow \mathbb{R}^{[0,10]}$, which is a stochastic process, depicted by representing its adjunct illustrating its expected value as well as its $2 \sigma$ error bars on each coordinate.  The bottom diagram in the same figure illustrates a parameterized stochastic process where the parameterization is over four measurements.  Using the above notation, $Z = \prod_{i=1}^4 (X \times Y)_i$ and  $\overline{P}(\cdot \mid  \{(\xv_i,y_i)\}_{i=1}^4)$ is a posterior probability measure given four measurements $\{\xv_i,y_i\}_{i=1}^4$.  These diagrams were generated under the hypothesis that the process is a GP.
 
\begin{figure}[H]
\begin{center}
\includegraphics[width=6in,height=4in]{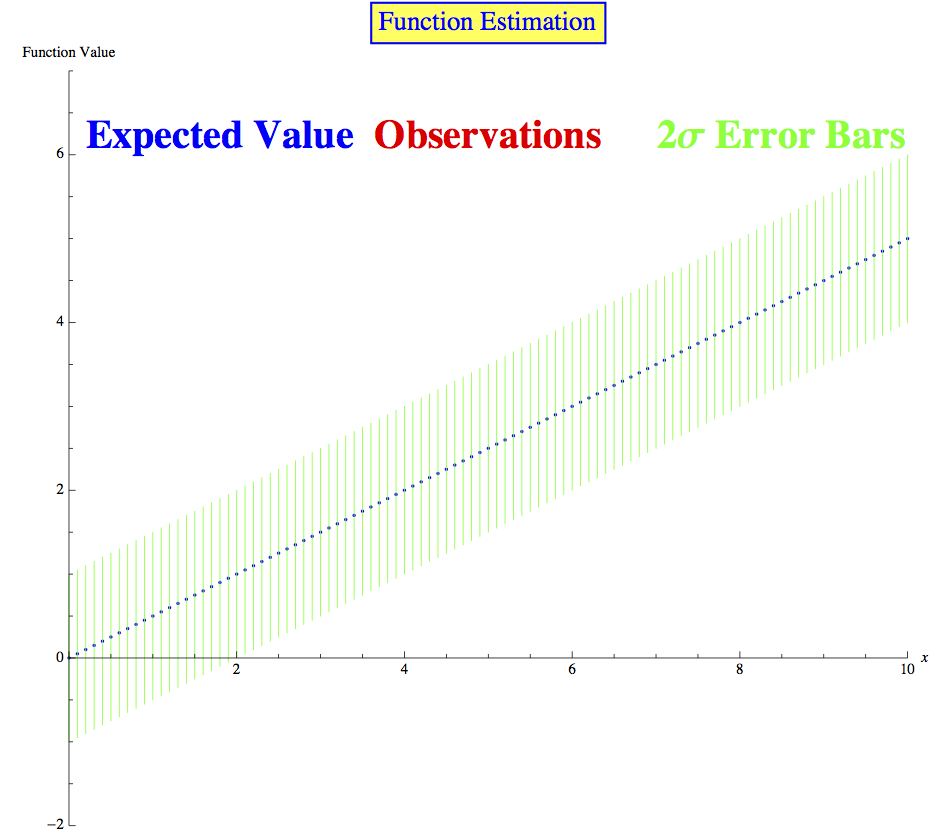}
\includegraphics[width=6in,height=4in]{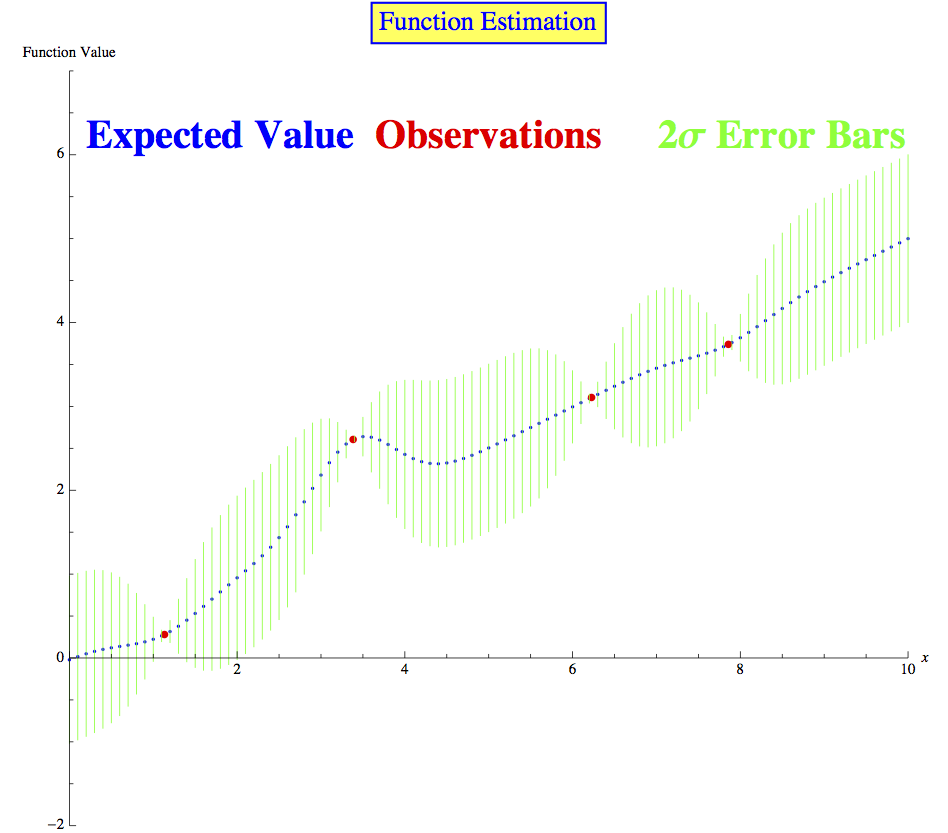}

\caption{\textsl{The top diagram shows a (prior) stochastic process represented by its adjunct \mbox{$\overline{P}\colon  [0,10] \rightarrow \mathbb{R}$} and characterized by its expected value and covariance.  The bottom diagram shows a parameterized stochastic process (the same process), also expressed by its adjunct, where the parameterization is over four measurements.}}
\end{center}
\label{fig:GPDiagram}
\end{figure}

\subsection{Gaussian Processes}  \label{sec:GP}
  To further explicate the use of stochastic processes we consider the special case of a stochastic process that has proven to be of extensive use for modeling in ML problems.  To be able to compute integrals, notably expectations,  we will assume hereafter that 
$Y = \mathbb{R}$ and $X=\mathbb{R}^n$ for some integer $n$, or a compact subset thereof with the standard Borel $\sigma$-algebras.  We use the bold notation $\xv$ to denote a vector in $X$. Because ML applications often simply stress scalar valued functions, we  have take $Y=\mathbb{R}$ and write elements in $Y$ as $y$.  At any rate, the generalization to an arbitrary Euclidean space amounts to carrying around vector notation and using vector valued integrals in the following.

For any finite subset $X_0 \subset X$ the set $X_0$ can be given the subspace $\sigma$-algebra which is the induced  $\sigma$-algebra of the inclusion map $\iota\colon  X_0 \hookrightarrow X$. Given any measurable $f\colon X \rightarrow Y$ the restriction of $f$ to $X_0$ is $f|_{X_0} =f \circ \iota$ and ``substitution'' of an element $\xv \in X$ into $f|_{X_0}$ is precomposition by the point $\xv\colon 1 \rightarrow X$ giving the commutative  $\M$  diagram in Figure~\ref{fig:inclusion}, where the  composite $f|_{X_0}(\xv) = f \circ \iota \circ \xv$ is equivalent to the map $ev_{\xv}(\fn)\colon  1 \rightarrow Y$. 
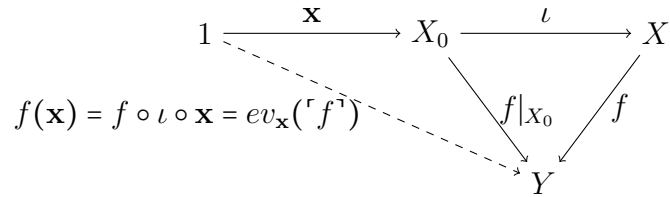
\begin{figure}[H]
 \begin{equation}  \nonumber
 \begin{tikzpicture}[baseline=(current bounding box.center)]
         
         \node         (1)          at     (3,0)      {$1$};
         \node         (X02)    at      (6,0)         {$X_0$};
	\node	(X2)	at	(9,0)	      {$X$};
	\node	(Y2)	at	(7.5,-2)               {$Y$};	
	\draw[->,above] (X02) to node {$\iota$} (X2);
	\draw[->,above] (1) to node {$\xv$} (X02);
	\draw[->, right] (X2) to node {$f$} (Y2);
         \draw[->,left, dashed ] (1) to  node [yshift = -3pt] {$f(\xv) = f \circ \iota \circ \xv = ev_{\xv}(\fn)$} (Y2);
         \draw[->,right] (X02) to node {$f|_{X_0}$} (Y2);

 \end{tikzpicture}
 \end{equation}
 \caption{The substitution/evaluation relation.}
 \label{fig:inclusion}
 \end{figure}

 Thus the inclusion map $\iota$ induces a measurable map 
 \be  \nonumber
 \begin{array}{lclcl}
 Y^{\iota} &\colon & Y^X & \rightarrow & Y^{X_0} \\
 &\colon & \fn & \mapsto & \ulcorner f \circ \iota \urcorner,
 \end{array}
 \ee
 which in turn induces the deterministic map $\delta_{Y^{\iota}}\colon  Y^X \rightarrow Y^{X_0}$ in $\prob$.   For any probability measure $P$ on the function space $Y^X$, we have the composite of $\prob$ arrows shown in the left diagram of Figure~\ref{fig:GPdef}.  For a singleton set $X_0=\{\xv\}$  this diagram reduces to the diagram on the right in Figure~\ref{fig:GPdef}.
 
 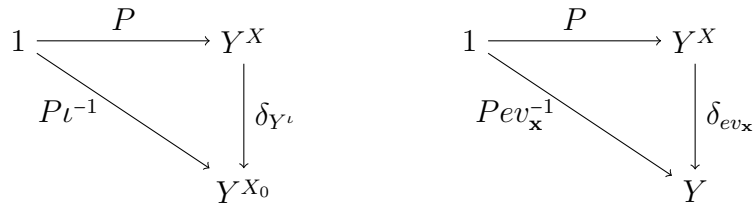
\begin{figure}[H]
  \begin{equation} \nonumber
 \begin{tikzpicture}[baseline=(current bounding box.center)]
         \node         (1)    at      (-1,0)         {$1$};
	\node	(YX)	at	(2,0)	      {$Y^X$};
	\node	(YX0)	at	(2,-2)               {$Y^{X_0}$};
	\draw[->,above] (1) to node {$P$} (YX);
	\draw[->, right] (YX) to node {$\delta_{Y^{\iota}}$} (YX0);
         \draw[->,left ] (1) to  node [xshift = -5pt]{$P \iota^{-1}$} (YX0);
         
         \node         (12)    at      (5,0)         {$1$};
	\node	(YX2)	at	(8,0)	      {$Y^X$};
	\node	(YX02)	at	(8,-2)               {$Y$};
	\draw[->,above] (12) to node {$P$} (YX2);
	\draw[->, right] (YX2) to node {$\delta_{ev_{\xv}}$} (YX02);
         \draw[->,left ] (12) to  node [xshift = -5pt]{$P ev_{\xv}^{-1}$} (YX02);
 \end{tikzpicture}
 \end{equation}
 \caption{The defining property of a Gaussian Process is the commutativity of a $\prob$ diagram.}
 \label{fig:GPdef}
 \end{figure}

Given $m \in Y^X$ and $k$ a bivariate function $k\colon X \times X \rightarrow \mathbb{R}$,
let $m|_{X_0} = m \circ \iota \in Y^{X_0}$ denote the restriction of $m$ to $X_0$ and similiarly let $k|_{X_0} = k \circ (\iota \times \iota)$ denote the restriction of $k$ to $X_0 \times X_0$.

 \begin{definition} A \textbf{Gaussian process} on $Y^X$ is a probability measure $P$ on the function space $Y^X$, denoted $P \sim \GP(m,k)$, such that \emph{for all finite subsets} $X_0$ of $X$  the push forward probability measure $P \iota^{-1}$ is a (multivariate) Gaussian distribution denoted \mbox{$P \iota^{-1} \sim \NN(m|_{X_0},k|_{X_0})$}.  
 \end{definition}
A bivariate function $k$ satisfying the condition in the definition is called the \emph{covariance function} of the Gaussian process $P$ while the function $m$ is the \emph{expected value}.   A Gaussian process is completely specified by its mean and covariance functions.  These two functions are defined \emph{pointwise} by
\be  \label{mean}
m(\xv) \triangleq \E_P[ev_{\xv}] = \int_{f \in Y^X} (ev_{\xv})( \fn) \, dP = \int_{f \in Y^X} f(\xv) \, dP
\ee
and by the vector valued integral
\be   \label{covariance}
\begin{array}{lcl}
k(\xv,\xv') & \triangleq &  \E_{P}[ (ev_{\xv} - \E_P[ev_{\xv}])  (ev_{\xv'} - \E_P[ev_{\xv'}]) ] \\
&=& \displaystyle{ \int_{f \in Y^X}} \left( f(\xv) - m(\xv) \right)^T \left( f(\xv') - m(\xv') \right)  \, dP.
\end{array}
\ee

Abstractly, if $P$ is given, then we could determine $m$ and $k$ by these two equations.  However in practice it is the two functions, $m$ and $k$  which are used to specify a GP $P$ rather than $P$ determining $m$ and $k$.  For general stochastic processes higher order moments $\E_{P}[ev_{\xv}^j]$, with $j>1$,   are necessary to characterize the process.  

For the covariance function $k$ we make the following assumptions for all $\xv,\zv \in X$,

\begin{enumerate}
\item  $k(\xv,\zv) \ge 0$,
\item $k(\xv,\zv) = k(\zv,\xv)$, and
\item $k(\xv,\xv) k(\zv,\zv) - k(\xv,\zv)^2 \ge 0$.
\end{enumerate}

\subsection{GPs via Joint Normal Distributions\footnote{This section is not required for an understanding of  subsequent material but only provided for purposes of  linking familiar concepts and ideas with the less familiar categorical perspective.}}  \label{sec:jointNormal}

A simple illustration of a GP as a probability measure on a function space can be given by consideration of a joint normal distribution.
Here we relate the familiar presentation of multivariate normal distributions as expressed in the language of random variables into the categorical framework and language, and illustrate that the resulting conditional distributions correspond to a GP.

Let  $\X$ and $\Y$ represent  two vector valued real random variables having a joint normal distribution

\be \nonumber
J  = \left[ 
    \begin{array}{c} 
           \X \\ 
           \Y 
    \end{array} 
\right]  \sim  \NN\left(    \left[ \begin{array}{c} \mu_1 \\ \mu_2 \end{array} \right], \left[ \begin{array}{cc} \sa_{11} & \sa_{12} \\ \sa_{21} & \sa_{22} \end{array} \right] \right)
\ee
with $\sa_{11}$ and $\sa_{22}$ nonsingular.

Represented categorically, these random variables $\X$ and $\Y$ determine distributions which we represent by $P_1$ and $P_2$ on two measurable spaces $X=\Rv{m}$ and $Y = \Rv{n}$ for some finite integers $m$ and $n$, and the various relationships between the $\prob$ maps is given by the diagram in Figure~\ref{fig:jointNormal}.

 \begin{figure}[H]
\begin{center}
 \begin{tikzpicture}[baseline=(current bounding box.center)]
         \node          (1)   at       (0,0)       {$1$};
         \node         (XY)  at      (0,-1.5)  {$X \times Y$};
 	\node	(X1)	at	(-3,-3)              {$X$};
	\node	(X2)	at	(3,-3)	               {$Y$};

         \draw[->,right] (1) to node {$J$} (XY);
         \draw[->,above,left] (XY) to node [yshift=3pt] {$\delta_{\pi_X}$} (X1);
         \draw[->,above,right] (XY) to node [yshift=3pt] {$\delta_{\pi_Y}$} (X2);
         \draw[->,left,out=180,in=90,looseness=1] (1) to node {$P_1 \sim \NN(\mu_1,\sa_{11})$} (X1);
         \draw[->,right,out=0,in=90,looseness=1] (1) to node {$P_2 \sim \NN(\mu_2,\sa_{22})$} (X2);
         
	\draw[->, above] ([yshift=2pt] X1.east) to node [yshift=3pt] {$\overline{\mcS}$} ([yshift=2pt] X2.west);
	\draw[->, below] ([yshift=-2pt] X2.west) to node [yshift=-3pt] {$\overline{\mcI}$} ([yshift=-2pt] X1.east);
         
 \end{tikzpicture}
\end{center}
\caption{The categorical characterization of a joint normal distribution. }
\label{fig:jointNormal}
\end{figure}
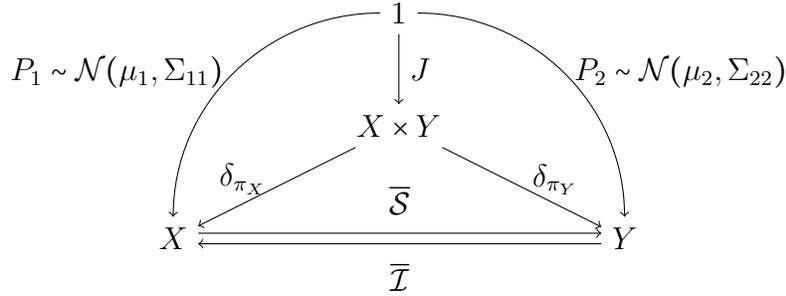

Here $\overline{\mcS}$ and $\overline{\mcI}$ are the conditional distributions
\begin{align*}
        &\overline{\mcS}_x \sim \NN\left(\mu_2 + \sa_{21} \sa_{11}^{-1}(x - \mu_1), \sa_{22} - \sa_{21} \sa_{11}^{-1} \sa_{12}\right)\\
        &\overline{\mcI}_y \sim \NN\left(\mu_1 + \sa_{12} \sa_{22}^{-1}(y - \mu_2), \sa_{11} - \sa_{12} \sa_{22}^{-1} \sa_{21}\right)
\end{align*}
 and the overline notation on the terms ``$\overline{\mcS}$'' and ``$\overline{\mcI}$'' is used to emphasize that the transpose of both of these conditionals are GPs given by a
bijective correspondence in Figure~\ref{fig:normal}.

 \begin{figure}[H]
\begin{center}
 \begin{tikzpicture}[baseline=(current bounding box.center)]
 	\node	(X1)	at	(0,-3)              {$X \otimes 1$};
	\node	(XY)	at	(0,0)	               {$ X \otimes Y^{X}$};
	\node	(X2)	at	(3,0)               {$Y$};
         \node         (1)    at      (-4,-3)             {$1$};
         \node         (YX) at      (-4,0)              {$Y^{X}$};

	\draw[->, left] (X1) to node  {$\Gamma_{\mcS}$} (XY);
	\draw[->,below, right] (X1) to node [xshift=0pt,yshift=0pt] {$\overline{\mcS}$} (X2);
	\draw[->,above] (XY) to node {$ev_{X,Y}$} (X2);
	\draw[->,left] (1) to node {$\mcS$} (YX);

 \end{tikzpicture}
\end{center}
\caption{The defining characteristic property of the evaluation function $ev$ for graphs. }
\label{fig:normal}
\end{figure}
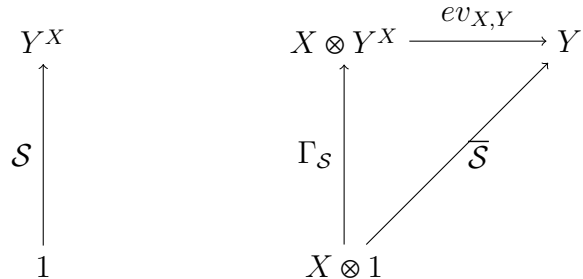
\noindent 

 In the random variable description, these conditionals $\overline{\mcS}_{\xv}$ and $\overline{\mcI}_y$ are often represented simply by
\be \nonumber
\mu_{\Y|\X} = \mu_2 + \sa_{21} \sa_{11}^{-1}(x - \mu_1) \quad   \mu_{\X|\Y} = \mu_1 + \sa_{12} \sa_{22}^{-1}(y - \mu_2)  
\ee
and
\be \nonumber
\sa_{\Y|\X} = \sa_{22} - \sa_{21} \sa_{11}^{-1} \sa_{12} \quad  \quad \sa_{\X|\Y} = \sa_{11} - \sa_{12} \sa_{22}^{-1} 
\ee

It is easily verified that this pair $\{\mcS,\mcI\}$ forms a sampling distribution/inference map pair; {\em i.e.}, the joint distribution can be expressed in terms of the prior $\X$ and sampling distribution $\overline{\mcS}$ or in terms of the prior $\Y$ and inference map $\overline{\mcI}$.  It is clear from this example that what one calls the sampling distribution and inference map depends upon the perspective of what is being estimated.

In subsequent developments, we do not assume a joint normal distribution on the spaces $X$ and $Y$. If such an assumption is reasonable, then the following constructions are greatly simplified by the structure expressed in Figure~\ref{fig:jointNormal}.  As noted previously, it is knowledge of the relationship between the distributions $P_1$ and $P_2$ which characterize the joint and, is the main modeling problem.  Thus the two perspectives on the problem are to find the conditionals, or equivalently, find the prior on $Y^X$ which specifies a function $X \rightarrow Y$ along with the noise model which is ``built into'' the sampling distribution.

\section{Bayesian Models for Function Estimation}             
\label{sec:samplingDist}
We now have all the necessary tools to  build several Bayesian models, both parametric and nonparametric, which illustrate the model building process for ML using CT.  To say we are building Bayesian models means we are constructing the two $\prob$ arrows, $P_H$ and $\mcS$, corresponding to (1) the prior probability, and (2) the sampling distribution of the diagram in Figure~\ref{fig:genericBM}.  The sampling distribution will generally be a composite of several simple $\prob$ arrows.
We start with the  nonparametric models which are in a modeling sense more basic than the   parametric models involving a fixed finite number  of parameters to be determined.  The inference maps $\mcI$ for all of the models will be constructed  in Section~\ref{sec:updating}.   

\subsection{Nonparametric Models}
  In estimation problems where the unknown quantity of interest is a function $f: X \rightarrow Y$, our hypothesis space $H$ will be the function space $Y^X$.  However, simply expressing the hypothesis space as $Y^X$ appears untenable because, in supervised learning, we never measure $Y^X$ directly, but only measure a finite number of sampling points $\{(\xv_i,y_i)\}_{i=1}^N$ satisfying some measurement model such as $y_i=f(\xv_i) + \epsilon$ where $f$ is an ``ideal'' function we seek to determine.    
  
With precise knowledge of the input state $\xv$ and assuming a generic stochastic process $P\colon 1 \rightarrow Y^X$, we are led to propose either the left $\delta_{\xv} \ltensor P$ or right $\delta_{\xv} \rtensor P$ tensor product as a prior on the hypothesis space $X \otimes Y^X$. However, when one of the components in a left or right tensor product  is a Dirac measure, then both the left and right tensors coincide and the choice of right or left tensor is irrelevant. In this case,  we denote the common probability measure by $\delta_{\xv} \otimes P$.  Moreover, a simple calculation shows the prior $\delta_{\xv} \otimes P = \Gamma_P( \cdot \mid \xv)$, the graph of $P$ at $\xv$.   Thus our proposed model, in analogy to the generic Bayesian model, is given by the diagram in Figure~\ref{fig:conditionalBM}.{\footnote{It would be interesting to analyze the more general case where there is uncertainty in the input state also and take the prior as $Q \rtensor P$ or $Q \ltensor P$ for some measure $Q$ on $X$.}

\begin{figure}[H]
 \begin{equation} \nonumber
 \begin{tikzpicture}[baseline=(current bounding box.center)]
         \node         (X)    at      (0,2)         {$1$};
	\node	(H)	at	(-2,0)	      {$X \otimes Y^X$};
	\node	(D)	at	(2,0)               {$X \otimes Y$};	
	\node         (d)    at      (5,1)       {$d$ is measurement data};
	\draw[->, left] (X) to node [xshift=-3pt] {$\Gamma_P(\cdot \mid  \xv)$} (H);
	\draw[->,right,dashed] (X) to node [xshift=2pt] {$d$} (D);
	
	\draw[->, above] ([yshift=2pt] H.east) to node [yshift=3pt] {$\mcS$} ([yshift=2pt] D.west);
 \end{tikzpicture}
 \end{equation}
 \caption{The generic nonparametric Bayesian model for stochastic processes.}
 \label{fig:conditionalBM}
 \end{figure}
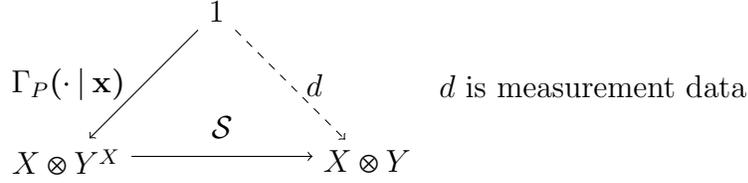

By a \emph{nonparametric} (Bayesian) model, we mean any model which fits into the scheme of Figure~\ref{fig:conditionalBM}.  For all of our analysis purposes we take $P \sim \GP(m,k)$.  A data measurement $d$, corresponding to a collection of sample data $\{\xv_i,y_i\}$ is, 
in ML applications,  generally taken as a Dirac measure, $d = \delta_{(\xv,y)}$.
As in all Bayesian problems, the measurement data $\{\xv_i,y_i\}_{i=1}^N$ can be analyzed either sequentially or as a single batch of data.   For analysis purpose in Section~\ref{sec:updating}, we consider the data  one point at a time (sequentially).

 \subsubsection{Noise Free Measurement Model} 
 
  In the noise free measurement model, we make the hypothesis that the data we observe---consisting of input output pairs $(\xv_i,y_i) \in X \times Y$---satisfies the condition that $y_i = f(\xv_i)$ where $f$ is the unknown function we are seeking to estimate.  While the actual measured data will generally not satisfy this hypothesis,  this model serves both as an idealization and a building block  for the subsequent noisy measurement model.  
 
Using the fundamental maps $\Gamma_{\cdot}\colon Y^X \rightarrow (X \otimes Y)^X$ and $ev\colon X \otimes (X \otimes Y)^X \rightarrow X \otimes Y$ gives a sequence of measurable maps which determine corresponding deterministic $\prob$ maps. This composite,  shown in Figure~\ref{fig:noiseFreeSD}, is our noise free sampling distribution.

\begin{figure}[H]
\begin{equation}  \nonumber
 \begin{tikzpicture}[baseline=(current bounding box.center)]

         \node (XYX2) at  (0,0)  {$X \otimes Y^X$};
         \node (XYg)  at  (4,0)   {$X \otimes (X \otimes Y)^X$};
         \node (Y)    at     (8,0)  {$X \otimes Y$};
         
	\draw[->,above] (XYX2) to node {$1 \otimes \delta_{\Gamma}$} (XYg);
	\draw[->,above] (XYg) to node {$\delta_{ev}$} (Y);
	\draw[->,below,out=-45,in=225,looseness=.3] (XYX2) to node {$\mcS_{nf}$ = composite} (Y);
	 \end{tikzpicture}
 \end{equation}
 \caption{The noise free sampling distribution $\mcS_{nf}$.}
 \label{fig:noiseFreeSD}
\end{figure}

This deterministic sampling distribution is given by the calculation of the composition, {\em i.e.}, evaluating the integral
\be  \nonumber
\begin{array}{lcl}
\mcS_{nf}(U \mid  (\xv,f)) &=& \int_{(\uv,g) \in X \otimes (X \otimes Y)^X} (\delta_{ev})_U(\uv,g ) \, d (1 \otimes \delta_{\Gamma})_{(\xv, f)}  \quad \textrm{for }U \in \sa_{X \otimes Y}\\
&=& (\delta_{ev})_U( \xv,  \Gamma_f ) \\
&=& \delta_{\Gamma_f(\xv)}(U) \\
&=& \delta_{(\xv,f(\xv))}(U) \\
&=& \delta_{(\Gamma_{\overline{\xv}}(ev_{\xv}(\fn)))}(U) \quad  \textrm{ because }  (\xv,f(\xv)) = \Gamma_{\overline{x}}(ev_{\xv}(f)) \\
&=& \chi_{U} (\Gamma_{\overline{\xv}}(ev_{\xv}(f))) \\
&=& \chi_{ ev_{\xv}^{-1}(\Gamma_{\overline{\xv}}^{-1}(U))}(f).
\end{array}
\ee
Using the commutativity of Figure~\ref{fig:graphMap}, 
 the noise free sampling distribution can also be written as $\mcS_{nf}(U \mid  (\xv, f))=\chi_{\Gamma_{\cdot}^{-1} \left( \hat{ev}_{\xv}^{-1}(U) \right)}(f)$.
 
Precomposing the sampling distribution with this prior probability measure  the composite 
\be \label{Q2}
\begin{array}{lcl}
(\mcS_{nf} \circ \Gamma_{P}(\cdot \mid  \xv))(U) 
&=& \int_{(\uv,f) \in X \otimes Y^X} \mcS_{nf}(U \mid  (\uv, f)) \, d(\underbrace{\Gamma_{P}( \cdot \mid  \xv)}_{ =P \Gamma_{\overline{\xv}}^{-1}}) \quad \textrm{ for }U \in \sa_{X \otimes Y} \\
&=& \int_{f \in Y^X} \mcS_{nf}(U \mid  \Gamma_{\overline{\xv}} (f)) \, dP \\
&=& \int_{f \in Y^X} \chi_{\Gamma_{\cdot}^{-1} \left( \hat{ev}_{\xv}^{-1}(U) \right)}(f) \, dP  \\
&=& P( \Gamma^{-1}_{\cdot} (\hat{ev}_{\xv}^{-1}(U)))  
\end{array}
\ee
By the relation $\Gamma_{\overline{\xv}} \circ ev_{\xv} =\hat{ev}_{\xv} \circ \Gamma_{\cdot}$ this  can also be written as
\be  \nonumber
(\mcS_{nf} \circ  \Gamma_{P}(\cdot \mid  \xv))(U) = P(ev_{\xv}^{-1}(\Gamma_{\overline{\xv}}^{-1}(U))).
\ee

Given that the probability measure $P$ is specified as a Gaussian process (which is defined in terms of how it restricts to finite subspaces $X_0 \subset X$), for computational purposes we need to consider the push forward probability measure of $P$ on $Y^X$ to $Y^{X_0}$ as in Figure~\ref{fig:GPdef}.  
Taking the special case with $X_0 = \{\xv \}$, the pushforward corresponds to composition with the deterministic projection map $\delta_{ev_{\xv}}$.   Starting with the diagram of Figure~\ref{fig:graphMap}, precomposing with $P$ and postcomposition with the deterministic map $\delta_{\pi_Y} \circ \delta_{\iota}$ gives  the diagram in Figure~\ref{fig:SDComputation0}. Then we can use the fact $P$ projected onto any coordinate is a Gaussian distribution to compute the likelihood that a measurement will occur in a measurable set $B \subset Y$.

\begin{figure}[H]
\begin{equation}  \nonumber
 \begin{tikzpicture}[baseline=(current bounding box.center)]
         \node  (1)  at (-1.5,0)    {$1$};
         \node (YX) at   (3,0)  {$Y^X$};
         \node (XYX) at  (6,0)  {$(X \otimes Y)^X$};
         \node (XY)  at  (9,0)   {$X \otimes Y$};
         \node (Y)  at    (6,-2)   {$Y$};
         
	\draw[->,above] (1) to node {$P \sim \GP(m,k)$} (YX);
	\draw[->,above] (YX) to node {$\Gamma_{\cdot}$} (XYX);
	\draw[->,above] (XYX) to node {$\hat{ev}_{\xv}$} (XY);
	\draw[->,right] (XY) to node {$\delta_{\pi_Y}$} (Y);
	\draw[->,above,right,dashed] (YX) to node  [xshift=3pt,yshift=4pt] {$\delta_{ev_{\xv}}$} (Y);
	\draw[->,below, left]  (1) to node [xshift = -5pt,yshift=-5pt] {$P ev_{\xv}^{-1} \sim \NN(m(\xv),k(\xv,\xv))$} (Y);

	 \end{tikzpicture}
 \end{equation}
 \caption{The distribution $P \sim \GP(m,k)$  can be evaluated on rectangles $U=A \times B$ by projecting onto the given $x$ coordinate.}
 \label{fig:SDComputation0}
 \end{figure}
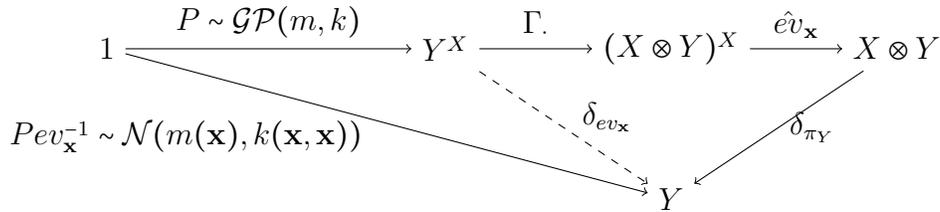
Under this assumption  $P \sim \GP(m,k)$ the expected value of the probability measure $(\mcS_{nf} \circ \Gamma_{P}(\cdot \mid  \xv))$ on the real vector space $X \otimes Y$ is
\be  \nonumber
\begin{array}{lcl}
\E_{(\mcS_{nf} \circ \Gamma_{P}(\cdot \mid  \xv))}[Id_{X \otimes Y}] &=&\int_{(\uv,\vv) \in X \otimes Y} (\uv,\vv) \, d ( P(\Gamma_{\cdot}^{-1} \hat{ev}_{\xv}^{-1}) \\ 
&=& \int_{g \in (X \otimes Y)^X} \hat{ev}_{\xv}(g) \, d(P \Gamma_{\cdot}^{-1}) \\
&=& \int_{f  \in Y^X} \hat{ev}_{\xv}(\Gamma(f)) \, dP \\
&=& \int_{f \in Y^X} (\xv, f(\xv)) \, dP \\
&=& (\xv,m(\xv)), 
\end{array}
\ee
where the last equation follows because on the two components of the vector valued integral, $\int_{f \in Y^X} f(\xv) \, dP = m(\xv)$ and $\int_{f \in Y^X} \xv \, dP = \xv$ as the integrand is constant.
The variance is\footnote{The squaring operator in the variance is defined component wise on the vector space $X \otimes Y$.}
\be  \nonumber
\E_{(\mcS_{nf} \circ \Gamma_{P}(\cdot \mid  \xv))}[ (Id_{X \otimes Y} - \E_{(\mcS_{nf} \circ \Gamma_{P}(\cdot \mid  \xv))}[Id_{X \otimes Y}])^2] = \E_{(\mcS_{nf} \circ \Gamma_{P}(\cdot \mid  \xv))}[ (Id_{X \otimes Y} - (\xv,m(\xv)))^2] , 
\ee
which when expanded gives
\be \nonumber
\begin{array}{lcl}
 &=& \int_{(\uv,v) \in X \otimes Y} (Id_{X \otimes Y} - (\xv,m(\xv)))^2(\uv,v) \, d(P(\Gamma_{\cdot}^{-1}ev_{\xv}^{-1})) \\
&=& \int_{f \in Y^X} (Id_{X \otimes Y} - (\xv,m(\xv)))^2  \underbrace{(ev_{\xv}(\Gamma_{\cdot}(f)))}_{=(\xv,f(\xv))}  \, dP \\
&=& \int_{f \in Y^X} \left(  (\xv - \xv)^2, \left(f(\xv) - m(\xv)\right)^2 \right) \, dP  \\
&=& (\mathbf{0}, k(\xv,\xv)).
\end{array}
\ee
Consequently this sampling distribution, together with the prior distribution $\delta_{\xv} \rtensor P=\Gamma_{P}(\cdot \mid  \xv)$, provide what we expect of such a model.

\subsubsection{Gaussian Additive Measurement Noise Model}  Additive noise measurement models are often expressed by the simple expression
\be \label{noiseEx}
z = y + \epsilon
\ee
where $y$ represents the state while the $\epsilon$ term itself represents a normally distributed random variable with zero mean and variance $\sigma^2$.  In categorical terms this expression corresponds to the map in Figure~\ref{fig:additiveNoiseModel}.

\begin{figure}[H]
\[
 \begin{tikzpicture}[baseline=(current bounding box.center)]
         \node  (1)  at (0,0)    {$1$};
         \node (Y)   at  (4,0)   {$Y$};
         
	\draw[->,above] (1) to node  {$M_y \sim \NN(y,\sigma^2)$} (Y);
	 \end{tikzpicture}
 \]
 \caption{The additive Gaussian noise measurement model.}
 \label{fig:additiveNoiseModel}
 \end{figure}

 Because the state $y$ in Equation~\ref{noiseEx} is arbitrary, this additive noise model is representative of the $\prob$ map $Y \stackrel{M}{\longrightarrow} Y$
defined by
 \be  \nonumber
 M(B\mid y) = M_y(B) \quad \forall y \in Y, \, \forall B \in \sa_Y.
 \ee
 
 Given a GP $P \sim \GP(f,k)$ on $Y^X$,  it follows that for any $\xv \in X$,  $Pev_{\xv}^{-1} \sim \NN(f(\xv),k(\xv,\xv))$ and for any $B \in \sa_Y$, the composition
 
\[
 \begin{tikzpicture}[baseline=(current bounding box.center)]
         \node  (1)  at (-2,0)    {$1$};
         \node (Y)   at  (4,0)   {$Y$};
         \node (Y2) at (8,0)    {$Y$};
         
	\draw[->,above] (1) to node  {$Pev_{\xv}^{-1} \sim \NN(f(\xv),k(\xv,\xv))$} (Y);
	\draw[->,above] (Y) to node {$M$} (Y2);
	 \end{tikzpicture}
 \]
is
\be \nonumber
\begin{array}{lcl}
(M \circ P ev_{\xv}^{-1})(B) &=& \int_{u \in Y} M_B(u) \, d(Pev_{\xv}^{-1}) \\
&=& \int_{u \in Y} \left( \frac{1}{\sqrt{2 \pi} \sigma} \int_{v \in B} e^{ -\frac{(v - u)^2}{2 \sigma^2}} \, dv \right) d(Pev_{\xv}^{-1}) \\
&=& \frac{1}{\sqrt{2 \pi k(f(\xv),f(\xv))}}  \,  \int_{u \in Y} \left( \frac{1}{\sqrt{2 \pi} \sigma} \int_{v \in B} e^{ -\frac{(v - u)^2}{2 \sigma^2}} \, dv \right)   e^{ - \frac{(u - f(\xv))^2}{2 \cdot k(f(\xv),f(\xv))}} du \\
&=& \frac{1}{2 \pi  \cdot \sigma \cdot \sqrt{k(f(\xv),f(\xv))}} \, \int_{v \in B}  \int_{u \in Y} e^{ -\frac{(v - u)^2}{2 \sigma^2}} \,  e^{ - \frac{(u - f(\xv))^2}{2 \cdot k(f(\xv),f(\xv))}}  \, du \,  dv \\
&=& \frac{1}{ \sqrt{2 \pi (k(\xv,\xv) + \sigma^2)}} \int_{v \in B} e^{-\frac{(v-f(\xv))^2}{2( k(\xv,\xv) + \sigma^2)}} \, dv.
\end{array}
\ee
Thus  this composite is the normal distribution 
 \begin{equation} \label{compositeMP} 
 \begin{tikzpicture}[baseline=(current bounding box.center)]
         \node  (1)  at (-4,0)    {$1$};
         \node (Y)   at  (4,0)   {$Y$};

         \draw[->,above] (1) to node {$M \circ P ev_{\xv}^{-1} \sim \NN(f(\xv),k(\xv,\xv) + \sigma^2)$} (Y);
 \end{tikzpicture}
 \end{equation}
 More generally we have the commutative $\prob$ diagram given in Figure~\ref{fig:compositeGP}, where, for all $f \in Y^X$,
\be  \label{noiseGP}
N_f \sim \GP(f,k_N) \quad \quad  
k_N(\xv,\xv') = \left\{ \begin{array}{ll} \sigma^2 & \textrm{ iff }\xv =\xv' \\ 0 & \textrm{ otherwise. } \end{array} \right. 
\ee

\begin{figure}[H]
\[
 \begin{tikzpicture}[baseline=(current bounding box.center)]
         \node  (1)  at (0,-1.5)    {$1$};
         \node (YX)   at  (4,0)   {$Y^X$};
         \node (YX2) at (8,0)    {$Y^X$};
         \node (Y)  at (4,-3)      {$Y$};
         \node (Y2)  at (8,-3)      {$Y$};
         
	\draw[->,above,left] (1) to node[xshift=-2pt,yshift=10pt]  {$P \sim \GP(f,k)$} (YX);
	\draw[->,above] (YX) to node {$N$} (YX2);
	\draw[->,below,left] (1) to node [xshift=-5pt,yshift=-4pt] {$Pev_{\xv}^{-1} \sim \NN(f(\xv),k(\xv,\xv))$} (Y);
	\draw[->,above] (Y) to node {$M$} (Y2);
	\draw[->,right] (YX) to node {$\delta_{ev_{\xv}}$} (Y);
	\draw[->,right] (YX2) to node {$\delta_{ev_{\xv}}$} (Y2);
	 \end{tikzpicture}
\]
 \caption{Construction of the generic Markov kernel $N$ for modeling the Gaussian additive measurement noise.}
 \label{fig:compositeGP}
 \end{figure}
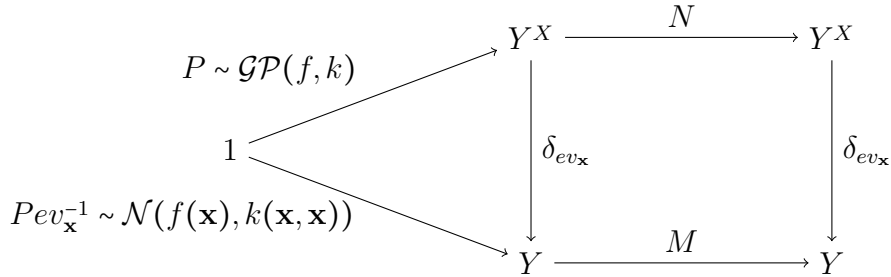

The commutativity of the right hand square in Figure~\ref{fig:compositeGP}  follows from
\be \nonumber
\begin{array}{lcl}
(\delta_{ev_{\xv}} \circ  N)(B \mid f) &=& \int_{g \in Y^X} (\delta_{ev_{\xv}})_B(g) \, dN_f \\
&=& N_f (ev_{\xv}^{-1}(B)) \\
&=& N_{f(\xv)}(B) \\
&=& \int_{y \in Y} M_B(y) \, d (\underbrace{\delta_{ev_{\xv}})_f}_{ = \delta_{f(\xv)}} \\
&=& (M \circ \delta_{ev_{\xv}})(B \mid f).
\end{array}
\ee

With this Gaussian additive noise measurement model $N$ our sampling distribution $\mcS_{nf}$ can easily be modified by incorporating the additional map $N$ into the sequence in Figure  \ref{fig:noiseFreeSD} to yield the Gaussian additive noise sampling distribution model $\mcS_n$  shown in Figure~\ref{fig:genericNoiseModel}.
\begin{figure} [H]
\[
 \begin{tikzpicture}[baseline=(current bounding box.center)]
         \node (XYX) at   (1,0)  {$X \otimes Y^X$};
         \node (XYX2) at  (4.5,0)  {$X \otimes Y^X$};
         \node (XYg)  at  (8.5,0)   {$X \otimes (X \otimes Y)^X$};
         \node (Y)    at     (12,0)  {$X \otimes Y$};

	\draw[->,below] (XYX) to node {$1_X \otimes N$} (XYX2);
	\draw[->,below] (XYX2) to node {$1_X \otimes \delta_{\Gamma_{\cdot}}$} (XYg);
	\draw[->,below] (XYg) to node {$\delta_{ev}$} (Y);
	\draw[->,below,out=-45,in=225,looseness=.3] (XYX) to node {$\mcS_n$ = composite} (Y);
	 \end{tikzpicture}
\]
 \caption{The sampling distribution model in $\prob$ with additive Gaussian noise.}
 \label{fig:genericNoiseModel}
 \end{figure}

Here $1_X \otimes N$ is, by the definition given in Section~\ref{sec:tensorP},
\be \nonumber
(1_X \otimes N) \left(U, (\xv,f) \right) = N( \Gamma_{\overline{\xv}}^{-1}(U) \mid  f) 
\ee
so the nondeterministic noisy sampling distribution is given by
\be  \label{noisySD}
\begin{array}{lcl}
\mcS_n(U \mid  (\xv,f)) &=& \left(\mcS_{nf} \circ (1 \otimes N) \right)(U \mid  (\xv,f))  \quad \textrm{for }U \in \sa_{X \otimes Y}\\ \\
&=& \int_{(\uv,g) \in X \otimes Y^X} (\mcS_{nf})_U(\uv,g) \, d (N(\Gamma_{\overline{\xv}}^{-1}(\cdot) \mid  f) \\
&=& \int_{g \in Y^X} (\mcS_{nf})_U( \Gamma_{\overline{\xv}}(g)) \, d N(\cdot \mid  f)  \\
&=& \int_{g \in Y^X} (\mcS_{nf})(U \mid   (\xv,g) ) \, d N( \cdot \mid  f)  \\
&=& \int_{g \in Y^X} \chi_{\Gamma_{\cdot}^{-1} \left( \hat{ev}_{\xv}^{-1}(U) \right)}(g) dN(\cdot \mid  f) \\
&=&N \left( \Gamma_{\cdot}^{-1} \left( \hat{ev}_{\xv}^{-1}(U) \right) \mid  f \right) \\ 
&=& N \left( ev_{\xv}^{-1}( \Gamma_{\overline{\xv}}^{-1}(U) ) \mid  f \right) \\
&=& N_f  ev_{\xv}^{-1}( \Gamma_{\overline{\xv}}^{-1}(U) ).
\end{array}
\ee

Just as we did for the GP $P\colon 1 \rightarrow Y^X$ in Figure~\ref{fig:SDComputation0},  each GP $N_f$ can be analyzed by its push forward measures onto any coordinate $\xv \in X$ to obtain the diagram in Figure~\ref{fig:SDComputation2}.

\begin{figure}[H]
\begin{equation}  \nonumber
 \begin{tikzpicture}[baseline=(current bounding box.center)]
         \node  (1)  at (-1.5,0)    {$1$};
         \node (YX) at   (3,0)  {$Y^X$};
         \node (XYX) at  (6,0)  {$(X \otimes Y)^X$};
         \node (XY)  at  (9,0)   {$X \otimes Y$};
         \node (Y)  at    (6,-2)   {$Y$};
         
	\draw[->,above] (1) to node {$N_f \sim \GP(f,k_N)$} (YX);
	\draw[->,above] (YX) to node {$\delta_{\Gamma_{\cdot}}$} (XYX);
	\draw[->,above] (XYX) to node {$\delta_{\hat{ev}_{\xv}}$} (XY);
	\draw[->,right] (XY) to node {$\delta_{\pi_Y}$} (Y);
	\draw[->,above,right,dashed] (YX) to node  [xshift=3pt,yshift=4pt] {$\delta_{ev_{\xv}}$} (Y);
	\draw[->,below, left]  (1) to node [xshift = -5pt,yshift=-5pt] {$N_f \left( ev_{\xv}^{-1}(\cdot) \right) \sim \NN(f(\xv),\sigma^2)$} (Y);

	 \end{tikzpicture}
 \end{equation}
 \caption{The GP $N_f$ can be evaluated on rectangles $U=A \times B$ by projecting onto the given $x$ coordinate.}
 \label{fig:SDComputation2}
 \end{figure}
 
Taking $U$ as a rectangle, $U = A \times B$, with $A \in \sa_X$ and $B \in \sa_Y$, the likelihood that a measurement will occur in the rectangle conditioned on $(\xv,f)$ is given  by
\be  \nonumber
\begin{array}{lcl}
\mcS_n(A \times B \mid  (\xv,f)) &=& N_f ev_{\xv}^{-1}( \Gamma_{\overline{\xv}}^{-1}(A \times B)) \\
&=& \delta_{\xv}(A) \cdot N_f ev_{\xv}^{-1}(B) \\
&=&  \delta_{\xv}(A) \, \cdot \, \frac{1}{\sqrt{2 \pi} \sigma} \int_{y \in B} e^{-\frac{(y - f(\xv))^2}{2 \sigma^2}} \, dy.
\end{array}
\ee

Using the associativity property of categories, from Figure~\ref{fig:genericNoiseModel} with a prior $\Gamma_{P}(\cdot \mid  \xv)$ on $X \otimes Y^X$, the composite $\mcS_n \circ \Gamma_{P}(\cdot \mid  \xv)$ can be decomposed as 
\be \nonumber
\mcS_n \circ \Gamma_{P}(\cdot \mid  \xv) = \mcS_{nf} \circ ((1_X \otimes N) \circ \Gamma_{P}(\cdot \mid  \xv))
\ee
while the term $((1_X \otimes N) \circ \Gamma_{P}(\cdot \mid  \xv)) = \Gamma_{N \circ P}(\cdot \mid  \xv)$ follows from the commutativity of the diagram in Figure~\ref{fig:GPComposite}, where, as shown in Equation~\ref{compositeMP}, $M \circ Pev_{\xv}^{-1} \sim \NN(m, k(\xv,\xv) + \sigma^2)$ which implies $N \circ P  \sim \GP(m,k+k_N)$.

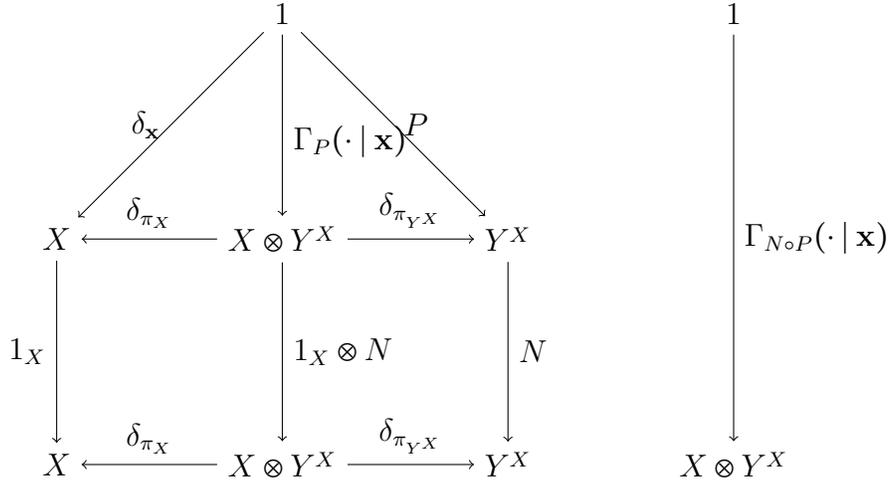
\begin{figure}[H]
\begin{equation}  \nonumber
 \begin{tikzpicture}[baseline=(current bounding box.center)]
         \node  (X)  at (0,0)    {$1$};
         \node (X1) at   (-3,-3)  {$X$};
         \node (X2) at  (-3,-6)   {$X$};
         
         \node (YX)    at     (3,-3)  {$Y^X$};
         \node (YX2)  at    (3,-6)   {$Y^X$};
         
         \node (XYX)  at (0,-3)  {$X \otimes Y^X$};
         \node (XYX2) at (0,-6)  {$X \otimes Y^X$};
         
         \node (X3) at  (6,0)    {$1$};
         \node (XYX3)  at (6, -6)  {$X \otimes Y^X$};
         
	\draw[->,above,left] (X) to node {$\delta_{\xv}$} (X1);
	\draw[->,left] (X1) to node {$1_X$} (X2);

	\draw[->,right] (X) to node {$P$} (YX);
	\draw[->,right] (YX) to node {$N$} (YX2);
	
	\draw[->,right] (X) to node [yshift=-6pt] {$\Gamma_P(\cdot \mid  \xv)$} (XYX);
	\draw[->,right] (XYX) to node {$1_X \otimes N$} (XYX2);
	
         \draw[->,above] (XYX) to node {$\delta_{\pi_{X}}$} (X1);
         \draw[->,above] (XYX2) to node {$\delta_{\pi_{X}}$} (X2);
         \draw[->,above] (XYX) to node {$\delta_{\pi_{Y^X}}$} (YX);
         \draw[->,above] (XYX2) to node {$\delta_{\pi_{Y^X}}$} (YX2);
         
         \draw[->,right] (X3) to node {$\Gamma_{N \circ P}(\cdot \mid  \xv)$} (XYX3);
         
	 \end{tikzpicture}
 \end{equation}
 \caption{The composite of the prior and noise measurement model is the graph of a GP at $\xv$.}
 \label{fig:GPComposite}
 \end{figure}

Using the fact $\mcS_n \circ \Gamma_{P}(\cdot \mid  \xv) = \mcS_{nf} \circ  \Gamma_{N \circ P}(\cdot \mid  \xv)$, the expected value of the composite $\mcS_n \circ \Gamma_{P}(\cdot \mid  \xv)$  is readily  shown to be

\be  \nonumber
\E_{(\mcS_n \circ \Gamma_{P}(\cdot \mid  \xv))}[Id_{X \otimes Y}] = (\xv,m(\xv)) 
\ee
 while the variance is
\be  \nonumber
\E_{(\mcS_n \circ \Gamma_{P}(\cdot \mid  \xv))}[ (Id_{X \otimes Y} - \E_{(\mcS_n \circ \Gamma_{P}(\cdot \mid  \xv))}[Id_{X \otimes Y}])^2] = (\mathbf{0}, k(\xv,\xv) + \sigma^2).
\ee

\subsection{Parametric Models}

A parametric model can be though of as carving out a subset of $Y^X$ specifying the form of functions which one wants to consider as valid hypotheses.  With this in mind, let us define a
$p$-dimensional  \emph{parametric map} as a measurable function
\[
i\colon\Rp  \longrightarrow Y^X 
\]
where $\Rp$ has the product $\sigma$-algebra with respect to the canonical projection maps onto the measurable space $\mathbb{R}$ with the Borel $\sigma$-algebra.  Note that $i(\av) \in Y^X$ corresponds (via the SMwCC structure) to a  function $ \overline{i(\av)}\colon X \rightarrow Y$.\footnote{Note that the function $i(\av)$ is unique by our construction of the transpose of the function $i(\av)  \in Y^X$.  The non-uniqueness aspect of the SMwCC structure only arises in the other direction - given a conditional probability measure there may be multiple functions satisfying the required commutativity condition.}
This parametric map $i$ determines the deterministic $\prob$ arrow $\delta_i\colon \Rp \rightarrow Y^X$, which in turn determines the deterministic tensor product arrow \mbox{$1_X \otimes \delta_i\colon X \otimes \Rp \longrightarrow X \otimes Y^X$}. This arrow serves as a bridge connecting the two forms of Bayesian models, the parametric and nonparametric models.

A parametric model consists of a parametric mapping  combined  with a nonparametric noisy measurement model $\mcS_n$ with prior  $(1_X \otimes \delta_i) \circ \Gamma_P(\cdot \mid  \xv)$ to give the diagram in Figure~\ref{fig:genericParametric} and we define a \emph{parametric Bayesian model} as any model which fits into the scheme of Figure~\ref{fig:genericParametric}.

\begin{figure}[H]
 \begin{equation} \nonumber
 \begin{tikzpicture}[baseline=(current bounding box.center)]
         \node         (X)    at      (1,2)         {$1$};
	\node	(H)	at	(-2,0)	      {$X \otimes \Rp$};
	\node         (YX)  at    ( 1,0)        {$X \otimes Y^X$};
	\node	(D)	at	(4,0)               {$X \otimes Y$};	
	\draw[->,left] (X) to node [xshift=-3pt] {$\Gamma_P(\cdot \mid  \xv)$} (H);
	\draw[->,below] (H) to node {$1_X \otimes \delta_i$} (YX);
	\draw[->, right,dashed] (X) to node [xshift=3pt] {$d$} (D);
	
	\draw[->, below] ([yshift=2pt] YX.east) to node  {$\mcS_n$} ([yshift=2pt] D.west);
 \end{tikzpicture}
 \end{equation}
 \caption{The generic  parametric Bayesian model.}
\label{fig:genericParametric}
 \end{figure}

In the ML literature, one generally assumes complete certainty with regards to the input state $\xv \in X$.  However, there are situations in which complete knowledge of the input state $\xv$ is itself  uncertain.  This occurs in object recognition problems where  $\xv$ is a feature vector which may be only partially observed because of  obscuration and such data is the only training data available.

For real world modeling applications there must be a noise model component associated with a parametric model for it to make sense.  For example we could estimate an unknown function as a  constant function, and hence have the $1$ parameter model $i\colon \mathbb{R} \rightarrow Y^X$ given by $i(a) = \overline{a}$, the constant function on $X$ with value $a$.  Despite how crude this approximation may be, we can still  obtain a ``best'' such Bayesian approximation to the function given measurement data where ``best''  is defined in the Bayesian  probabilistic sense -  given a prior and a measurement the posterior gives the best estimate under the given modeling assumptions.  Without a noise component, however, we cannot even account for the fact our data is different than our model which, for analysis and prediction purposes, is a worthless model.

\begin{example} \textbf{Affine Parametric Model}
 Let  $X = \Rn$ and $p=n+1$.  The affine  parametric model is given by considering  the valid hypotheses to consist of  affine functions
\be  \label{affineMap}
\begin{array}{lclcl}
F_{\mathbf{a}} &:& X & \rightarrow & Y \\
&:& \xv & \mapsto &  \sum_{j=1}^n a_j x_j + a_{n+1}
\end{array}
\ee
where $\xv = (x_1,x_2,\ldots,x_n) \in X$, the ordered $(n+1)-tuple$ \mbox{$\mathbf{a} = (a_1,\ldots,a_n, a_{n+1}) \in \mathbb{R}^{^{n+1}}$} are fixed parameters so $F_{\mathbf{a}}  \in Y^X$ and  the parametric map
\be  \nonumber
\begin{array}{lclcl}
i&:& \mathbb{R}^{n+1} & \longrightarrow & Y^X \\
&:&\mathbf{a} & \mapsto & i(\av) = \overline{F_{\mathbf{a}}}
\end{array} 
\ee
specifies the subset of all possible affine  models $F_{\mathbf{a}}$.  

In particular, if $n=2$ and the test data consist of two data classes, say with labels $-1$ and $1$, which is separable then the coefficients $\{a_1,a_2,a_3\}$ specify the hyperplane separating the data points as shown in  Figure~\ref{fig:hyperPlane}.

\begin{figure}[H]
\begin{center}
\includegraphics[width=5in,height=3in]{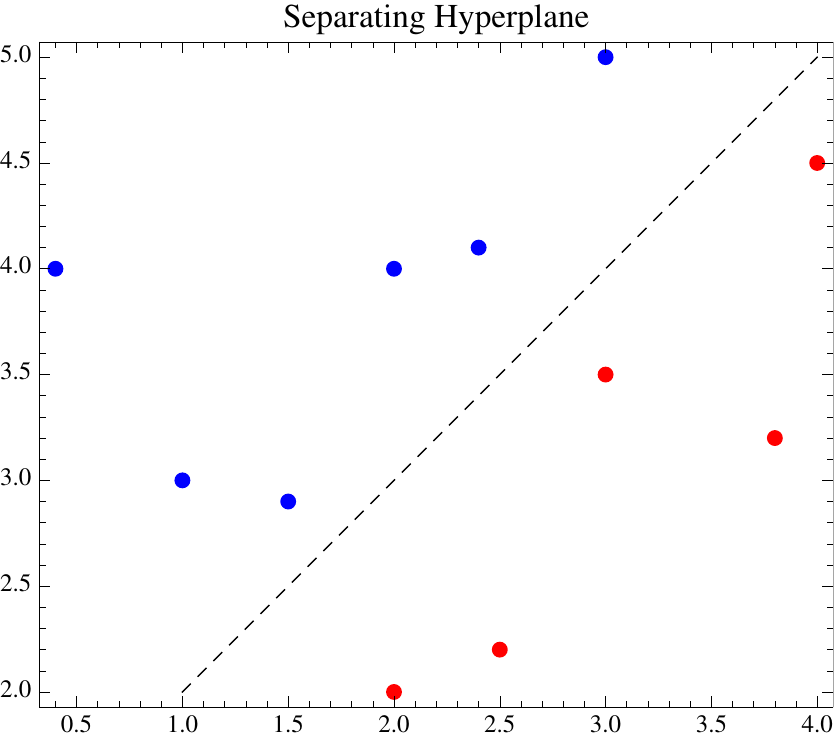}
\end{center}
\caption{An affine model suffices for separable data.}
\label{fig:hyperPlane}
\end{figure}

In this particular example where the class labels are integer valued,  the resulting function we are estimating will not be integer valued but, as usual, approximated by real values.
\end{example}

Such parametric models are useful to avoid over fitting data because the number of parameters are finite and fixed with respect to the number of measurements in contrast to nonparametric methods in which each measurement serves as a parameter defining the updated probability measure on $Y^X$. 

More generally, for any parametric map $i$ take the canonical basis vectors $\e_j$, which are  the $j^{th}$ unit vector in $\Rp$,  and let the image of the  basis elements  $\{ \e_j \}_{j=1}^p$ under the parametric map $i$ be $i(\e_j) =  f_j \in Y^X$.  
Because $Y^X$ forms a real vector space under pointwise addition and scalar multiplication, $(f+g)(\xv) = f(\xv) + g(\xv)$ and $(\alpha f)(\xv) = \alpha (f(\xv))$ for all $f,g \in Y^X, \xv \in X$, and $\alpha \in \mathbb{R}$, we observe that the ``image carved out'' by the parametric map $i$ is just the span of the image of the basis elements $\{e_j\}_{j=1}^p$. 
 In the above example $f_j = \pi_j$, for $j=1,2$ where $\pi_j$ is the canonical projection map $\mathbb{R}^2 \rightarrow \mathbb{R}$, and $f_3 = \overline{1}$, the constant function with value $1$ on all points $\xv \in X$.  Thus the image is as specified by the Equation~\ref{affineMap}.

 \begin{example} \textbf{Elliptic Parametric Model}  When the data is not linearly separable as in the previous example, but rather of the form shown in Figure~\ref{fig:quadratic}, then a higher order parametric model is required. 
 \begin{figure}[H]
\begin{center}
\includegraphics[width=4in,height=3in]{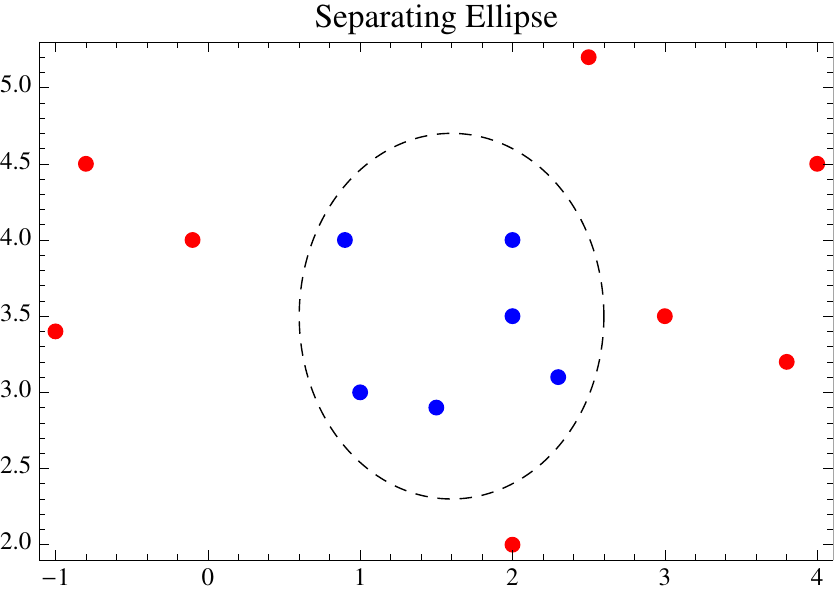}
\end{center}
\caption{An elliptic parametric model suffices to separate the data.}
\label{fig:quadratic}
\end{figure}

Taking  $X = \mathbb{R}^n$ and $p=n^2+n+1$, the elliptic  parametric model is given by considering  the valid hypotheses to consist of  all elliptic functions
\be  
\begin{array}{lclcl}
F_{\mathbf{a}} &:& X & \rightarrow & Y \\
&:& \xv & \mapsto &  \sum_{j=1}^n a_j x_j +  \sum_{j=1}^n \sum_{k=1}^n a_{n + n(j-1) + k} x_j x_k + a_{n^2 + n +1}
\end{array}
\ee
where $\xv = (x_1,x_2,\ldots,x_n) \in X$, the ordered $(n^2 + n +1)-tuple$ \mbox{$\mathbf{a} = (a_1,\ldots, a_{n^2 + n+1}) \in \mathbb{R}^{n^2 + n +1}$} are fixed parameters so $F_{\mathbf{a}}  \in Y^X$ and  the parametric map
\be  \nonumber
\begin{array}{lclcl}
i&:& \mathbb{R}^{^{n^2+n+1}} & \longrightarrow & Y^X \\
&:&\mathbf{a} & \mapsto & i(\av) = \overline{F_{\mathbf{a}}}
\end{array} 
\ee
specifies the subset of all possible elliptic models $F_{\mathbf{a}}$.  

With this model the linearly nonseparable data becomes separable.  This is the basic idea behind support vector machines (SVMs):  simply embed the data into a higher order space where it can be (approximately) separated by a higher order parametric model.

\end{example}
   
 Returning to the general construction of the Bayesian model for the parametric model we  take the Gaussian additive noise model, Equation~\ref{noisySD}, and expand the diagram in Figure~\ref{fig:genericParametric}  to the diagram in Figure~\ref{fig:parametricModel}, where the parametric model sampling distribution can be readily determined on rectangles $A \times B \in \sa_{X \otimes Y}$  by
\be  \nonumber
\begin{array}{lcl}
\mcS(A \times B \mid  (\xv,\av)) &=&   N \left( ev_{\xv}^{-1}( \Gamma_{\overline{\xv}}^{-1}(A \times B) ) \mid  \underbrace{i(\av)}_{=F_{\av}} \right) \\
&=& N_{F_{\av}} ev_{\xv}^{-1}(\Gamma_{\overline{\xv}}^{-1}(A \times B) ) \\
&=& \delta_{\xv}(A) \cdot \frac{1}{\sqrt{2 \pi} \sigma} \int_{y \in B} e^{- \frac{ (y - F_{\av}(\xv))^2}{2 \sigma^2}} \, dy.
\end{array}
\ee

\begin{figure}[H]
\begin{equation}  \nonumber
 \begin{tikzpicture}[baseline=(current bounding box.center)]

         \node  (1)    at    (1.8,3)    {$1$};
         \node (Rn)   at  (-4,0)    {$X \otimes \Rp$};
         \node (XYX)   at  (-1,0)  {$X \otimes Y^X$};
         \node (XYX2) at  (2,0)  {$X \otimes Y^X$};
         \node (XYg)  at  (6,0)   {$X \otimes (X \otimes Y)^X$};
         \node (Y)    at     (9,0)  {$X \otimes Y$};
         
         \draw[->,left,above] (1) to node [xshift=-16pt,yshift=3pt] {$\Gamma_P(\cdot \mid  \xv)$} (Rn);
         \draw[->,above] (Rn) to node {$1_X \otimes \delta_{i}$} (XYX);
         \draw[->,above] (XYX) to node {$1_X \otimes N$} (XYX2);
	\draw[->,above] (XYX2) to node {$1_X \otimes \delta_{\Gamma}$} (XYg);
	\draw[->,above] (XYg) to node {$\delta_{ev}$} (Y);
	\draw[->,below,out=-45,in=225,looseness=.3] (Rn) to node {$\mcS$} (Y);
	\draw[->,above,out=45,in=135,looseness=.5] (XYX) to node {$\mcS_n$} (Y);
	 \end{tikzpicture}
 \end{equation}
 \caption{The parametric model sampling distribution as a composite of four components.}
 \label{fig:parametricModel}
\end{figure}
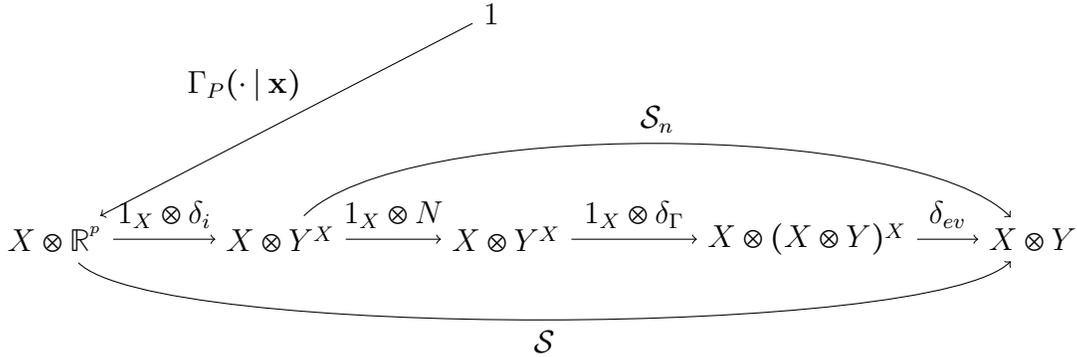

Here we have used the fact $N_{F_{\av}} ev_{\xv}^{-1} \sim \NN(F_{\av}(\xv),\sigma^2)$ which follows from Equation~\ref{noiseGP} and the property that a GP evaluated on any coordinate is a normal distribution with the mean and variance evaluated at that coordinate.

\section{Constructing  Inference Maps}  \label{sec:updating} 

We now proceed to construct the inference maps $\mcI$ for each of the models specified in the previous section.  This construction permits  the updating of the GP prior  distributions $P$  for the nonparametric models and the normal priors $P$ on $\mathbb{R}^k$ for the parametric models through the relation that the posterior measure is given by $\mcI \circ d$, where $d$ is a data measurement.  The resulting analysis produces the familiar updating rules for the mean and covariance functions characterizing a GP.   

\subsection{The noise free inference map}  \label{sec:noiseFree}

Under a prior probability of the form $\delta_{\xv} \otimes P = \Gamma_P(\cdot \mid  \xv)$ on the hypothesis space $X \otimes Y^X$, which is a one point measure with respect to the component $X$,  the sampling distribution $\mcS_{nf}$ in Figure~\ref{fig:noiseFreeSD} can be viewed as a \emph{family} of  
deterministic $\prob$ maps---one for each point $\xv \in X$. 
\begin{figure}[H]
\begin{equation}  \nonumber
 \begin{tikzpicture}[baseline=(current bounding box.center)]

         \node (YX) at  (0,0)  {$Y^X$};
         \node (Y)    at     (4,0)  {$Y$};
         
	\draw[->,above] (YX) to node {$\mcSS = \delta_{ev_{\xv}}$} (Y);
	 \end{tikzpicture}
 \end{equation}
 \caption{The noise free sampling distributions $\mcSS$ given the prior $\delta_{\xv} \otimes P$ with the dirac measure on the $X$ component.}
 \label{fig:noiseGivenPt}
\end{figure}
\noindent
Using the property that $\delta_{ev_{\xv}}(B \mid  f) = \ch_{ev_{\xv}^{-1}(B)}(f)$ for all $B \in \sa_Y$ and $f \in Y^X$, the resulting deterministic sampling distributions (one for each $\xv \in X$) are given by 
\be \label{char}
\mcSS(B \mid  f)= \ch_{ev_{\xv}^{-1}(B)}(f). 
\ee
This special case of the prior $\delta_{\xv} \otimes P$, which is the most important one for many ML applications and the one implicitly assumed in ML textbooks, permits  a complete mathematical analysis.

Given the probability measure  $P \sim \GP(m,k)$ and $\mcSS = \delta_{ev_{\xv}}$, it follows 
the composite is the pushforward probability measure
\be
\mcSS \circ P = Pev_{\xv}^{-1},
\ee
which is the special case of Figure~\ref{fig:GPdef}   
with $X_0 = \{\xv \}$.   Using the fact that $P$ projected onto any coordinate is a normal distribution as shown in Figure~\ref{fig:inference0}, it follows that the expected mean is
\be  \nonumber
\begin{array}{lcl}
\E_{P ev_{\xv}^{-1}}(Id_Y) &=& \E_{P}(ev_{\xv}) \\
&=& m(\xv)
\end{array}
\ee
while the expected variance is
\be  \nonumber
\begin{array}{lcl}
\E_{P ev_{\xv}^{-1}}( Id_Y - \E_{P ev_{\xv}^{-1}}(Id_Y))^2) &=& \E_{P }( ev_{\xv} - \E_{P}(ev_{\xv}))^2) \\
&=& k(\xv,\xv).
\end{array}
\ee
These are precisely specified by the characterization $P ev_{\xv}^{-1} \sim \NN(m(\xv),k(\xv,\xv))$.

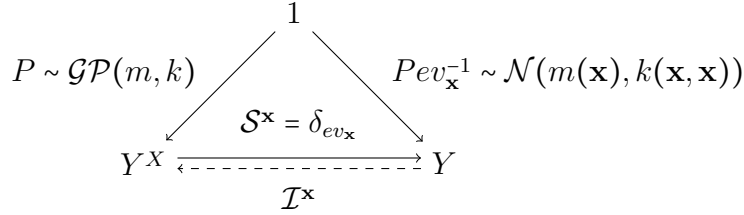
\begin{figure}[H]
\begin{equation}  \nonumber
 \begin{tikzpicture}[baseline=(current bounding box.center)]
         \node  (1)  at (0,0)    {$1$};
         \node (YX) at   (-2,-2)  {$Y^X$};
         \node (Y)  at    (2,-2)   {$Y$};
         
	\draw[->,above,left] (1) to node [xshift = -5pt,yshift=5pt] {$P \sim \GP(m,k)$} (YX);
	\draw[->, above] ([yshift=2pt] YX.east) to node [yshift=3pt] {$\mcSS =\delta_{ev_{\xv}}$} ([yshift=2pt] Y.west);
	\draw[->, below,dashed] ([yshift=-2pt] Y.west) to node [yshift=-3pt] {$\mcII$} ([yshift=-2pt] YX.east);
		
	\draw[->,above, right]  (1) to node [xshift = 5pt,yshift=5pt] {$P ev_{\xv}^{-1} \sim \NN(m(\xv),k(\xv,\xv))$} (Y);

	 \end{tikzpicture}
 \end{equation}
 \caption{The composite of the prior distribution $P \sim \GP(m,k)$  and the sampling distribution $\mcSS$ give the coordinate projections as priors on $Y$.}
 \label{fig:inference0}
 \end{figure}
 
Recall that the corresponding inference map $\mcII$ is any $\prob$ map satisfying the necessary and sufficient condition of Equation~\ref{eqn::product_rule}, {\em i.e.}, for all $\A \in \sa_{Y^X}$ and $B \in \sa_Y$,
\be \label{sufficientCondition}  
\int_{f \in \A} \mcSS( B \mid  f) \,   dP   = \int_{y \in B}\mcII(\A \mid  y) \, d(P ev_{\xv}^{-1}).  
\ee
Since the $\sigma$-algebra of $Y^X$ is generated by elements $ev_{\zv}^{-1}(A)$, for $\zv \in Y$ and $A \in \sa_Y$, we can take $\A = ev_{\zv}^{-1}(A)$ in the above expression to obtain the equivalent necessary and sufficient condition on $\mcII$ of
\be \nonumber
\int_{f \in ev_{\zv}^{-1}(A)} \mcSS(B \mid  f) \, dP = \int_{y \in B}\mcII( ev_{\zv}^{-1}(A) \mid  y) \, d(P ev_{\xv}^{-1}).       
\ee
From Equation~\ref{char}, $\mcSS(B \mid  f) = \ch_{ev_{\xv}^{-1}(B)}(f)$, so substituting this value into the  left hand side of this equation reduces that term to $P(ev_{\xv}^{-1}(B) \cap ev_{\zv}^{-1}(A))$.  Rearranging the order of the terms  it follows the condition on the inference map $\mcII$ is 
\be \nonumber
\int_{y \in B}\mcII(ev_{\zv}^{-1}(A) \mid  y) \, d(P ev_{\xv}^{-1}) = P(ev_{\xv}^{-1}(B) \cap ev_{\zv}^{-1}(A)).            
\ee
Since the left hand side of this expression is integrated with respect to the pushforward probability measure $P ev_{\xv}^{-1}$ it is equivalent to
\be \nonumber
\begin{array}{lcl}
\int_{y \in B}\mcII(ev_{\zv}^{-1}(A) \mid  y) \, d(P ev_{\xv}^{-1}) &=& \int_{f \in \, ev_{\xv}^{-1}(B)}\mcII(ev_{\zv}^{-1}(A) \mid  ev_{\xv}(f)) \, dP  \\
&=&   \int_{f \in \, ev_{\xv}^{-1}(B)}\mcII ev_{\zv}^{-1}(A \mid  ev_{\xv}(f)) \, dP.
\end{array}
\ee
In summary, if  $\mcII$ is to be an inference map for the prior $P$ and sampling distribution $\mcSS$, then it is necessary and sufficient that it satisfy the condition
\be \nonumber
  \int_{f \in \, ev_{\xv}^{-1}(B)}\mcII ev_{\zv}^{-1}(A \mid  ev_{\xv}(f)) \, dP  = P(ev_{\xv}^{-1}(B) \cap ev_{\zv}^{-1}(A)).  
\ee
 
Given a (deterministic) measurement\footnote{Meaning the arrow $d= \delta_{y}$ in Figure~\ref{fig:conditionalBM}. In general it is unnecessary to assume deterministic measurements in which case the composite $\mcII \circ d$ represents the posterior.} at $(\xv,y)$, the stochastic process \mbox{$\mcII(\cdot \mid  y):1 \rightarrow Y^X$} is the posterior of  $P \sim \GP(m,k)$.
This posterior, denoted $P^1_{Y^X} \triangleq \mcII(\cdot \mid  y)$, is generally not unique.  However we can require that the posterior  $P^1_{Y^X}$ be a GP specified by updated mean and covariance functions $m^1$ and $k^1$ respectively, which  depend upon the conditioning value $y$, so $P^1_{Y^X} \sim \GP(m^1,k^1)$.
To determine $P^1_{Y^X}$, and hence the desired inference map $\mcII$, we make a hypothesis about the updated mean and covariance functions $m^1$ and $k^1$ characterizing  $P^1_{Y^X}$ given a measurement at the pair $(\xv,y) \in X \times Y$.  Let us assume the updated mean function is of the form
\be  \label{updateMean}
m^1(\zv) = m(\zv) + \frac{k(\zv,\xv)}{k(\xv,\xv)} (y - m(\xv)) 
\ee
and the updated covariance function is of the form
\be \label{updateCov}
k^1(\wv,\zv) = k(\wv,\zv) - \frac{k(\wv,\xv) k(\xv,\zv)}{k(\xv,\xv)}. 
\ee

To prove these updated functions suffice to specify the inference map $\mcII(\cdot \mid  y) = P^1_{Y^X} \sim \GP(m^1,k^1)$ satisfying the necessary and sufficient condition we simply evaluate 
\be \nonumber
\int_{f \in \, ev_{\xv}^{-1}(B)}\mcII ev_{\zv}^{-1}(A \mid  ev_{\xv}(f)) \, dP
\ee
by substituting $\mcII(\cdot \mid  f(\xv)) = P^1(m^1,k^1)$ and verify that it yields $P(ev_{\xv}^{-1}(B) \cap ev_{\zv}^{-1}(A))$. Since $\mcII ev_{\zv}^{-1}(\cdot \mid  f(\xv))= P^1_{Y^X} ev_{\xv}^{-1}$ is a normal distribution of mean 
\be \nonumber
m^1(\zv)=  m(\xv) + \frac{k(\zv,\xv)}{k(\xv,\xv)} \cdot (f(\xv) - m(\xv))
\ee
 and  covariance 
 \be \nonumber
 k^1(\zv,\zv) = k(\zv,\zv)-\frac{k(\zv,\xv)^2}{k(\xv,\xv)}
 \ee
it follows that
\[
\int_{f \in \, ev_{\xv}^{-1}(B)}  \mcII ev_{\zv}^{-1}(A \mid  ev_{\xv}(f)) \, dP = \int_{f \in \, ev_{\xv}^{-1}(B)} \left( \frac{1}{\sqrt{2 \pi k^1(\zv,\zv)}}   \int_{v \in A} e^{\frac{-(m^1(\zv) - v)^2}{2 k^1(\zv,\zv)}} dv \right) dP
\]
which can be expanded to 
\[
 \int_{f \in \, ev_{\xv}^{-1}(B)} \left(  \frac{1}{\sqrt{2 \pi k^1(\zv,\zv)}}   \int_{v \in A} e^{\frac{-(m(\zv) + \frac{k(\zv,\xv)}{k(\xv,\xv)} (f(\xv) - m(\xv)) - v)^2}{2 k^1(\zv,\zv)}} dv \right) dP
 \]
 and equals
 \[
  \int_{y \in B} \left(  \frac{1}{\sqrt{2 \pi k^1(\zv,\zv)}}   \int_{v \in A} e^{\frac{-(m(\zv) + \frac{k(\zv,\xv)}{k(\xv,\xv)} (y - m(\xv)) - v)^2}{2 k^1(\zv,\zv)}} dv \right) dPev_{\xv}^{-1}.
\]
Using  $P_{Y^X} ev_{\xv}^{-1} \sim \NN(m(\xv),k(\xv,\xv))$ we can rewrite the expression as
\be \nonumber
\frac{1}{\sqrt{2 \pi} \mid \Omega\mid } \int_{y \in B} \int_{v \in A} e^{- \frac{1}{2} (\mathbf{u} - \overline{\mathbf{u}})^T \Omega^{-1} (\mathbf{u} - \overline{\mathbf{u}})} dv \, dy
\ee
where 
\be  \nonumber
\uv = \left( \begin{array}{c}  y \\ v \end{array} \right)  \quad \quad  \overline{\uv} =  \left( \begin{array}{c}  m(\xv) \\ m(\zv) \end{array} \right) 
\ee
and
\be  \nonumber
\Omega = \left( \begin{array}{cc} k[\xv,\xv] & k[\xv,\zv] \\ k[\zv,\xv] & k[\zv,\zv] \end{array} \right),
\ee
which we recognize as a normal distribution $\NN(\overline{\mathbf{u}}, \Omega)$.

On the other hand, we claim that

\begin{equation}  \nonumber
 \begin{tikzpicture}[baseline=(current bounding box.center)]
         \node  (1)  at (0,0)    {$1$};
         \node (YY) at   (8,0)  {$Y_{\xv} \times Y_{\zv}$,};
         
	\draw[->,above] (1) to node {$P(ev_{\xv}^{-1}(\cdot) \cap ev_{\zv}^{-1}(\cdot))$} (YY);

  \end{tikzpicture}
 \end{equation}
where $Y_{\xv}$ and $Y_{\zv}$ are two copies of $Y$, 
is also a normal distribution of mean  $\overline{u} = (m(\xv), m(\zv))$ with covariance matrix $\Omega$.\footnote{Formally the arguments should be numbered in the given probability measure as  $P(ev_{\xv}^{-1}(\#1) \cap ev_{\zv}^{-1}(\#2))$ because $ev_{\xv}^{-1}(A) \cap ev_{\zv}^{-1}(B) \ne ev_{\xv}^{-1}(B) \cap ev_{\zv}^{-1}(A)$. However the subscripts can be used to identify which component measurable sets are associated with each argument. }  To prove our claim consider the $\prob$ diagram in Figure~\ref{fig:coolProof}  where $X_0 = \{\xv,\zv\}$, $\iota:X_0 \hookrightarrow X$ is the inclusion map referenced in Section~\ref{sec:GP}, and $ev_{\xv} \times ev_{\zv}$ is an isomorphism between the two different representations of the set of all measurable functions $Y^{X_0}$ alluded to in the second paragraph of Section~\ref{sec:functions}. 

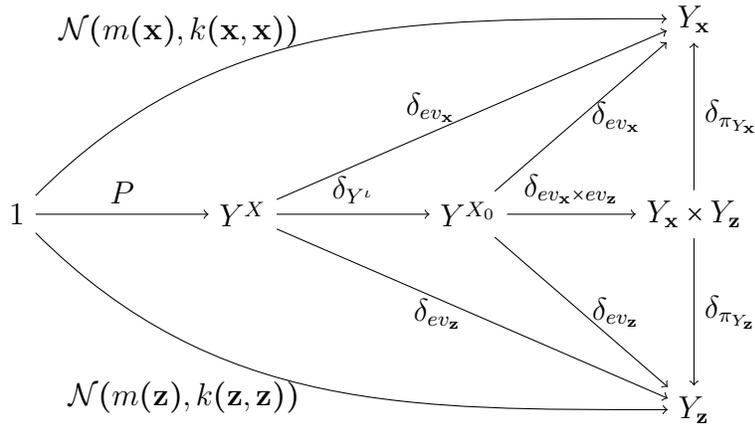
\begin{figure}[H]
\begin{equation}  \nonumber
 \begin{tikzpicture}[baseline=(current bounding box.center)]
         \node  (1)  at (0,0)    {$1$};
         \node (YX) at   (3,0)  {$Y^X$};
         \node (Y1) at  (9,2.6)   {$Y_{\xv}$};
           \node (Y2) at  (9,-2.6)   {$Y_{\zv}$};
           \node (YY) at (6,0)    {$Y^{X_0}$};
           \node (YxY) at (9,0)  {$Y_{\xv} \times Y_{\zv}$};
         
	\draw[->,above] (1) to node {$P$} (YX);
         \draw[->,above,above] (YX) to node {$\delta_{Y^{\iota}}$} (YY);            
	\draw[->,above, left] (YX) to node [xshift=-2pt, yshift=3pt] {$\delta_{ev_{\xv}}$} (Y1);
	\draw[->,below, left] (YX) to node {$\delta_{ev_{\zv}}$} (Y2);
         \draw[->, right] (YxY) to node {$\delta_{\pi_{Y_{\xv}}}$} (Y1);
          \draw[->, right] (YxY) to node {$\delta_{\pi_{Y_{\zv}}}$} (Y2);
          
          \draw[->, right] (YY) to node {$\delta_{ev_{\xv}}$} (Y1);
          \draw[->, right] (YY) to node {$\delta_{ev_{\zv}}$} (Y2);

          \draw[->,above] (YY) to node {$\delta_{ev_{\xv} \times ev_{\zv}}$} (YxY);
          
          \draw[->,out=45,in=180,left] (1) to node [xshift=-5pt,yshift=3pt] {$\NN(m(\xv),k(\xv,\xv))$} (Y1);
          \draw[->,out=-45,in=180,left] (1) to node [xshift=-5pt,yshift=-3pt]  {$\NN(m(\zv),k(\zv,\zv))$} (Y2);   
  \end{tikzpicture}
 \end{equation}
 \caption{Proving the joint distribution $\delta_{ev_{\xv} \times ev_{\zv}} \circ \delta_{Y^{\iota}}  \circ  P = P(ev_{\xv}^{-1}(\cdot) \cap ev_{\zv}^{-1}(\cdot)))$ is a normal distribution $\NN(\overline{\mathbf{u}}, \Omega)$.}
 \label{fig:coolProof}
 \end{figure}

 The diagram in Figure~\ref{fig:coolProof} commutes because
 $\delta_{\pi_{Y_{\xv}}} \circ \delta_{ev_{\xv} \times ev_{\zv}}= \delta_{\xv}$ and  $\delta_{\pi_{Y_{\zv}}} \circ \delta_{ev_{\xv} \times ev_{\zv}}= \delta_{\zv}$
 while, using $(ev_{\xv} \times ev_{\zv}) \circ Y^{\iota} = (ev_{\xv},ev_{\zv})$, 
 \be \nonumber
\begin{array}{lcl}
(\delta_{ev_{\xv} \times ev_{\zv}} \circ \delta_{Y^{\iota}} \circ P) (A \times B) &=& \int_{f \in Y^{X}}  \underbrace{\delta_{(ev_{\xv},ev_{\zv})}(A \times B \mid  f) }_{= (\ch_{A \times B})(ev_{\xv}, ev_{\zv})(f) } \, d P \\
&=& \int_{f \in Y^{X}} ( \ch_{ev_{\xv}^{-1}(A)} \cdot \ch_{ev_{\zv}^{-1}(B)})(f) \, d P \\
&=& \int_{f \in Y^{X}} \ch_{ev_{\xv}^{-1}(A) \cap ev_{\zv}^{-1}(B)}(f) \, d P  \\
&=&   P( ev_{\xv}^{-1}(A) \cap ev_{\zv}^{-1}(B)).
\end{array}
\ee

Moreover, the covariance $k$ of $P(ev_{\xv}^{-1}(\cdot) \cap ev_{\zv}^{-1}(\cdot)))$ is represented by the matrix $\Omega$ because by definition of $P$, in terms of $m$ and $k$, its restriction to $Y^{X_0}\cong Y_{\xv} \times Y_{\zv}$ has covariance $k\mid _{X_0} \cong \Omega$.

Consequently the necessary and sufficient condition for $\mcII = P^1_{Y^X} \sim \GP(m^1,k^1)$  to be an inference map is satisfied by the projection of $P_{Y^X}^1$ onto any single coordinate $\zv$ which corresponds to the restriction of $P_{Y^X}^1$ via the deterministic map $Y^{\iota}:Y^X \rightarrow Y^{X_0}$ with  $X_0=\{\zv\}$ as in Figure~\ref{fig:GPdef}.  But this procedure immediately extends to all finite subsets $X_0 \subset X$ using matrix algebra and consequently we conclude that the necessary and sufficient condition for $\mcII$ to be an inference map for the prior $P$ and the noise free sampling distribution $\mcSS$ is satisfied.

Writing the prior  GP as $P\sim \GP(m^0,k^0)$ the recursive updating equations are
\be \label{mRecursive}
m^{i+1}\left( \zv \mid  (\xv_i,y_i) \right) = m^{i}(\zv) + \frac{k^i(\zv,\xv_i)}{k^i(\xv_i,\xv_i)} (y_i - m^i(\xv_i)) \quad  \textrm{ for }  i=0,\ldots,N-1
\ee
and 
\be \label{kRecursive}
k^{i+1}((\wv,\zv) \mid  (\xv_i,y_i)) = k^i(\wv,\zv) - \frac{k^i(\wv,\xv_i) k^i(\xv_i,\zv)}{k^i(\xv_i,\xv_i)}  \quad \textrm{ for }  i=0,\ldots,N-1
\ee
where the terms on the left denote the posterior mean and covariance functions of $m^i$ and $k^i$ given a new measurement $(\xv_i,y_i)$.  These expressions  coincide with the standard formulas written for $N$ arbitrary measurements $\{(\xv_i,y_i)\}_{i=1}^{N-1}$, with $X_0=(\xv_0,\ldots,\xv_{N-1})$  a finite set of independent points of $X$ with corresponding measurements $\yv^T = (y_0,y_1,\ldots,y_{N-1})$,
\be \label{meanUpdate}
\tilde{m}(\zv \mid  X_0) = m(\zv) + K(\zv,X_0) K(X_0,X_0)^{-1}  (\yv - m(X_0))
\ee
where  $m(X_0) = (m(\xv_0),\ldots,m(\xv_{N-1}))^T$, and 
\be \label{coUpdate}
\tilde{k}( (\wv,\zv) \mid  X_0) = k(\wv,\zv) - K(\wv,X_0) K(X_0,X_0)^{-1} K(X_0,\zv)
\ee
where  $K(\wv,X_0)$ is the row vector with components $k(\wv,\xv_i)$, $K(X_0,X_0)$ is the matrix with components $k(\xv_i,\xv_j)$, and $K(X_0,\zv)$ is a column vector with components $k(\xv_i,\zv)$.\footnote{When the points are not independent then one can use a perturbation method or other procedure to avoid degeneracy.}  The notation $\tilde{m}$ and $\tilde{k}$ is used to differentiate these standard expressions from ours above.  Equations~\ref{meanUpdate} and \ref{coUpdate} are a computationally efficient way to keep track of the updated mean and covariance functions.  One can easily verify the recursive equations determine the standard equations using induction.  
 
 A review of the derivation of $P^1_{Y^X}$ indicates that the posterior $P^1_{Y^X} \sim \GP(m^1,k^1)$ is actually parameterized by the measurement $(\xv_1,y_1)$ because the above derivation holds for any measurement $(\xv_1,y_1)$ and this pair of values uniquely determines $m^1$ and $k^1$ through the Equations~\ref{mRecursive} and \ref{kRecursive}, or equivalently Equations ~\ref{meanUpdate} and \ref{coUpdate}, for a single measurement.

By the SMwCC structure of $\prob$ each parameterized GP $P_{Y^X}^1$ can be put  into the  bijective correspondence shown in Figure~\ref{fig:smwcc}, where
\be  \nonumber
\begin{array}{lcl}
\overline{P_{Y^X}^1}(B \mid  (z,(\xv,y))) &=& P_{Y^X}^1( ev_{\zv}^{-1}(B) \mid  (\xv,y)) \quad \quad \forall B \in \sa_Y, \zv \in X, y \in Y \\
&=& P_{Y^X}^1 ev_{\zv}^{-1} (B \mid  (\xv,y))  \\
&=& \frac{1}{\sqrt{ 2 \pi k^1(\zv,\zv)}} \int_{v \in B} e^{ -\frac{(v - m^1(z))^2}{2 k^1(\zv,\zv)}} \, dv \\
&=& \frac{1}{\sqrt{ 2 \pi \frac{k(\xv,\xv) k(\zv,\zv) - k(\xv,\zv)^2}{k(\xv,\xv)}}} \int_{v \in B} e^{ -\frac{(v - (m(z) + \frac{k(\zv,\xv)}{ k(\xv,\xv)}(y -\xv))^2}{2 \frac{k(\xv,\xv) k(\zv,\zv) - k(\xv,\zv)^2}{k(\xv,\xv)}}} \, dv
\end{array}
\ee   
which is a probability measure on $Y$ conditioned on $\zv$ and parameterized by the pair $(\xv,y)$.  Iterating this process we obtain the viewpoint that the parameterized process $P_{Y^X}( ev_{\zv}^{-1}(B) \mid  \{ (\xv_i,y_i)\}_{i=1}^N)$ is a posterior conditional probability parameterized over $N$ measurements.

  \begin{figure}[H]
\begin{equation}  \nonumber
 \begin{tikzpicture}[baseline=(current bounding box.center)]
         \node  (1)  at (-.2,0)    {$X \otimes Y$};
         \node  (YX)  at  (2.7,0)   {$Y^X$};
         \node (X) at   (-.2,-1)  {$X \otimes (X \otimes Y)$};
         \node (Y) at  (2.7,-1)  {$Y$};
         \node  (pt1)    at (-1,-.5)   {$$};
         \node  (pt2)    at (3.5,-.5)   {$$};
         \node  (pt3)    at (4.1,.25)   {$$};
         \node  (pt4)    at (4.1,-1.5)   {$$};
         
	\draw[->,above] (1) to node  {$P_{Y^X}^1$} (YX);
	\draw[->,below] (X) to node {$\overline{P_{Y^X}^1}$} (Y);
	
	\draw[-] (pt1) to node {$$} (pt2);
	
	\pgfsetlinewidth{.5ex}
	\draw[<->] (pt3) to node {$$} (pt4);

	 \end{tikzpicture}
 \end{equation}
 \caption{Each GP $P_{Y^X}^1$, which is parameterized by a measurement $(\xv,y) \in X \otimes Y$, determines a conditional $\overline{P_{Y^X}^1}$.}
 \label{fig:smwcc}
 \end{figure}

\subsection{The noisy measurement inference map}  When the measurement model has additive Gaussian noise which is iid on each slice $\xv \in X$, the resulting inference map is easily given by observing that from Equation~\ref{compositeMP},  the composite $\delta_{ev_{\xv}} \circ N \circ P \sim \NN(m(\xv),k(\xv,\xv) + k_N(\xv,\xv))$. Thus, the noisy sampling distribution along with the prior $P \sim \GP(m,k)$ can be viewed as a noise free distribution $P \sim \GP(m, \kappa)$ on $Y^X$, where $\kappa \triangleq k +k_N$, and $k_N$ is given by equation~\ref{noiseGP}.  This is clear from the composite of Figure~\ref{fig:GPComposite} with the Dirac measure $\delta_{\xv}$ on the $X$ component.   Now the noisy measurement inference map for the Bayesian model with prior $P$ and sampling distribution $\mcSS = \delta_{\xv} \circ N$, as shown in Figure~\ref{fig:noisyGA}, can be determined by decomposing it into two simpler Bayesian problems whose inference maps are (1)  trivial (the identity map)  and (2) already known.

\begin{figure}[H]
\begin{equation}  \nonumber
 \begin{tikzpicture}[baseline=(current bounding box.center)]

         \node (1) at (3,2) {$1$};
         \node (YX) at  (0,0)  {$Y^X$};
         \node (YX2) at  (3,0)  {$Y^X$};
         \node (Y)    at     (6,0)  {$Y$};
         
         \draw[->,above] (1) to node [xshift = -25pt] {$P \sim \NN(m,k)$} (YX);
         \draw[->,above] (YX) to node {$N$} (YX2);
         \draw[->,right,dashed] (1) to node {$N \circ P$} (YX2);
         \draw[->,right,dashed] (1) to node [xshift = 3pt] {$\delta_{\xv} \circ N \circ P \sim \NN(m(\xv),\underbrace{k(\xv,\xv) + k_N(\xv,\xv)}_{=\kappa(\xv,\xv)})$} (Y);
	\draw[->,above] (YX2) to node {$\delta_{ev_{\xv}}$} (Y);
	\draw[->,below,out=-45,in=225,looseness=.5] (YX) to node {$\mcSS$} (Y);
	
	\node  (com)  at  (6,-2)   {$\Downarrow$ Decomposition};
	
         \node (13) at (3,-3) {$1$};
         \node (YX3) at  (0,-5)  {$Y^X$};
         \node (YX23) at  (3,-5)  {$Y^X$};
         
         \draw[->,above] (13) to node [xshift = -25pt] {$P \sim \NN(m,k)$} (YX3);
         \draw[->,above] ([yshift=2pt] YX3.east) to node {$N$} ([yshift=2pt] YX23.west);
         \draw[->,right,dashed] (13) to node {$N \circ P$} (YX23);
         \draw[->,below,dashed] ([yshift=-2pt] YX23.west) to node {$\mcI_*$} ([yshift=-2pt] YX3.east);
	
         \node (14) at (7,-3) {$1$};
         \node (YX4) at  (7,-5)  {$Y^X$};
         \node (Y4) at  (10,-5)  {$Y$};
         
         \draw[->,above] (14) to node [xshift = -25pt] {$N \circ P$} (YX4);
         \draw[->,above] ([yshift=2pt] YX4.east) to node {$\delta_{\xv}$} ([yshift=2pt] Y4.west);
         \draw[->,right,dashed] (14) to node {$\delta_{\xv} \circ N \circ P$} (Y4);
         \draw[->,below,dashed] ([yshift=-2pt] Y4.west) to node {$\mcI_{nf}$} ([yshift=-2pt] YX4.east);
         
	 \end{tikzpicture}
 \end{equation}
 \caption{Splitting the Gaussian additive noise Bayesian model (top diagram)  into two separate Bayesian models (bottom two diagrams) and composing the inference maps for these two simple Bayesian models gives the inference map for the original Gaussian additive Bayesian model.}
 \label{fig:noisyGA}
\end{figure}
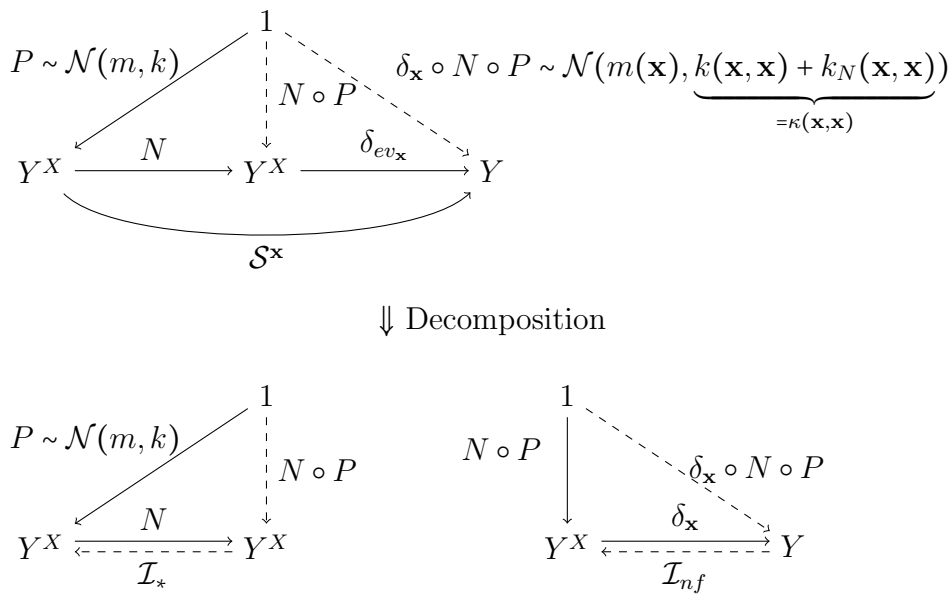
\noindent
Observe that the composition of the two bottom diagrams is the top diagram.  The bottom diagram on the right is a noise free Bayesian model with GP prior $N \circ P$ and sampling distribution $\delta_{\xv}$ whose inference map $\mcI_{nf}$ we have already determined analytically in Section~\ref{sec:noiseFree}.  Given a measurement $y \in Y$ at $\xv \in X$, the inference map is given by the updating Equations~\ref{updateMean} and \ref{updateCov} for the mean and covariance functions characterizing the GP on $Y^X$.  The resulting posterior GP on $Y^X$ can then be viewed as a measurement on $Y^X$ for the bottom left diagram, which is a Bayesian model with prior $P$ and sampling distribution $N$.  The inference map $\mcI_{\star}$ for this diagram is the identity map on $Y^X$, $\mcI_{\star} = \delta_{Id_{Y^X}}$.  This is easy to verify using Bayes product rule (Equation~\ref{productRule}), $\int_{a \in A} N(B\mid a) \, dP = \int_{f \in B} \delta_{Id_{Y^X}}(A \mid  f) \, d(N \circ P)$, for any $A, B \in \sa_{Y^X}$.  Composition of these two inference maps, $\mcI_{nf}$ and $\mcI_{\star}$ then yields the resulting inference map for the Gaussian additive noise Bayesian model.

With this observation both of the recursive updating schemes given by Equations~\ref{mRecursive} and \ref{kRecursive} are valid for the Gaussian additive noise model with $k$ replaced by $\kappa$.  The corresponding standard expressions for the noisy model are then 
\be \nonumber
\tilde{m}(\zv \mid  X_0) = m(\zv) + K(\zv,X_0) K(X_0,X_0)^{-1} (\yv - m(X_0))
\ee
and 
\be  \nonumber
\tilde{\kappa}( (\wv,\zv) \mid  X_0) = \kappa(\wv,\zv) - K(\wv,X_0) K(X_0,X_0)^{-1} K(X_0,\zv),
\ee
where the quantities like $K(\wv,X_0)$ are as defined previously (following Equation~\ref{coUpdate}) except now $k$ is replaced by $\kappa$.  For $\wv \ne \zv$ and neither among the measurements $X_0$ these expressions, upon substituting in for $\kappa$, reduce to the  familiar expressions
\be \nonumber
\tilde{m}(\zv \mid  X_0) = m(\zv) + K(\zv,X_0) (K(X_0,X_0)+ \sigma^2 \textit{I})^{-1} (\yv - m(X_0))
\ee
and 
\be \nonumber
\tilde{k}( (\wv,\zv) \mid  X_0) = k(\wv,\zv) - K(\wv,X_0) (K(X_0,X_0)+ \sigma^2 \textit{I})^{-1} K(X_0,\zv),
\ee
which provide for a computationally efficient way to compute the mean and covariance of a GP given a finite number of measurements.

\subsection{The inference map for  parametric models}   \label{sec:parInf}
Under the prior $\delta_{\xv} \otimes P$ on the hypothesis space in the parametric model, Figure~\ref{fig:parametricModel},  the parametric sampling distribution model can be viewed as  a family of models, one for each $\xv \in X$, given by the diagram in Figure~\ref{fig:noiseFreeParametric}.

\begin{figure}[H]
\begin{equation}  \nonumber
 \begin{tikzpicture}[baseline=(current bounding box.center)]

         \node (Rp) at (-3,0) {$\Rp$};
         \node (YX) at  (0,0)  {$Y^X$};
         \node (YX2) at  (3,0)  {$Y^X$};
         \node (Y)    at     (6,0)  {$Y$};
         
         \draw[->,above] (Rp) to node {$\delta_{i}$} (YX);
         \draw[->,above] (YX) to node {$N$} (YX2);
	\draw[->,above] (YX2) to node {$\delta_{ev_{\xv}}$} (Y);
	\draw[->,below,out=-45,in=225,looseness=.5] (Rp) to node {$\mcSS_p$} (Y);
	 \end{tikzpicture}
 \end{equation}
 \caption{The Gaussian additive noise parametric sampling distributions $\mcSS_p$ viewed as a family of sampling distributions, one for each $\xv \in X$.}
 \label{fig:noiseFreeParametric}
\end{figure}
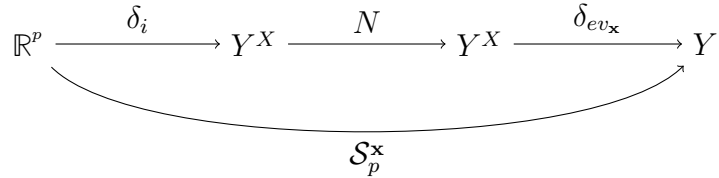

The sampling distribution can be computed as
\be  \nonumber
\begin{array}{lcl}
\mcSS_p(B \mid  \av) &=& (\delta_{ev_{\xv}} \circ N \circ \delta_{i})(B \mid  \av) \\
&=& \int_{f \in Y^X} (\delta_{ev_{\xv}} \circ N)(B \mid  f) \, d\underbrace{(\delta_i)_{\av}}_{\delta_{F_{\av}}} \\
&=& (\delta_{ev_{\xv}} \circ N)(B \mid  F_{\av}) \\
&=& N(ev_{\xv}^{-1}(B) \mid  F_{\av}) .
\end{array}
\ee
Because  $N_{F_{\av}} \sim GP(F_{\av},k_N)$, it follows that 
\be  \nonumber
N(ev_{\xv}^{-1}(\bullet) \mid  F_{\av}) = N_{F_{\av}}ev_{\xv}^{-1} \sim \NN(F_{\av}(\xv), \sigma^2)
\ee
and consequently 
\be  \nonumber
\mcSS_p(B \mid  \av) = \frac{1}{\sqrt{2 \pi} \sigma} \int_{B} e^{-\frac{(y-F_{\av}(\xv))^2}{2 \sigma^2}} \, dy.
\ee

Taking the prior $P:1 \rightarrow \Rp$ as  a normal distribution with  mean $\mv$ and covariance function  $k$, it follows that the composite $\mcSS_p \circ P \sim \NN(F_{\mv}(\xv),k(\xv,\xv)+\sigma^2)$
while the inference map $\mcII_p$ satisfies, for all $B \in \sa_Y$ and all $\A \in \sa_{\mathbb{R}^p}$, 
\be   \nonumber
\int_{\mathbf{a} \in \A} \mcSS_p( B \mid  \mathbf{a}) \,   dP   = \int_{y \in B}\mcII_p(\A \mid  y) \, d(\mcSS_p \circ P). 
\ee

To determine this inference map $\mcII_p$ it is necessary to require the parametric map
\be \nonumber
\begin{array}{lclcl}
i&:&\Rp  & \longrightarrow & Y^X  \\
&:& \av & \mapsto &  i_{\av}
\end{array}
\ee
be an injective linear homomorphism.  Under this condition, which can often be achieved simply by eliminating redundant modeling parameters, we can explicitly determine the inference map for the parameterized model, denoted $\mcII_p$, by decomposing it  into two  inference maps as displayed in the diagram in Figure~\ref{fig:inferenceMap}.

\begin{figure}[H]
\begin{equation}  \nonumber
 \begin{tikzpicture}[baseline=(current bounding box.center)]

         \node (1)    at  (0,2)   {$1$};
         \node (Rp) at (-3,0) {$\Rn$};
         \node (YX) at  (0,0)  {$Y^X$};
         \node (Y)  at (3,0)  {$Y$};
         \node (com) at (6,0)  {$\mcSS_{p} = \mcSS_n \circ \delta_{i}$};
         
         \draw[->,right] (1) to node {$Pi^{-1}$} (YX);
         \draw[->,left,above] (1) to node {$P$} (Rp);
         \draw[->,above] ([yshift=2pt] Rp.east) to node {$\delta_{i}$} ([yshift=2pt] YX.west);
         \draw[->,below] ([yshift=-2pt] YX.west) to node {$\mcI_{\star}$} ([yshift=-2pt] Rp.east);

         \draw[->,above] ([yshift=2pt] YX.east) to node {$\mcSS_n$} ([yshift=2pt] Y.west);
         \draw[->,below] ([yshift=-2pt] Y.west) to node {$\mcII_n$} ([yshift=-2pt] YX.east);     
         \draw[->,above,right] (1) to node [yshift=4pt] {$\mcSS_n \circ Pi^{-1}$} (Y);    
         
         \draw[->,out=225,in=-45,looseness=.5,below] (Y) to node {$\mathcal{I}^{\xv}_{p}$} (Rp);

	 \end{tikzpicture}
 \end{equation}
 \caption{The inference map for the parametric model is a composite of two inference maps.}
 \label{fig:inferenceMap}
\end{figure}

We first show the stochastic process $Pi^{-1}$ is a GP and by taking the sampling distribution $\mcSS_n = \delta_{ev_{\xv}} \circ N$ as the noisy measurement model we can use the result of the previous section to provide us with the inference map $\mcII_n$ in Figure~\ref{fig:inferenceMap}.   

\begin{lemma} Let $\kv$ be the matrix representation of the covariance function $k$.  The
the push forward of $P ~ \NN(\mathbf{m},k)$ by $i$ is a GP $Pi^{-1} \sim \GP(i_{\mv}, \hat{k})$, where $\hat{k}(\uv,\vv) = \uv^T \kv \vv$.
\end{lemma}
\begin{proof} We need to show that the push forward of $Pi^{-1}$ by the restriction map \mbox{$Y^{\iota}: Y^X \longrightarrow Y^{X_0}$} is a normal distribution for any finite subspace \mbox{$\iota: X_0 \hookrightarrow X$}.  Consider the commutate  diagram in Figure~\ref{fig:lemma}, where $Y_{\xv}$ is a copy of $Y$, $X_0 = (\xv_1,\ldots,\xv_{n'})$,  and   
\[
ev_{\xv_1} \times \ldots \times ev_{\xv_{n'}}:Y^{X_0} \rightarrow  \prod_{\xv \in X_0} Y_{\xv}
\]
is the canonical isomorphism. 

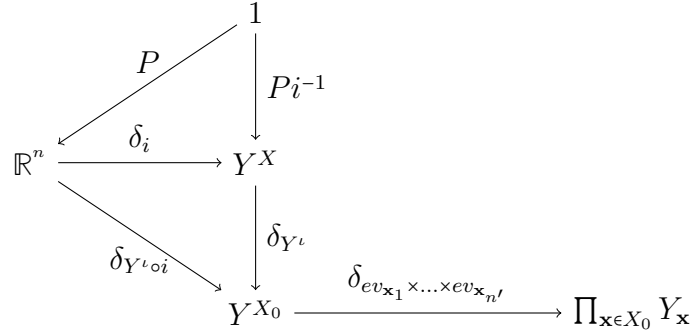
\begin{figure}[H]
\begin{equation}  \nonumber
 \begin{tikzpicture}[baseline=(current bounding box.center)]

         \node (1)    at  (0,2)   {$1$};
         \node (Rp) at (-3,0) {$\Rn$};
         \node (YX) at  (0,0)  {$Y^X$};
         \node (YX0)    at     (0,-2)  {$Y^{X_0}$};
         \node (Yprod)  at  (5,-2)  {$\prod_{\xv \in X_0} Y_{\xv}$};
         
         \draw[->,right] (1) to node {$Pi^{-1}$} (YX);
         \draw[->,left,above] (1) to node {$P$} (Rp);
         \draw[->,above] (Rp) to node {$\delta_{i}$} (YX);
         \draw[->,right] (YX) to node {$\delta_{Y^{\iota}}$} (YX0);
	\draw[->,left, below] (Rp) to node {$\delta_{Y^{\iota} \circ i}$} (YX0);
         \draw[->,above] (YX0) to node {$\delta_{ev_{\xv_1} \times \ldots \times ev_{\xv_{n'}}}$} (Yprod);
	 \end{tikzpicture}
 \end{equation}
 \caption{The restriction of  $Pi^{-1}$.}
 \label{fig:lemma}
\end{figure}

 The composite of the measurable maps
\be
\left( (ev_{\xv_1} \times \ldots \times ev_{\xv_{n'}}) \circ Y^{\iota} \circ i \right)(\av) = (i_{\av}(\xv_1),\ldots,i_{\av}(\xv_{n'}))
\ee
from which  it follows that  the composite map $\delta_{ev_{\xv_1} \times \ldots \times ev_{\xv_{n'}}} \circ \delta_{Y^{\iota}} \circ Pi^{-1} \sim \NN(  X_0^T \mv, X_0^T \mathbf{k} X_0)$.         
\end{proof}

Now the diagram

\begin{figure}[H]
\begin{equation}  \nonumber
 \begin{tikzpicture}[baseline=(current bounding box.center)]

         \node (1)    at  (-1.5,2)   {$1$};
         \node (Rn) at (-3,0) {$\Rn$};
         \node (YX) at  (0,0)  {$Y^X$};
         
         \draw[->,right] (1) to node {$Pi^{-1}$} (YX);
         \draw[->,left] (1) to node {$P$} (Rn);
         \draw[->,above] ([yshift=2pt] Rn.east) to node {$\delta_{i}$} ([yshift=2pt] YX.west);
         \draw[->, below,dashed] ([yshift=-2pt] YX.west) to node {$\mcI_{\star}$} ([yshift=-2pt] Rn.east);         

	 \end{tikzpicture}
 \end{equation}
\end{figure}
\noindent
with the sampling distribution for this Bayesian problem  as $\delta_{i}$.   Let $\e_j^T = (0, \ldots, 0,1,0,\ldots,0)$ be the $j^{th}$ unit vector of $\Rp$ and let $i_{\e_j} = f_j \in Y^X$.  The elements $\{f_j\}_{j=1}^p$ form the components of a basis for the image of $i$ by the assumed  injective property of $i$.  Let this finite basis have a dual basis $\{f_j^*\}_{j=1}^p$ so that $f_k^*(f_j) = \delta_{k}(j)$.

 Consider the measurable map
\be \nonumber
\begin{array}{lclcl}
f_1^* \times \ldots \times f_p^* &:& Y^X & \rightarrow & \Rp \\
&;&g& \mapsto & ( f^*_1(g),\ldots, f_p^*(g)),
\end{array}
\ee
Using the linearity of the parameter space $\Rp$ it follows $\av = \sum_{i=1}^p a_i \e_i$ and consequently
\be \nonumber
\begin{array}{lcl}
\left((f_1^* \times \ldots \times f_p^*) \circ i \right)(\av) &=& ( f_1^*(i_{\av}),\ldots, f_p^*(i_{\av}) ) \\
&=& \av \quad \textrm{ using }f_j^*(i_{\av}) = f_j^*( \sum_{k=1}^p a_k f_k) =a_j
\end{array}
\ee
and hence $(f_1^* \times \ldots \times f_p^*) \circ i = id_{\Rp}$ in $\M$.
Now it follows the corresponding inference map $\mcI_{\star} = \delta_{f_1^* \times \ldots \times f_p^*}$  
because the necessary and sufficient condition for $\mcI_{\star}$ is given, for all $ev_{\zv}^{-1}(B) \in \sa_{Y^X}$ (which generate $\sa_{Y^X}$) and all $\A \in \sa_{\Rn}$,  by
\be \label{split}
\int_{\av \in \A} \delta_i(ev_{\zv}^{-1}(B)\mid  \av) \, dP = \int_{g\in ev_{\zv}^{-1}(B)} \mcI_{\star}(\A \mid  g) \, dPi^{-1}
\ee
with the left hand term reducing to the expression
\be  \nonumber
\int_{ \mathbf{a} \in \A} \ch_{i^{-1}(ev_{\zv}^{-1}(B))}(\mathbf{a}) \, dP = P(i^{-1}(ev_{\zv}^{-1}(B)) \cap \A).
\ee
On the other hand, using $\mcI_{\star} = \delta_{f_1^* \times \ldots \times f_p^*}$, the right hand term of Equation~\ref{split} also reduces to the same expression since
\be \nonumber
\begin{array}{lcl}
\int_{g\in ev_{\zv}^{-1}(B)} \delta_{f_1^* \times \ldots \times f_p^*}(\A \mid  g) \, d(Pi^{-1}) &=& \int_{\av \in i^{-1}(ev_{\zv}^{-1}(B))} \ch_{((f_1^* \times \ldots \times f_p^*) \circ i)^{-1}(\A)}(\av) \, dP \\
&=&  \int_{\av \in i^{-1}(ev_{\zv}^{-1}(B))} \ch_{\A}(\av) \, dP \\
&=& P(i^{-1}(ev_{\zv}^{-1}(B)) \cap \A)
\end{array}
\ee
thus proving $\mcI_{\star} = \delta_{f_1^* \times \ldots \times f_p^*}$.

Taking 
\be \nonumber
\mcII_p = \mcI_{\star} \circ \mcII_n,
\ee
it follows that for a given measurement $(\xv,y)$ that the composite is  
\be
\mcII_p = \mcII_n( (f_1^* \times \ldots \times f_p^*)^{-1}(\cdot) \mid  y)
\ee
which is the push forward measure of the GP $\mcII_n( \cdot \mid  y) \sim \GP(i_{\mv}^1, \kappa^1)$ where  (as defined previously) $\kappa = k + k_N$ and
\be
i_{\mv}^1(\zv) = i_{\mv}(\zv) + \frac{\kappa(\zv,\xv)}{\kappa(\xv,\xv)}(y - i_{\mv}(\xv)) 
\ee
and
\be
\kappa^1(\uv,\vv) = \kappa(\uv,\vv) - \frac{ \kappa(\uv,\xv) \kappa(\xv, \vv)}{ \kappa(\xv,\xv)}.
\ee
This GP projected onto any finite subspace $\iota: X_0 \hookrightarrow X$ is a normal distribution and,  for $X_0 = \{\xv_1,\xv_2,\ldots,\xv_n\}$, it follows that
\[
\begin{tikzpicture}[baseline=(current bounding box.center)]

         \node (1)          at    (-5.7,0) {$1$};
         \node (YX)       at   (0,0)  {$Y^X$};
         \node (YX0)    at   (2.8,0)   {$Y^{X_0}$};
         \node (Rp)      at   (2.8,-3)  {$\prod_{i=1}^n Y_i \cong \Rn$};

         \draw[->,above] (YX) to node {$\delta_{Y^{\iota}}$} (YX0);
         \draw[->,right] (YX0) to node {$\delta_{ev_{\xv_1} \times \ldots \times ev_{\xv_n}}\mid _{Y^{X_0}}$} (Rp);
         \draw[->,left] (YX) to node {$\delta_{ev_{\xv_1} \times \ldots \times ev_{\xv_n}}$} (Rp);
         \draw[->,above] (1) to node {$\mcII_n(\bullet \mid  y) \sim \GP(i_{\mv}^1,\kappa^1)$} (YX);
         \draw[->,below,left] (1) to node [xshift=-3pt,yshift=-9pt] {$\mcI_p(\bullet \mid  y)  \sim \NN( (i_{\mv}^1(\xv_1),\ldots,i_{\mv}^1(\xv_n))^T, \kappa^1\mid _{X_0})$} (Rp);
         
	 \end{tikzpicture}
\]
where $Y_i$ is a copy of $Y =\mathbb{R}$ and the restriction $\delta_{ev_{\xv_1} \times \ldots \times ev_{\xv_n}}\mid _{Y^{X_0}}$ is an isomorphism.
The inference map $\mcI_p(\bullet \mid  y)$ is the updated normal distribution on $\Rn$ given the measurement $(\xv,y)$ which can be rewritten as 
\be \nonumber
\mcI_p(\bullet \mid  y)  \sim \NN(\mv + K(X_0,\xv) \kappa(\xv,\xv)^{-1} (y - \mv^T \xv),  \kappa^1\mid _{X_0}),
\ee
where $X_0$ is now viewed as the ordered set $X_0 = (\xv_1,\ldots,\xv_n)$ and $K(X_0,\xv)$ is the $n$-vector with components $\kappa(\xv_j,\xv)$.

 Iterating this updating procedure for 
$N$ measurements $\{(\xv_i,y_i)\}_{i=1}^{N}$ the $N^{th}$ posterior  coincides with the analogous noisy measurment inference updating Equations~\ref{meanUpdate} and \ref{coUpdate} with $\kappa$ in place of $k$.

\section{Stochastic Processes as Points}   

Having defined stochastic processes  we would be remiss not to mention the Markov process---one of the most familiar type of processes used for modeling.   Many applications can be approximated by Markov models and a familiar example is the Kalman filter which we describe below as it is the archetype.  While Kalman filtering is not commonly viewed as a ML problem, it is useful to put it into perspective with respect to the Bayesian modeling paradigm.

By looking at Markov processes we are immediately led to a generalization of the definition of a stochastic process which is due to Lawvere and Meng~\cite{Meng}.    To motivate this we start with the elementary idea first before giving the generalized definition of a stochastic process.

\subsection{Markov processes via Functor Categories}  Here we assume knowledge of the definition of a functor, and refer the unfamiliar reader to any standard text on category theory.  Let $T$ be any set with a total (linear) ordering $\le$  so for every $t_1,t_2 \in T$ either $t_1 \le t_2$ or $t_2 \le t_1$. (Here we have switched from our standard  ``$X$'' notation to ``$T$'' as we wish to convey the image of a space with properties similar to time as modeled by the real line.) We can view $(T,\le)$ as  a category with the objects as the elements and the set of arrows from one object to another as
\be \nonumber
hom_T(t_1,t_2) = \left\{ \begin{array}{ll} \star \textrm{ iff }t_1 \le t_2 \\ \emptyset \textrm{ otherwise } \end{array} \right.
\ee

The functor category $\prob^T$ has as objects functors $\F: (T,\le) \rightarrow \prob$ which  play an important role in the theory of stochastic processes, and we formally give the following definition. 

\begin{definition} A \emph{Markov transformation} is a functor $\F: (T,\le) \rightarrow \prob$.  
\end{definition} 
From the modeling perspective we look at the image of the functor $\F \in_{ob} \prob^T$ in the category $\prob$ so given any sequence of ordered points $\{t_i\}_{i=1}^{\infty}$ in $T$ their image under $\F$ is shown in Figure~\ref{fig:MMFunctor}, where $ \F_{t_i,t_{i+1}} = \F( \le )$ is a $\prob$ arrow.

 \begin{figure}[H]
\begin{center}
 \begin{tikzpicture}[baseline=(current bounding box.center)]
	\node	(x1)	at	(0,0)	                  {$\F(t_1)$};
	\node         (x2)   at    (3,0)           {$\F(t_2)$};
	\node          (x3)  at    (6,0)           {$\F(t_3)$};
	\node         (dots)   at    (9,0)           {$\ldots$};
	
	\draw[->, above] (x1) to node  {$\F_{t_1,t_2}$} (x2);
         \draw[->, above] (x2) to node {$\F_{t_2,t_3}$} (x3);
         \draw[->, above] (x3) to node {$\F_{t_3,t_4}$} (dots);     
          
 \end{tikzpicture}
\end{center}
\caption{A Markov transformation as the image of a $\prob$ valued  Functor.}
\label{fig:MMFunctor}
\end{figure}

 By functoriality, these arrows satisfy the conditions
\begin{enumerate}
\item $\F_{t_i,t_i} = id_{t_i}$, and
\item $\F_{t_i,t_{i+2}} = \F_{t_{i+1},t_{i+2}} \circ \F_{t_i, t_{i+1}}$
\end{enumerate}
Using the definition of composition in $\prob$ the second condition can be rewritten as 
\be \nonumber
\F_{t_i,t_{i+2}}(B \mid  x) = \int_{u \in \F(t_{i+1})} \F_{t_{i+1},t_{i+2}}(B \mid  u)  \,  d\F_{t_i, t_{i+1}}(\cdot \mid  x)
\ee
for $x \in \F(t_i)$ (the ``state'' of the process at time $t_i$) and $B \in \sa_{F(t_{i+2})}$.  This equation is  called the Chapman-Kolomogorov relation and can be used, in the non categorical characterization, to define a Markov process.  

The important aspect to note about this definition of a Markov model is that the measurable spaces $\F(t_i)$ can be distinct from the other measurable spaces $\F(t_j)$, for $j \ne i$, and of course the arrows $\F_{t_i,t_{i+1}}$ are in general distinct.    This simple  definition of a Markov transformation as a functor captures the property of an evolving process being ``memoryless'' since if we know where the process $\F$ is at $t_i$, say $x \in \F(t_i)$, then its expectation at $t_{i+1}$ (as well as higher order moments) can be determined without regard to its ``state'' prior to $t_i$.

The arrows of the functor category $\prob^T$ are natural transformations $\eta: \F \rightarrow \G$, for $\F, \G \in_{ob} \prob^T$, and hence satisfy the commutativity relation given in Figure~\ref{fig3} for every $t_1,t_2 \in T$ with $t_1 \le t_2$.

 \begin{figure}[H]
\begin{center}
 \begin{tikzpicture}[baseline=(current bounding box.center)]
	\node	(x1)	at	(0,0)	                  {$\F(t_1)$};
	\node         (x2)   at    (0,-3)           {$\F(t_2)$};

	\node	(y1)	at	(3,0)	            {$\G(t_1)$};
	\node         (y2)   at    (3,-3)           {$\G(t_2)$};
	
	\draw[->, left] (x1) to node  {$\F_{t_1,t_2}$} (x2);
         \draw[->, above] (x1) to node {$\eta_{t_1}$} (y1);
         \draw[->, above] (x2) to node {$\eta_{t_2}$} (y2); 
         	\draw[->, right] (y1) to node  {$\G_{t_1,t_2}$} (y2);
              
 \end{tikzpicture}
\end{center}
\caption{An arrow in $\prob^T$ is a natural transformation.}
\label{fig3}
\end{figure}

The functor category $\prob^T$ has a terminal object $\mathbf{1}$ mapping $t \mapsto 1$ for every $t \in T$ and this object $\mathbf{1} \in_{ob} \prob^T$ allows us to generalize the definition of a stochastic process.\footnote{The elementary definition of a stochastic process, Definition~\ref{stochasticDef}, as a probability measure on a function space suffices for what we might call standard ML. For more general constructions, such as Markov Models and Hierarchical Hidden Markov Models (HHMM) the generalized definition is required. } 

\begin{definition}  \label{generalizedStoch}  Let $X$ be \emph{any category}.  A \emph{stochastic process} is a point in the category $\prob^X$,  {\em i.e.}, a $\prob^X$ arrow $\eta: \mathbf{1} \rightarrow \F$ for some $\F \in_{ob} \prob^X$.\footnote{In \emph{any category}  with a terminal object $1$ an arrow whose domain is $1$ is called a point.  So an arrow $x :1 \rightarrow X$ is called  a point of $X$ whereas $f: X  \rightarrow Y$ is sometimes referred to as a  generalized element to emphasize that it ``varies''  over the domain.  It is constructive to consider what this means in the category of Sets and why the terminology is meaningful.}
\end{definition}

Different categories $X$ correspond to different types of stochastic processes.  Taking the simplest possible case let $X$ be a set considered as a discrete category---the objects are the elements $x \in X$ while there are no nonidentity arrows in $X$ viewed as a category.  This case generalizes Definition~\ref{stochasticDef}  because, for $Y$ a fixed measurable space we have the functor $\hat{Y}: X \rightarrow \prob$ mapping each object $x \in_{ob} X$ to a copy $Y_x$ of $Y$ and this special case corresponds to Definition~\ref{stochasticDef}.

Taking  $X= T$, where $T$ is a totally ordered set (and subsequently viewed as a category with one arrow between any two elements), 
and looking at the image of 
 \begin{figure}[H]
\begin{center}
 \begin{tikzpicture}[baseline=(current bounding box.center)]

	\node	(x1)	at	(0,0)	                  {$t_1$};
	\node         (x2)   at    (3,0)           {$t_2$};
	\node          (x3)  at    (6,0)           {$t_3$};
	\node         (dots)   at    (9,0)           {$\ldots$};
	
	\draw[->, above] (x1) to node  {$\le$} (x2);
         \draw[->, above] (x2) to node {$\le$} (x3);
         \draw[->, above] (x3) to node {$\le$} (dots);     
                   
 \end{tikzpicture}
\end{center}
\end{figure}
\noindent
under the stochastic process $\mu: \mathbf{1} \rightarrow \F$ gives the commutative diagram in Figure~\ref{fig4}.
 \begin{figure}[H]
\begin{center}
 \begin{tikzpicture}[baseline=(current bounding box.center)]
         \node         (1)     at      (4.5, 3)        {$1$};
	\node	(x1)	at	(0,0)	                  {$\F(t_1)$};
	\node         (x2)   at    (3,0)           {$\F(t_2)$};
	\node          (x3)  at    (6,0)           {$\F(t_3)$};
	\node         (dots)   at    (9,0)           {$\ldots$};
	
	\draw[->, above] (x1) to node  {$\F_{t_1,t_2}$} (x2);
         \draw[->, above] (x2) to node {$\F_{t_2,t_3}$} (x3);
         \draw[->, above] (x3) to node {$\F_{t_3,t_4}$} (dots);     
         
         \draw[->,left] (1) to node {$\mu_{t_1}$} (x1);
         \draw[->,left] (1) to node {$\mu_{t_2}$} (x2);         
         \draw[->,left] (1) to node {$\mu_{t_3}$} (x3);
         \draw[->,left] (1) to node {$\mu_{t_4}$} (dots);
                   
 \end{tikzpicture}
\end{center}
\caption{A Markov model as the image of a stochastic process.}
\label{fig4}
\end{figure}
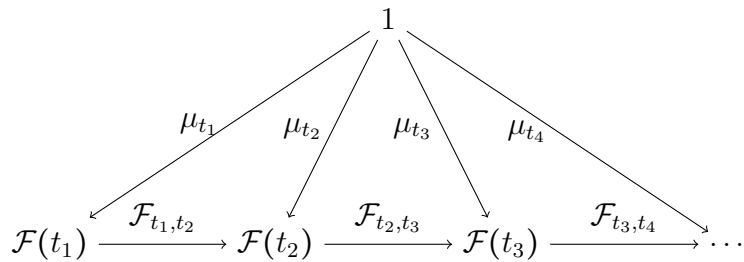
\noindent
From this perspective a stochastic process $\mu$ can be viewed as a family of probability measures on the measurable spaces $\F(t_i)$, and the stochastic process $\mu$ 
 coupled with a $\prob^T$ arrow $\eta: \F \rightarrow \G$ maps one Markov model to another
 \begin{figure}[H]
\begin{center}
 \begin{tikzpicture}[baseline=(current bounding box.center)]
         \node         (1)     at      (4.5, 3)        {$1$};
	\node	(x1)	at	(0,0)	                  {$\F(t_1)$};
	\node         (x2)   at    (3,0)           {$\F(t_2)$};
	\node          (x3)  at    (6,0)           {$\F(t_3)$};
	\node         (dots)   at    (9,0)           {$\ldots$};

	\node	(y1)	at	(0,-3)	                  {$\G(t_1)$};
	\node         (y2)   at    (3,-3)           {$\G(t_2)$};
	\node          (y3)  at    (6,-3)           {$\G(t_3)$};
	\node         (dots2)   at    (9,-3)           {$\ldots$};
	
	\draw[->, above] (x1) to node  {$\F_{t_1,t_2}$} (x2);
         \draw[->, above] (x2) to node {$\F_{t_2,t_3}$} (x3);
         \draw[->, above] (x3) to node {$\F_{t_3,t_4}$} (dots);     
         
         	\draw[->, above] (y1) to node  {$\G_{t_1,t_2}$} (y2);
         \draw[->, above] (y2) to node {$\G_{t_2,t_3}$} (y3);
         \draw[->, above] (y3) to node {$\G_{t_3,t_4}$} (dots2);     
         
         \draw[->,left] (1) to node {$\mu_{t_1}$} (x1);
         \draw[->,left] (1) to node {$\mu_{t_2}$} (x2);         
         \draw[->,left] (1) to node {$\mu_{t_3}$} (x3);
         \draw[->,left] (1) to node {$\mu_{t_4}$} (dots);
                   
         \draw[->, left] (x1) to node  {$\eta_{t_1,t_2}$} (y1);
         \draw[->, left] (x2) to node {$\eta_{t_2,t_3}$} (y2);
         \draw[->, left] (x3) to node {$\eta_{t_3,t_4}$} (y3);     

 \end{tikzpicture}
\end{center}
\end{figure}

One can also observe that GPs can be defined using this generalized definition of a stochastic process.  For $X$ a measurable space it follows for any finite subset $X_0 \subset X$ we have the inclusion map $\iota: X_0 \hookrightarrow X$ which is a measurable function, using the subspace $\sigma$-algbra for $X_0$, and we are led back to Diagram~\ref{fig:GPdef} with the stochastic process $P: \textbf{1} \rightarrow \hat{Y}$, where $\hat{Y}$ is as defined in the paragraph above following Definition~\ref{generalizedStoch}, which satisfies the appropriate restriction property defining a GP.

These simple examples illustrate that different stochastic processes can be obtained by either varying the structure of the category $X$ and/or by placing additional requirements on the projection maps, e.g., requiring the projections be normal distributions on finite subspaces of the exponent category $X$.

\subsection{Hidden Markov Models}  
To bring in the Bayesian aspect of Markov models it is necessary to consider the measurement process associated with a sequence as in Figure~\ref{fig4}.  In particular, consider the standard diagram
 \begin{figure}[H]
\begin{center}
 \begin{tikzpicture}[baseline=(current bounding box.center)]
 	\node	(1)	at	(0,0)		         {$1$};
	\node         (X)   at    (-1.5,-2)           {$\F(t_1)$};
	\node         (Y)   at   (1.5, -2)              {$Y_{t_1}$};
         
	\draw[->, left] (1) to node  {$\mu_{t_1}$} (X);
         \draw[->, above] (X) to node {$\mcS_{t_1}$} (Y);
         \draw[->, dashed,right] (1) to node {$d_{t_1}$} (Y);
 \end{tikzpicture}
\end{center}
\end{figure}
\noindent
which characterizes a Bayesian model, where $Y_{t_1}$ is a copy of a $Y$ which is a data measurement space, $\mcS_{t_1}$ is interpreted as a measurement model and  $d_{t_1}$ is an actual data measurement on the ``state'' space $\F(t_1)$.  This determines an inference map $\mcI_{t_1}$ so that given a measurement $d_{t_1}$ the posterior probability on $\F(t_1)$ is $\mcI_{t_1} \circ d_{t_1}$.  Putting the two measurement models together with the Markov transformation model $\F$  we obtain the following diagram in Figure~\ref{fig:HMM}.
 \begin{figure}[H]
\begin{center}
 \begin{tikzpicture}
  	\node	(1)	at	(0,2)		         {$1$};
	\node         (X1)   at    (-2,0)           {$\F(t_1)$};
	\node         (X2)   at    (2,0)           {$\F(t_2)$};
	
	\node         (Y1)   at    (-2,-3)           {$Y_{t_1}$};
	\node         (Y2)   at    (2,-3)           {$Y_{t_2}$};
	
	\draw[->, left,above] (1) to node [xshift=-2pt]  {$\mu_{t_1}$} (X1);
         \draw[->, above] (X1) to node {$\F_{t_1,t_2}$} (X2);

         \draw[->, left] ([xshift=-2pt] X1.south) to node [yshift=8pt] {$\mcS_{t_1}$} ([xshift=-2pt] Y1.north);
         \draw[->, right] ([xshift=2pt] Y1.north) to node [yshift=8pt] {$\mcI_{t_1}$} ([xshift=2pt] X1.south);
         \draw[->,right,dashed] (1) to node [yshift = -9pt] {$d_{t_1}$} (Y1);
         \draw[->,above,dashed,out = -180,in=90,looseness=1] (1) to node [xshift=-25pt] {$\hat{\mu}_{t_1} =  \mcI \circ d_{t_1}$} (X1);
         \draw[->, right] (1) to node {$\F_{t_1,t_2} \circ \hat{\mu}_{t_1}$} (X2);
         
         \draw[->, left] ([xshift=-2pt] X2.south) to node  [yshift=8pt] {$\mcS_{t_2}$} ([xshift=-2pt] Y2.north);
         \draw[->, right] ([xshift=2pt] Y2.north) to node  [yshift=8pt] {$\mcI_{t_2}$} ([xshift=2pt] X2.south);
         
 \end{tikzpicture}
\end{center}
\caption{The hidden Markov model viewed in $\prob$.}
\label{fig:HMM}
\end{figure}

This is the hidden Markov process in which given a prior probability $\mu_{t_1}$ on the space $\F_{t_1}$ we can use the measurement $d_{t_1}$ to update the prior to the posterior $\hat{\mu}_{t_1} = \mcI_{t_1} \circ d_{t_1}$ on $\F(t_1)$. The posterior then composes with $\F_{t_1,t_2}$ to give the prior $\F_{t_1,t_2} \circ \hat{\mu}_{t_1}$ on $\F(t_2)$, and now the process can be repeated indefinitely.  The Kalman filter is an example  in which the Markov map $\F_{t_1,t_2}$ describe the linear dynamics of some system under consideration (as in tracking a satellite), while the sampling distributions $\mcS_{t_1}$ model the noisy measurement process which for the Kalman filter is Gaussian additive noise. Of course one can easily replace the linear dynamic by a nonlinear dynamic and the Gaussian additive noise model by any other measurement model, obtaining an extended Kalman filter, and the above form of the diagram does not change at all, only the $\prob$ maps change.

\section{Final Remarks}  
In closing, we would like to make a few comments on the use of category theory for ML, where the largest potential payoff lies in exploiting the abstract framework that categorical language provides.   This section assumes a basic familiarity with monads and should be viewed as only providing conceptual directions for future research which we believe are relevant for the mathematical development of learning systems.  Further details on the theory of monads can be found in most category theory books, while the basics as they relate to our discussion below can be found in our previous paper~\cite{Culbertson}, in which we provide the simplest possible example of a decision rule on a discrete space.

Seemingly all aspects of ML including Dirichlet distributions and unsupervised learning (clustering)  can  be characterized using the category $\prob$.   As an elementary example, mixture models can be developed by consideration of the space of all (perfect) probability measures $\mathscr{P}X$  on a measurable space $X$ endowed with the coarsest $\sigma$-algebra such that the evaluation maps $ev_{B}: \mathscr PX \to [0,1]$ given by $ev_{B}(P) = P(B)$, for all $B \in \sa_X$,  are measurable.  This actually defines the object mapping of  a functor $\mathscr P : \prob \to \M$ which sends a measurable space $X$ to the space $\mathscr PX$ of probability measures on $X$. On arrows, $\mathscr P$ sends the $\prob$-arrow $f: X \to Y$ to the measurable function $\mathscr Pf: \mathscr PX \to \mathscr PY$ defined pointwise on $\sa_{Y}$ by
\be \nonumber
	\mathscr Pf(P)(B) = \int_{X} f_{B} \, dP.
\ee
This functor is called the Giry monad, denoted $\G$, and the Kleisli category $K(\G)$ of the Giry monad is equivalent to $\prob$.\footnote{See Giry\cite{Giry} for the basic definitions and equivalence of these categories.}  The reason we have chosen to present the material from the perspective of $\prob$ rather that $K(\G)$ is that the existing literature on ML uses Markov kernels rather than the equivalent arrows in $K(\G)$.
The Giry monad  determines the nondeterministic $\prob$ mapping
\begin{equation}     \nonumber
 \begin{tikzpicture}[baseline=(current bounding box.center)]

         \node  (X)  at  (4,0)   {$X$};
         \node (X2) at   (2,0)  {$\px X$};

	\draw[->,below] (X2) to node {$\varepsilon_X$} (X);
	 \end{tikzpicture}
 \end{equation}
\noindent
given by
 $\varepsilon_X(P,B) = ev_B(P) = P(B)$ for all $P \in \px(X)$ and all $B \in \sa_X$.   Using this construction, any probability measure $P$ on $\px X$ then yields a mixture of probability measures on $X$ through the composite map 
\begin{equation}     \nonumber
 \begin{tikzpicture}[baseline=(current bounding box.center)]
         \node  (1)  at  (0,0)   {$1$};
         \node  (X)  at  (2,-2)   {$X$};
         \node (X2) at   (-2,-2)  {$\px X$};
         \draw[->,left] (1) to node {$P$} (X2);
	\draw[->,below] (X2) to node {$\varepsilon_X$} (X);
	\draw[->,dashed,right] (1) to node {$\varepsilon_X \circ P=$ \,  A mixture model.} (X);
	 \end{tikzpicture}
 \end{equation}

We have briefly introduced the Kleisli category $K(\G) \, (\cong \prob)$ because it is a subcategory $\mathcal D$ of the Eilenberg--Moore category of $\G$-algebras, which we call the  category of decision rules,\footnote{Doberkat~\cite{Doberkat} has analyzed the Kleisli category under the condition that the arrows are not only measurable but also continuous.  This is an unnecessary assumption, resulting in all finite spaces having no decision rules,  though his considerable work on this category $K(\G)$ provides much useful insight as well as applications of  this category. } because the objects of this category are $\M$ arrows $r: \px X \rightarrow X$ sending a probability measure $P$ on $X$ to an actual element of $X$ satisfying some basic properties including $r(\delta_x) = x$.  Thus $r$ acts as a decision rule converting a probability measure on $X$ to an actual element of $X$ and, if $P$ is deterministic, takes that measure to the point $x \in X$ of nonzero measure.\footnote{Measurable spaces are defined only up to isomorphism, so that if two elements $x, y \in X$ are nondistinguishable in terms of the $\sigma$-algebra, meaning there exist no measurable set $A \in \sa_X$ such that $x \in A$ and $y \not \in A$, then $\delta_x = \delta_y$ and we also identify $x$ with $y$.}  Decision theory is generally presented from the perspective of taking probability measures on $X$ and, usually via a family of loss functions $\theta:X \rightarrow \mathbb{R}$,  making a selection among a family of possible choices $\theta \in \Theta$ where $\Theta$ is some measurable space rather than $X$. However,  it can clearly be viewed from this more basic viewpoint. 

The largest potential payoff in using category theory for ML and related applications appears to  be in integrating decision theory with probability theory, expressed in terms of the category $\mathcal{D}$, which would provide a basis for an automated reasoning system.  While the Bayesian framework presented in this paper can fruitfully be exploited to construct estimation of unknown functions it still lacks the ability to \emph{make decisions} of any kind.  Even if we were to invoke a list of simple rules to make decisions the category $\prob$ is too restrictive to implement these rules.  By working in the larger category of decision rules $\mathcal{D}$, it is possible to implement both the Bayesian reasoning presented in this work as well as decision rules as part of larger  reasoning system.  Our perspective on this problem is that Bayesian reasoning in general is inadequate---not only because it lacks the ability to make decisions---but because it is a \emph{passive} system which ``waits around'' for additional measurement data.  An automated reasoning system must take self directed action as in commanding itself to ``swivel the camera $45$ degrees right to obtain necessary additional information'', which is a (decision) command and control component which can be integrated with Bayesian reasoning.  An intelligent system would in addition, based upon the work of Rosen~\cite{Rosen},  in which he employed categorical ideas,  possess an anticipatory component.  While he did not use the language of SMCC it is clear this aspect was his intention and critical  in his method of modeling intelligent systems, and within the category $\mathcal{D}$ this additional aspect can also be modeled.

\section{Appendix A: Integrals over probability measures. }

The following three properties are the only three properties used throughout the paper to derive the values of integrals defined over probability measures.  
\begin{enumerate}
\item The integral of any measurable function $f:X \rightarrow \mathbb{R}$ with respect to a dirac measure satisfies
\be \nonumber
\int_{u \in X} f(u) \, d\delta_{x} = f(x).
\ee
This is straightforward to show using standard measure theoretic arguments.

\item Integration with respect to a push forward measure can be pulled back.  Suppose $f: X \rightarrow Y$ is any measurable function, $P$ is a probability measure on $X$, and $\phi:Y \rightarrow \mathbb{R}$ is any measurable function.  Then
\be \nonumber
 \int_{y \in Y} \phi(y) \, d(Pf^{-1}) = \int_{x \in X} \phi(f(x))  \, dP 
\ee
To prove this simply show that it holds for $\phi = \ch_B$, the characteristic function at $B$, then extend it to any simple function, and finally use the monotone convergence theorem to show it holds for any measurable function.

\item Suppose $f: X \rightarrow Y$ is any measurable function and $P$ is a probability measure on $X$. Then 
\be \nonumber
\int_{x \in X} \delta_{f}(B \mid x) \, dP = \int_{x \in X} \ch_{B}(f(x)) \, dP = P(f^{-1}(B))
\ee
This is a special case of case (2) with $\phi=\ch_B$.
\end{enumerate}

\section{Appendix B:  The weak closed structure in $\prob$}

Here is a simple illustration of the weak closed property of $\prob$ using finite spaces.
Let $X=2 = \{0,1\}$ and $Y=\{a,b,c\}$, both with the powerset $\sigma$-algebra.  This yields the powerset $\sigma$-algebra on $Y^X$ and each function can be represented by an ordered pair, such as $(b,c)$ denoting the function $f(1)=b$ and $f(2)=c$.   Define two probability measures $P,Q$ on $Y^X$ by
\be   \nonumber
\begin{array}{lcccl}
P( \{(b,c) \}) &=& .5 &=& P(\{(c,b)\}) \\
Q( \{(b,b)\}) &=& .5 &=& Q(\{(c,c)\})
\end{array}
\ee
and both measures having a value of $0$ on all other singleton measurable sets.  Both of these probability measures on $Y^X$ yield the same conditional probability measure

\begin{figure}[H]
\begin{equation} \nonumber
 \begin{tikzpicture}[baseline=(current bounding box.center)]

         \node (X) at  (0,0)  {$(X, \sa_X)$};
         \node (Y)  at  (4,0)   {$(Y, \sa_Y)$};

	\draw[->,above] (X) to node {$\overline{P}=\overline{Q}$} (Y);

 \end{tikzpicture}
 \end{equation}
\end{figure}
\noindent
since 
\be  \nonumber
\begin{array}{lcccl}
\overline{P}(\{a\} | 1) &=& 0 &=& \overline{Q}(\{a\} | 1) \\
\overline{P}(\{b\} | 1) &=& .5 &=& \overline{Q}(\{b\} | 1) \\
\overline{P}(\{c\} | 1) &=& .5 &=& \overline{Q}(\{c\} | 1)
\end{array}
\ee
and
\be   \nonumber
\begin{array}{lcccl}
\overline{P}(\{a\} | 2) &=& 0 &=& \overline{Q}(\{a\} | 2) \\
\overline{P}(\{b\} | 2) &=& .5 &=& \overline{Q}(\{b\} | 2) \\
\overline{P}(\{c\} | 2) &=& .5 &=& \overline{Q}(\{c\} | 2)
\end{array}
\ee
Since $P \ne Q$ the uniqueness condition required for the closedness property fails and only the existence condition  is satisfied.

\let\Section\section 
\def\section*#1{\Section{#1}} 

 \bibliographystyle{plain}

\vspace{.3in}
\begin{tabbing}
Jared Culbertson \hspace{2.7in}     \= Kirk Sturtz \\
RYAT, Sensors Directorate \> Universal Mathematics \\
Air Force Research Laboratory, WPAFB \> Vandalia, OH 45377 \\
Dayton, OH 45433 \> \email{kirksturtz@UniversalMath.com} \\
\email{jared.culbertson@us.af.mil}    \\
\end{tabbing}

\end{document}